\documentclass[reqno]{amsart} 
\usepackage{amsmath,amssymb,amsthm}
\usepackage{mathabx}
\usepackage[pdftex]{graphicx}
\usepackage{a4wide}
\usepackage{algorithm,algorithmic}
\usepackage{color}
\usepackage{soul}
\usepackage{array}

\newtheorem{theorem}{Theorem}[section]
\newtheorem{lemma}[theorem]{Lemma}
\theoremstyle{definition}
\newtheorem{example}{Example}[section]

\theoremstyle{remark} 
\newtheorem{remark}{Remark}[section]

\providecommand{\apriori}{\emph{a priori} }
\providecommand{\aposteriori}{\emph{a posteriori} }

\newcommand{\R}{\mathbb{R}}
\newcommand{\N}{\mathbb{N}}

\providecommand{\FCal}{\mathcal{F}}
\providecommand{\hatF}{\widehat{\mathcal{F}}}

\providecommand{\abs}[1]{{\left| #1 \right|}}

\providecommand{\parenthesis}[1]{\left( #1 \right)}
\providecommand{\Prob}[1]{\mathrm{P}\mspace{-2mu}\left( #1 \right)}

\providecommand{\Exp}[1]{{\ensuremath{\mathrm{E}}\mspace{-2mu}\left[#1\right]}}

\providecommand{\Var}[1]{{\ensuremath{\mathrm{Var}\!\left( #1 \!\right)}}}
                
\providecommand{\bigO}[1]{{\mathcal{O}}\!\left(#1\right)}

\providecommand{\littleO}[1]{{ o}\!\left(#1\right)}

\providecommand{\T}{\mathrm{T}}                

\providecommand{\Ito}{It\^{o}~}

\providecommand{\tol}{{\ensuremath{\mathrm{TOL}}}}

\providecommand{\tolT}[1][]{{\ensuremath{\mathrm{TOL}_{\mathrm{T}}^{\mspace{2mu}#1}}}}

\providecommand{\tolS}[1][]{{\ensuremath{\mathrm{TOL}_{\mathrm{S}}^{\mspace{2mu}#1}}}}

\providecommand{\Dt}[1]{{\ensuremath{\Delta t^{\{#1\}}}}}

\providecommand{\DtMax}{{\ensuremath{\Delta t_{\mathrm{max}}}}}

\providecommand{\WL}[1]{\ensuremath{W^{\{#1\}}}}

\providecommand{\nRefine}{{\ensuremath{N_{\mathrm{refine}}}}}

\providecommand{\barX}{\overline{X}}

\providecommand{\barXT}[1][T]{\overline{X}_{#1} }
\providecommand{\barXL}[1][\ell]{\overline{X}_T^{ \{#1\}}}

\providecommand{\g}[1]{\ensuremath{g\mspace{-3mu} \left(\mspace{-1mu} #1 \mspace{-2mu}\right)}}

\providecommand{\gXT}{\ensuremath{\g{X_T}}}

\providecommand{\gBarXT}{\ensuremath{\g{\barXT}}}
\providecommand{\DlGO}[1][\ell]{\ensuremath{\Delta_{ #1 }\g{ \omega_{i,\ell} } }}
\providecommand{\DlG}[1][\ell]{{\ensuremath{\Delta_{ #1 } g }}}

\providecommand{\rhoBar}{\overline{\rho}}                

\providecommand{\rBar}{\overline{r}}                

\providecommand{\phiX}[1]{\ensuremath{\varphi_{\mspace{-3mu}x}\mspace{-3mu}\left( #1\right)}}
\providecommand{\phiXi}[2]{\ensuremath{\varphi_{\mspace{-2mu}#1}\mspace{-3mu}\left( #2\right)}}
                
\providecommand{\phiXX}[1]{\ensuremath{\varphi_{\mspace{-3mu}xx}\mspace{-3mu}\left( #1\right)}}
\providecommand{\phiXXX}[1]{\ensuremath{\varphi_{\mspace{-3mu}xxx}\mspace{-3mu}\left( #1\right)}}                
\providecommand{\phiXXXX}[1]{\ensuremath{\varphi_{\mspace{-3mu}xxxx}\mspace{-3mu}\left( #1\right)}} 
\providecommand{\phiBarX}[1]{\ensuremath{\overline{\varphi}_{\mspace{-2mu}x, #1 }}}                
\providecommand{\phiBarXi}[2]{\ensuremath{\overline{\varphi}_{\mspace{-2mu}#1, #2 }}}                

\providecommand{\yy}[2]{{Y^{(#1)}_{#2}}}
\providecommand{\yyi}[3]{{Y^{(#1)}_{#3,#2}}}
\providecommand{\ydi}[3]{{\Delta \yyi{#1}{#2}{#3}}}

\providecommand{\neqq}{{\neq}}

\providecommand{\Cc}{{\ensuremath{C_\text{C}}}}

\providecommand{\checkDA}{\widecheck{\Delta a} }
\providecommand{\tildeDA}{\widetilde{\Delta a} }
\providecommand{\checkDa}{\widecheck{\Delta a} }
\providecommand{\tildeDa}{\widetilde{\Delta a} }
\providecommand{\checkDb}{\widecheck{\Delta b} }
\providecommand{\tildeDb}{\widetilde{\Delta b} }
\providecommand{\checkDB}{\widecheck{\Delta b} }
\providecommand{\tildeDB}{\widetilde{\Delta b} }
\providecommand{\sumDGL}[1][\ell]{\ensuremath{\mathcal{S}_{\DGL}}}

\providecommand{\roundUp}[1]{\ensuremath{\left\lceil #1 \right\rceil}}
\providecommand{\A}[1]{\mathcal{A}\left(#1\right)}
\providecommand{\V}[1]{\mathcal{V}\left(#1\right)}

\providecommand{\AMLSimple}{\mathcal{A}_{{}_\mathcal{M L}}}
\providecommand{\CML}{\mathcal{C}_{{}_\mathcal{M L}}}

\begin{document}

\title[Construction of an MSE Adaptive Euler--Maruyama Method Applied
to MLMC]{Construction of a Mean Square Error Adaptive Euler--Maruyama
  Method With Applications in Multilevel Monte Carlo}

\author[H.~Hoel]{H{\aa}kon Hoel$^{\dagger,\ddagger}$}
\thanks{$^\dagger$ Computer, Electrical and Mathematical Sciences and Engineering (CEMSE) Division
King Abdullah University of Science and Technology, Thuwal 23955-6900, Kingdom
    of Saudi Arabia ({\tt juho.happloa@kaust.edu.sa, raul.tempone@kaust.edu.sa})}

\thanks{$^\ddagger$ Department of Mathematics, University of Oslo, P.O. Box1053, 
Blindern, NO--0316 Oslo, Norway (Corresponding author: {\tt haakonah@math.uio.no})}

\author[J.~H{\"a}pp{\"o}l{\"a}]{Juho H{\"a}pp{\"o}l{\"a}$^\dagger$}

\author[R.~Tempone]{Ra\'ul Tempone$^\dagger$}



\begin{abstract}
  A formal mean square error expansion (MSE) is derived for
  Euler--Maruyama numerical solutions of stochastic differential
  equations (SDE).  The error expansion is used to construct a
  pathwise, \emph{a posteriori}, adaptive time-stepping Euler--Maruyama
  algorithm for numerical solutions of SDE, and the resulting algorithm is
  incorporated into a multilevel Monte Carlo (MLMC) algorithm for weak
  approximations of SDE. This gives an efficient MSE adaptive MLMC
  algorithm for handling a number of low-regularity approximation
  problems. In low-regularity numerical example problems, the
  developed adaptive MLMC algorithm is shown to outperform the uniform
  time-stepping MLMC algorithm by orders of magnitude, producing output
  whose error with high probability is bounded by $\tol>0$ at the
  near-optimal MLMC cost rate $\bigO{\tol^{-2} \log(\tol)^4}$.
\end{abstract}

\maketitle

\setcounter{tocdepth}{2}
\tableofcontents

\section{Introduction}

SDE models are frequently applied in
mathematical finance~\cite{Shreve04, Platen06, Glasserman04}, where
an observable may, for example, represent the payoff of an option.
SDE are also used to model the dynamics of multiscale physical, chemical or
biochemical systems~\cite{Gillespie00,Milstein03,Skeel02, Talay02},
where, for instance, concentrations, temperature and energy
may be sought observables.

Given a filtered, complete probability space $(\Omega, \FCal, (\FCal_t)_{0\leq t \leq T} , \text{P})$, we
consider the \Ito SDE
\begin{equation}\label{eq:sdeProblem}
\begin{split}
dX_t &= a(t,X_t) dt + b(t,X_t) dW_t, \qquad t \in (0,T],\\
X_0 &= x_0,
\end{split}
\end{equation}
where $X:[0,T] \times \Omega \to \R^{d_1}$ is a stochastic process with
randomness generated by a $K$-dimensional Wiener process, $W:[0,T]
\times \Omega \to \R^{d_2}$, with independent components, $W = (W^{(1)},
W^{(2)}, \ldots, W^{(d_2)})$, and $a: [0,T] \times \R^{d_1} \to\R^{d_1}$ and $b:
[0,T] \times \R^{d_1} \to \R^{d_1 \times d_2}$ are the drift and diffusion
coefficients, respectively. The considered filtration $\FCal_t$ 
is generated from the history of the Wiener process $W$ up to 
time $t$ and the possible outcomes of the initial data $X_0$, and
succeedingly completed with all $\mathrm{P}$-outer measure zero
sets of the sample space $\Omega$. That is
\[
\FCal_t := \overline{\sigma( \{W_{s}\}_{0\le s \le t}) \lor \sigma(X_0) }
\]
where the operation $\mathcal{A}\lor \mathcal{B}$ denotes the $\sigma$-algebra
generated by the pair of $\sigma$-algebras $\mathcal{A}$ and $\mathcal{B}$, i.e.,
$\mathcal{A}\lor \mathcal{B} := \sigma(\mathcal{A}, \mathcal{B})$,
and $\overline{\mathcal{A}}$ denotes the P-outer measure 
null-set completion of $\mathcal{A}$,
\begin{equation*}\label{eq:nullSetCompletion}
\overline{\mathcal{A}} := \mathcal{A} \lor \left\{ A \subset \Omega\, \Big|\,  \inf_{\hat A
  \in \{ \check A \in \mathcal{A}\, | \, \check A \supset A \}} \Prob{\hat A} = 0\right\}.
\end{equation*}

The contributions of this work are twofold. First, an \aposteriori
adaptive time-stepping algorithm for computing numerical realizations
of SDE using the Euler--Maruyama method is developed. And second, for
a given observable $g:\R^{d_1} \to \R$, we construct a mean square error
(MSE) adaptive time-stepping multilevel Monte Carlo (MLMC) algorithm
for approximating the expected value, $\Exp{g(X_T)}$, under the
following constraint:
\begin{equation}\label{eq:mlmcGoal}
\Prob{\abs{ \Exp{ \gXT } - \mathcal{A}} \leq \tol} \geq 1-\delta.
\end{equation}
Here, $\mathcal{A}$ denotes the algorithm's approximation of
$\Exp{g(X_T)}$ (examples of which are given in Equations~\eqref{eq:mcApprox} 
and~\eqref{eq:MLestimator}) and $\tol$ and $\delta >0$ are
accuracy and confidence constraints, respectively. 

The rest of this paper is organized as follows: First, in
Section~\ref{sec:mcem}, we review the Monte Carlo methods and their
use with the Euler--Maruyama integrator.  This is followed by
discussion of Multilevel Monte Carlo methods and adaptivity for SDEs.
The theory, framework and numerical examples for the MSE adaptive
algorithm is presented in Section~\ref{sec:mseAdaptivity}.  In
Section~\ref{sec:adaptiveMLMC}, we develop the framework for the MSE
adaptive MLMC algorithm and present implementational details in
algorithms with pseudocode. In Section~\ref{sec:MlmcExamples}, we
compare the performance of the MSE adaptive and uniform MLMC
algorithms in a couple of numerical examples, one of which is a
low-regularity SDE problem.  Finally, we present brief conclusions
followed by technical proofs and the extension of the main result to
higher-dimensional problems in the appendices.

\subsection{Monte Carlo Methods and the Euler--Maruyama Scheme.}\label{sec:mcem}

Monte Carlo (MC) methods provide a robust and typically non-intrusive way
to compute weak approximations of SDE. The convergence rate of MC
methods does not depend on the dimension of the problem; 
for that reason, MC is particularly effective on multi-dimensional problems. In
its simplest form, an approximation by the MC method consists of the following two
steps:
\begin{enumerate}

\item[(A.1)] Compute $M$ independent and identically distributed numerical realizations, $\barXT(\omega_{m})$, of the
  SDE~\eqref{eq:sdeProblem}.

\item[(A.2)] Approximate $\Exp{\g{X_T} }$ by the sample average
\begin{equation*}\label{eq:mcApprox}
\mathcal{A} :=  \sum_{m=1}^M \frac{\g{\barXT(\omega_{m})}}{M}.
\end{equation*}
\end{enumerate}
Regarding ordinary differential equations (ODE), the theory for 
numerical integrators of different orders for scalar SDE is vast. Provided
sufficient regularity, higher order integrators generally yield higher
convergence rates ~\cite{Kloeden92}. With MC methods it is straightforward to determine
that the goal~\eqref{eq:mlmcGoal} is fulfilled at the computational
cost $\bigO{\tol^{-2-1/\alpha}}$, where $\alpha\geq 0$ denotes the 
weak convergence rate of the numerical method, as defined 
in Equation~\eqref{eq:weakRate}.

As a method of temporal discretization, the Euler--Maruyama scheme is given by
\begin{equation}\label{eq:eulerMaruyama}
\begin{split}
\barX_{t_{n+1}} & =  \barX_{t_n} + a(t_n, \barX_{t_n}) \Delta t_n + b(t_n, \barX_{t_n}) \Delta W_n,\\
\barX_0     & = x_0,
\end{split}
\end{equation}
using time steps $\Delta t_n = t_{n+1}-t_n$ and Wiener increments 
$\Delta W_{n} = W_{t_{n+1}} - W_{t_n} \sim N(0, \Delta t_n I_{d_2} )$,
where $I_{d_2}$ denotes the $d_2\times d_2$ identity matrix.
In this work, we will focus exclusively on Euler--Maruyama time-stepping. 
The Euler--Maruyama scheme, which may be considered the SDE-equivalent of 
the forward-Euler method for ODE, has, under sufficient regularity,
first-order weak convergence rate
\begin{equation}\label{eq:weakRate}
\abs{ \Exp{ \g{X_T} - \g{\barXT} }  }  = \bigO{\max_n \Delta t_n},
\end{equation}
and also first-order MSE convergence rate
\begin{equation}\label{eq:strongRate}
\Exp{  \parenthesis{ \g{X_T} - \g{\barXT} }^2  }  = \bigO{  \max_n \Delta t_n},
\end{equation}
cf.~\cite{Kloeden92}. For multi-dimensional SDE problems, higher order
schemes are generally less applicable, as either the diffusion
coefficient matrix has to fulfill a rigid commutativity condition, or
Levy areas, required in higher order numerical schemes, have to be
accurately approximated to achieve better convergence rates
than those obtained with the Euler--Maruyama method
\cite{Kloeden92}.
  
\subsection{Uniform and Adaptive Time-Stepping MLMC}

MLMC is a class of MC methods that uses a hierarchy of subtly
correlated and increasingly refined realization ensembles to reduce
the variance of the sample estimator. In comparison with single-level
MC, MLMC may yield orders of magnitude reductions in the computational
cost of moment approximations. MLMC was first introduced
by Heinrich~\cite{Heinrich98,HeinrichSin99} for approximating integrals
that depend on random parameters. For applications in SDE problems,
Kebaier~\cite{Kebaier05} introduced a two-level MC method and demonstrated
its potential efficiency gains over single-level
MC. Giles~\cite{Giles08} thereafter developed an MLMC algorithm for SDE,
exhibiting even higher potential efficiency gains. Presently, MLMC is a
vibrant and growing research topic, (cf.~\cite{Giles15, Giles14, Barth12, Cliffe11, 
Teckentrup13, Mishra12,AbdulLateef14}, and references therein).

\subsubsection{MLMC Notation} \label{subsec:mlmcNotation}

We define the multilevel estimator by
\begin{equation} \label{eq:MLestimator}
\begin{split}
\AMLSimple := \sum_{\ell=0}^L \sum_{i=1}^{M_\ell} 
\frac{ \DlGO  }{M_\ell},
\end{split}
\end{equation}
where 
\[
 \Delta_\ell g(\omega) := 
 \begin{cases} 
    \g{\barXL[0](\omega) }, &\text{if} \quad \ell = 0,\\
    \g{\barXL[\ell](\omega) } - \g{\barXL[\ell-1](\omega) }, & \text{otherwise.}
 \end{cases}
\]
Here, the positive integer, $L$, denotes the final level of the estimator, $M_\ell$ is the number of sample
realizations on the $\ell$-th level, and the realization pair,
$\barXL(\omega_{i,\ell})$ and $\barXL[\ell-1](\omega_{i,\ell})$, are
generated by the Euler--Maruyama method~\eqref{eq:eulerMaruyama} using
the \emph{same} Wiener path, $W(\omega_{i,\ell})$, sampled on the
respective meshes, $\Dt{\ell}$ and $\Dt{\ell-1}$,
(cf.~Figure~\ref{fig:Brownian_Bridge}). For consistency, we also
introduce the notation $\WL{\ell}(\omega)$ for the Wiener path
restricted to the mesh $\Dt{\ell}(\omega)$.

\begin{figure}[h!] 
\centering
\includegraphics[width=0.49\textwidth]{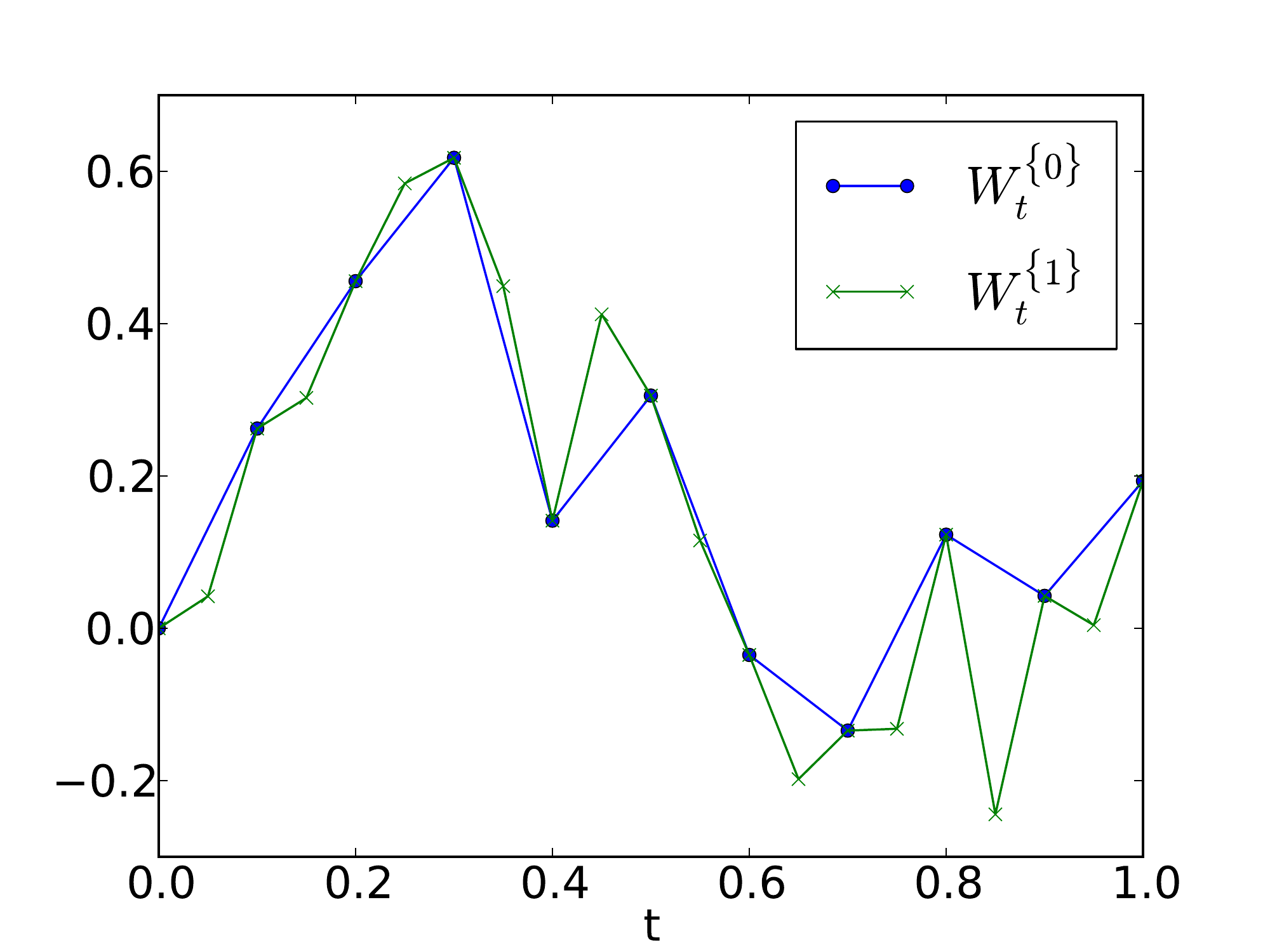}
\includegraphics[width=0.49\textwidth]{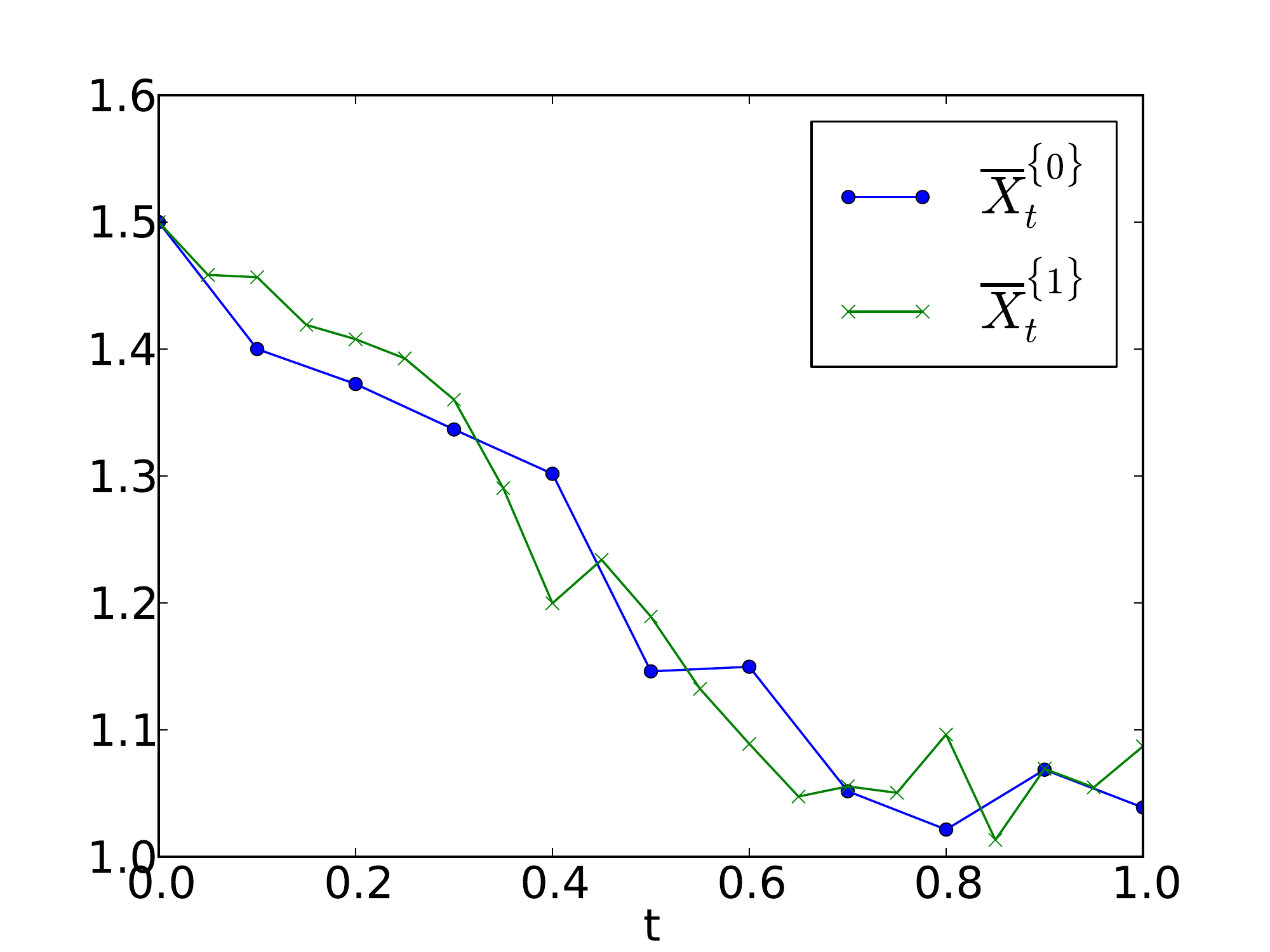}
\caption{({\bf{Left}}) A sample Wiener path, $W$, generated on
  the coarse mesh, $\Dt{0}$, with uniform step size $1/10$ (blue line
  ). The path is thereafter Brownian bridge interpolated onto a
  finer mesh, $\Dt{1}$, which has uniform step size of $1/20$ (green line).
  (\textbf{Right}) Euler--Maruyama numerical solutions of the
  Ornstein-Uhlenbeck SDE problem, $dX_t = 2(1-X_t)dt+ 0.2dW_t$, with
  initial condition $X_0 = 3/2$, are computed on the meshes $\Dt{0}$
  (blue line) and $\Dt{1}$ (green line) using Wiener
  increments from the respective path resolutions. }
\label{fig:Brownian_Bridge}
\end{figure}

\subsubsection{Uniform Time-Stepping MLMC}\label{subsec:uniformMLMC}

In the uniform time-stepping MLMC introduced in~\cite{Giles08}, the
respective SDE realizations $\{ \barXL[\ell]\}_\ell$ are constructed
on a hierarchy of uniform meshes with geometrically decaying step
size, $\mathrm{min} ~\Dt{\ell} = \mathrm{max} ~\Dt{\ell} = T/N_\ell$,
and $N_\ell = c^\ell N_0$ with $c \in \mathbb{N}\setminus \{1\}$ and
$N_0$ a finite integer. For simplicity, we consider the uniform
time-stepping MLMC method with $c=2$.

\subsubsection{Uniform Time-Stepping MLMC Error and Computational Complexity}
By construction, the multilevel estimator is telescoping in expectation, i.e.,
$\Exp{\AMLSimple} =\Exp{\g{\barXL[L]}}$. Using this property, we may
conveniently bound the multilevel approximation error:
\[
\begin{split}
 \abs{\Exp{\gXT} - \AMLSimple}
 \leq   \underbrace{\abs{ \Exp{\gXT - \g{ \barXL[L] }}}}_{=:\mathcal{E}_T }
+\underbrace{\abs{\Exp{ \g{ \barXL[L] } - \AMLSimple}}}_{=:\mathcal{E}_S}.
\end{split}
\]
The approximation goal~\eqref{eq:mlmcGoal} is then reached by ensuring that
the sum of the \emph{bias}, $\mathcal{E}_T$, and the \emph{statistical
  error}, $\mathcal{E}_S$, is bounded from above by $\tol$, e.g., by
the constraints $\mathcal{E}_T \leq \tol/2$ and 
$\mathcal{E}_S \leq \tol/2$, (see Section~\ref{subsec:errorControl} for more details 
on the MLMC error control). For the MSE error goal,
\[
\Exp{ \parenthesis{\Exp{\gXT} - \AMLSimple}^2 } \leq \tol^2,
\]
the following theorem states the optimal computational cost for MLMC:

\begin{theorem}[Computational cost of deterministic MLMC; Cliffe et al.~\cite{Cliffe11}]\label{thm:cliffeComplexity}
  Suppose there are constants $\alpha, \beta, \gamma$ such that
  $\alpha \geq \frac{\min(\beta, \gamma)}{2}$ and
\begin{itemize}
 
 \item[(i)] $\quad \abs{ \Exp{\g{\barXL[\ell]} - \g{X_T}  } } = \bigO{N_\ell^{- \alpha}}$,
 
 \item[(ii)] $\quad \Var{\DlG  }  = \bigO{N_\ell^{- \beta }}$,
 
 \item[(iii)] $\quad \text{Cost}(\DlG ) = \bigO{N_\ell^{\gamma}}$.  

 \end{itemize}

Then, for any $\tol<e^{-1}$, there exists an $L$ and a sequence 
$\{M_\ell\}_{\ell=0}^L$ such that 
\begin{equation}\label{eq:cliffeMseGoal}
 \Exp{ \parenthesis{\AMLSimple - \Exp{\g{X_T}} }^2 } \leq \tol^2, 
\end{equation}
and 
\begin{equation}
\label{eq:MLMCCost}
 \text{Cost}(\AMLSimple) = 
 \begin{cases} 
\bigO{\tol^{-2}}, & \text{if } \beta> \gamma,\\
 \bigO{\tol^{-2}\log(\tol)^2}, & \text{if }  \beta = \gamma,\\
  \bigO{\tol^{-2 + \frac{\beta-\gamma}{\alpha}} } , &\text{if }  \beta < \gamma.
   \end{cases}
\end{equation}
 \end{theorem}
 In comparison, the computational cost of achieving the
 goal~\eqref{eq:cliffeMseGoal} with single-level MC is
 $\bigO{\tol^{-2-\gamma/\alpha}}$. Theorem~\ref{thm:cliffeComplexity}
 thus shows that for any problem with $\beta>0$, MLMC will
 asymptotically be more efficient than single-level MC. Furthermore,
 the performance gain of MLMC over MC is particularly apparent in
 settings where $\beta \geq \gamma$. The latter property is linked to
 the contributions of this work. In low-regularity SDE
 problems, e.g., Example~\ref{ex:mlmcSingularityGbm} below 
and~\cite{Yan02, Avikainen09}, the uniform time-stepping Euler--Maruyama
 results in convergence rates for which $\beta < \gamma$.
 More sophisticated integrators can preserve rates such that 
 $\beta \geq \gamma$. 


\begin{remark}
  Similar accuracy vs. complexity results to
  Theorem~\ref{thm:cliffeComplexity}, requiring slightly stronger
  moment bounds, have also been derived for the approximation
  goal~\eqref{eq:mlmcGoal} in the asymptotic setting when $\tol \downarrow 0$,
  cf.~\cite{HoelMLMC14, Collier14}.
\end{remark}

\subsubsection{MSE A Posteriori Adaptive Time-Stepping}

In general, adaptive time-stepping algorithms seek to fulfill one of 
two equivalent goals~\cite{Bangerth03}:
\begin{enumerate}

\item[(B.1)] Provided a computational budget $N$ and a norm $\| \cdot \|$, determine the possibly non-uniform mesh,
 which minimizes the error $\left\| \g{X_T} - \g{\barX_T} \right\|$.

\item[(B.2)] Provided an error constraint $\left\|\g{X_T} - \g{\barX_T} \right\| \leq
  \tol$, determine the possibly non-uniform mesh, which achieves
  the constraint at the minimum computational cost.
\end{enumerate}

Evidently, the refinement criterion of an adaptive algorithm depends
on the error one seeks to minimize.  In this work, we consider
adaptivity goal (B.1) with the error measured in terms of
the MSE. This error measure is suitable for MLMC algorithms
as it often will lead to improved convergence rates, $\beta$ (since
$\Var{\DlG } \leq \Exp{\DlG^2} $), which by
Theorem~\ref{thm:cliffeComplexity} may reduce the computational
cost of MLMC. In Theorem~\ref{thm:mseExpansion1D}, 
we derive the following error expansion for the MSE of
Euler--Maruyama numerical solutions of the SDE~\eqref{eq:sdeProblem}:
\begin{equation}\label{eq:errorExpansionIntro}
\Exp{\parenthesis{g(X_T) -\gBarXT}^2} = \Exp{ \sum_{n=0}^{N-1}   \rhoBar_n
  \Delta t_n^2  + \littleO{\Delta t_n^2} },
\end{equation}
where the error density, $\rhoBar_n$, is a function of the local error
and sensitivities from the dual solution of the SDE problem, as
defined in~\eqref{eq:rhoBar}. The error
expansion~\eqref{eq:errorExpansionIntro} is an \aposteriori
error estimate for the MSE, and in our
adaptive algorithm, the mesh is refined by equilibration of the
expansion's \emph{error indicators} 
\begin{equation}\label{eq:rBarIntro}
\rBar_n  := \rhoBar_n \Delta t_n^2, \quad \text{for} \quad n=0,1,\ldots, N-1.
\end{equation}

\subsubsection{An MSE Adaptive MLMC Algorithm}\label{subsubsec:adaptiveMlmcMesh}
Using the described MSE adaptive algorithm, we construct an MSE
adaptive MLMC algorithm in Section~\ref{sec:adaptiveMLMC}. The MLMC
algorithm generates SDE realizations, $\{ \barXL[\ell]\}_\ell$,
on a hierarchy of \emph{pathwise} adaptively refined meshes,
$\{ \Dt{\ell} \}_\ell$. The meshes are nested, i.e., for all realizations
$\omega \in \Omega$, 
\[
\Dt{0}(\omega) \subset \Dt{1}(\omega) \subset \ldots \Dt{\ell}(\omega)
\subset \ldots, 
\]
with the constraint that the number of time steps in
$\Dt{\ell}$,
$ \abs{ {\Dt{\ell}}}$,
is bounded  by
$2 N_\ell$:
\begin{equation*}
\abs{ \Dt{\ell} } < 2N_\ell= 2^{\ell+2} N_{-1}.
\end{equation*}
Here, $N_{-1}$ 
denotes the pre-initial number of time steps; it is a bounded 
integer set in advance of the computations.  
This corresponds to the hierarchy setup for the uniform 
time-stepping MLMC algorithm in Section~\ref{subsec:uniformMLMC}.

The potential efficiency gain of adaptive MLMC is experimentally 
illustrated in this work using the drift blow-up problem
\[
dX_t =  \frac{r X_t}{|t-\xi|^p} \,dt+ \sigma X_t\, dW_t, \quad X_0 = 1.
\] 
This problem is addressed in Example~\ref{ex:mlmcSingularityGbm} for the three
different singularity exponents $p = 1/2, 2/3$ and $3/4$, with a
pathwise, random singularity point $\xi \sim U(1/4, 3/4)$, an observable
$g(x) = x$, and a final time $T=1$.  For the given singularity
exponents, we observe experimental deteriorating convergence rates, $\alpha = (1-p)$
and $\beta =2(1-p)$, for the uniform time-stepping Euler--Maruyama integrator, while for
the adaptive time-step Euler--Maruyama we observe $\alpha \approx
1$ and $\beta \approx 1$. Then, as predicted by
Theorem~\ref{thm:cliffeComplexity}, we also observe an order
of magnitude difference in computational cost between the two algorithms
(cf.~Table~\ref{tab:complexityResults}).
\begin{table}[h!]
\caption{Observed computational cost -- disregarding
    $\log(\tol)$ multiplicative factors of finite order -- for the drift blow-up study in 
    Example~\ref{ex:mlmcSingularityGbm}.  } 
  \label{tab:complexityResults}
  \centering
  {
  \begin{tabular}{p{4.5cm}p{2.8cm}p{2.8cm}}
    \hline \noalign{\smallskip}
     & \multicolumn{2}{ c}{Observed computational cost}\\
     \multicolumn{1}{c}{Singularity exponent $p$} &  Adaptive MLMC & Uniform MLMC \\
    \noalign{\smallskip}\hline\noalign{\smallskip}
    \multicolumn{1}{c}{$1/2$} & $\tol^{-2}$ & $\tol^{-2}$\\
    \multicolumn{1}{c}{$2/3$} & $\tol^{-2}$ & $\tol^{-3}$\\
    \multicolumn{1}{c}{$3/4$} & $\tol^{-2}$ & $\tol^{-4}$\\
     \noalign{\smallskip}\hline\noalign{\smallskip}
    \end{tabular}}
\end{table}

\subsubsection{Earlier Works on Adaptivity for SDE}

Gaines' and Lyons' work~\cite{Gaines97} is one of the
seminal contributions on adaptive algorithms for SDE. They present an algorithm
that seeks to minimize the \emph{pathwise} error of the mean and
variation of the local error conditioned on the $\sigma$-algebra
generated by (i.e., the values at which the
Wiener path has been evaluated in order to numerically
integrate the SDE realization) $\{W_{t_n}\}_{n=1}^{N}$. The method may be used in
combination with different numerical integration methods, and an
approach to approximations of potentially needed 
Levy areas is proposed, facilitated by a binary tree
representation of the Wiener path realization at its
evaluation points. As for \aposteriori
adaptive algorithms, the error indicators in Gaines' and Lyons' algorithm
are given by products of local errors and weight terms, but, unlike
in \aposteriori methods, the weight
terms are computed from \apriori estimates, making their
approach a hybrid one.

Szepessy et al.~\cite{Szepessy01} introduced \aposteriori weak
error based adaptivity for the
Euler--Maruyama algorithm with numerically computable error indicator
terms.  Their
development of weak error adaptivity took inspiration from Talay
and Tubaro's seminal work~\cite{Talay90}, where an error
expansion for the weak error was derived for the Euler--Maruyama
algorithm when uniform time steps were used. In~\cite{HoelMLMC14}, Szepessy
et al.'s weak error adaptive algorithm was used in the construction of a
weak error adaptive MLCM algorithm. To the best of our knowledge, the
present work is the first on MSE \aposteriori adaptive algorithms for
SDE both in the MC- and MLMC setting.

Among other adaptive algorithms for SDE, many have refinement
criterions based only or primarily on estimates of the local error.
For example in \cite{Hofman00}, where the step-size depends on the size of
the diffusion coefficient for a MSE Euler--Maruyama adaptive
algorithm; in \cite{Mattingly05}, the step-size is controlled by the
variation in the
size of the drift coefficient in the constructed
Euler--Maruyama adaptive algorithm, which preserves the long-term ergodic
behavior of the true solution for many SDE problems;
and in \cite{Silvana12}, a local error based adaptive Milstein
algorithm is developed for solving multi-dimensional chemical Langevin
equations.

\addtocontents{toc}{\protect\setcounter{tocdepth}{1}}

\addtocontents{toc}{\protect\setcounter{tocdepth}{2}}

\section{Derivation of the MSE A Posteriori Adaptive Algorithm}\label{sec:mseAdaptivity}
In this section, we construct an MSE \aposteriori adaptive algorithm
for SDE whose realizations are numerically integrated by the Euler--Maruyama
algorithm~\eqref{eq:eulerMaruyama}. Our goal is, in rough terms, to obtain an 
algorithm for solving the SDE problem~\eqref{eq:sdeProblem} that for a
fixed number of intervals $N$, determines the time-stepping, $\Delta t_0,
\Delta t_1, \ldots, \Delta t_{N-1}$ such that the MSE,
$\Exp{ \parenthesis{ \gBarXT - \g{X_T} }^2 }$ is minimized.  That is,
\begin{equation}\label{eq:adaptivityGoal}
\Exp{ \parenthesis{ \gBarXT -  \g{X_T} }^2 } \to \min!, \quad N \quad \text{given}
\end{equation}
The derivation of our adaptive algorithm consists of two steps. First,
an error expansion for the MSE is presented in Theorem~\ref{thm:mseExpansion1D}. Based
on the error expansion, we thereafter construct a mesh refinement
algorithm. At the end of the section, we apply the adaptive algorithm
to a few example problems.

\subsection{The Error Expansion}
Let us now present a leading-order error expansion for the
MSE~\eqref{eq:adaptivityGoal} of the SDE problem~\eqref{eq:sdeProblem}
in the one-dimensional (1D) setting, i.e., when $X_t \in \R$ and the drift and diffusion
coefficients are respectively of the form $a: [0,T] \times \R \to \R$
and $b: [0,T] \times \R \to \R$.  An extension of the MSE error
expansion to multi-dimensions is given in
Appendix~\ref{subsec:mseHihgerDim}. To state the error expansion
Theorem, some notation is needed. Let $X_s^{x,t}$ denote the
solution of the SDE~\eqref{eq:sdeProblem} at time $s \ge t$, when the
initial condition is $X_t =x$ at time $t$, i.e.,
\begin{equation}\label{eq:sdeFlow}
X_s^{x,t} := x + \int_t^s a(u, X_u) du + \int_t^s b(u, X_u) dW_u, \qquad s\in [t, T],
\end{equation}
and in light of this notation, $X_t$ is shorthand for
$X_t^{x_0,0}$. For a given observable $g$, the payoff-of-flow map
function is defined by $\varphi(t, x)= g(X_T^{x,t})$. We 
also make use of the following function space notation
\[
\begin{split}
&C(U)             := \{f:U \to \R \, |\, f \text{ is continuous} \},\\
&C_b(U)           := \{f:U \to \R \, |\, f \text{ is continuous and bounded} \},\\
&C_b^k(\R)         := \Big\{f:\R \to \R \, |\, f \in C(\R) \text{ and
} \frac{d^j}{dx^j} f \in C_b(\R) \text{ for all integers }  1 \le j \leq k \Big\},\\
&C_b^{k_1,k_2}([0,T] \times \R):= \Big\{f:[0,T] \times \R \to \R \, |\, f \in C([0,T] \times \R) \text{ and } \\
&\quad \partial_t^{j_1}\partial_x^{j_2} f  \in C_b([0,T]\times \R) 
\text{ for all integers s.t. }  j_1 \leq k_1 \text{ and } 1 \le j_1+j_2 \leq k_2 \Big\}.
\end{split}
\]

We are now ready to present our mean square expansion result, namely,
\begin{theorem}[1D MSE leading-order error expansion]~\label{thm:mseExpansion1D} 
Assume that drift and diffusion coefficients and input data of the SDE~\eqref{eq:sdeProblem} 
fulfill
\begin{enumerate}
\item[(R.1)]
$a,b \in C_b^{2,4}([0,T]\times \R)$,

\item[(R.2)] there exists a constant $C>0$ such that
\begin{align*}
|a(t,x)|^2 + |b(t,x)|^2 & \leq C(1+|x|^2), & \forall x \in \R \text{ and } \forall t \in [0,T],\\
\end{align*}

\item [(R.3)] $g' \in C^3_b(\R)$ and 
there exists a $k \in \N$ such 
\begin{equation}\label{eq:growthG}
|g(x)| + |g'(x)| \leq C(1+|x|^k), \qquad \forall x \in \R,
\end{equation}

\item[(R.4)] for the initial data, $X_0 \in \FCal_0$ and $\Exp{|X_0|^p} < \infty$ for all $p\ge 1$.
\end{enumerate}
Assume further the mesh points $0=t_0<t_1< \ldots <t_N = T$ 
\begin{enumerate}
\item[(M.1)] are stopping times for which
$t_n \in \FCal_{t_{n-1}}$ for $n=1,2,\ldots, N$,

\item[(M.2)] for all mesh realizations, there exists a deterministic integer, $\check N$, and a
  $c_1>0$ such that 
$c_1\check{N} \le N \le \check N$ and
a $c_2>0$ such that $\max_{n\in \{0,1,\ldots, N-1\}} \Delta t_n < c_2 \check{N}^{-1}$,

\item[(M.3)] and there exists a $c_3>0$ such that for all $p \in [1,8]$ 
and $n \in \{0,1,\ldots, \check{N}-1\}$
\[
\Exp{\Delta t_n^{2p}} \leq c_3 \parenthesis{\Exp{ \Delta t_n^2 } }^{p}.
\]  

\end{enumerate}
Then, as $\check{N}$ increases, 
\begin{equation}\label{eq:mseExpansion1DNotComp}
\Exp{\parenthesis{g(X_T) -\gBarXT}^2} = \sum_{n=0}^{\check{N}-1} \Exp{ 
  \phiX{t_n,\barX_{t_n}} \frac{(b_x b)^2}{2}(t_n, \barX_{t_n}) \Delta
  t_n^2  + \littleO{\Delta t_n^2} },
\end{equation}
where we have defined $t_n = T$ and $\Delta t_n=0$ for all $n \in
\{N,N+1, \ldots, \check{N}\}$.  And replacing the first variation,
$\phiX{t_n,\barX_n}$, by the numerical approximation, $\phiBarX{n}$,
as defined in \eqref{eq:backwardScheme1D}, yields the following to
leading order all-terms-computable error expansion:
\begin{equation}\label{eq:mseExpansion1D}
\Exp{\parenthesis{g(X_T) -\gBarXT}^2} = \sum_{n=0}^{\check{N}-1} \Exp{ 
  \phiBarX{n}^2  \frac{(b_x b)^2}{2}(t_n, \barX_{t_n}) \Delta t_n^2  +
\littleO{\Delta t_n^2} }.
\end{equation}

\end{theorem}
\begin{remark}
In condition (M.2) of the above theorem we have introduced $\check N$
to denote the deterministic upper bound for the number of time steps
in all mesh realizations. Moreover, from this point on the mesh points
$\{t_n\}_n$ and time steps $\{\Delta t_n\}_n$ are defined for all
indices $\{0,1,\ldots,\check{N}\}$ with the natural extension $t_n =
T$ and $\Delta t_n=0$ for all $n\in \{N+1, \ldots, \check{N}\}$. In
addition to ensuring an upper bound on the complexity of a numerical
realization and that $\max_n \Delta t_n \to 0$ as $\check N \to \infty$,
replacing the random $N$ (the smallest integer value for which $t_N =
T$ in a given mesh) with the deterministic $\check N$ in the MSE error
expansion~\eqref{eq:mseExpansion1D} simplifies our proof of
Theorem~\ref{thm:mseExpansion1D}.
\end{remark}

\begin{remark}
  For most SDE problems on which it is relevant to apply
  \aposteriori adaptive integrators, at least one of the
  regularity conditions (R.1), (R.2), and (R.3) and the mesh
  adaptedness assumption (M.1) in Theorem~\ref{thm:mseExpansion1D} will
  not be fulfilled. In our adaptive algorithm, the error
  expansion~\eqref{eq:mseExpansion1D} is interpreted in a formal sense
  and only used to facilitate the systematic construction of a mesh refinement
  criterion. 
  
  When applied to low-regularity SDE problems where some of the
  conditions (R.1), (R.2), or (R.3), do not hold, the actual
  leading-order term of the error expansion~\eqref{eq:mseExpansion1D}
  may contain other or additional terms besides $\phiBarX{n}^2
  \frac{(b_x b)^2}{2}(t_n, \barX_{t_n}) $ in the error density.
  Example~\ref{ex:mlmcSingularityGbm} presents a problem where
  \emph{ad hoc} additional terms are added to the error density.
\end{remark}

\subsubsection{Numerical Approximation of the First Variation}\label{subsec:firstVar}

The first variation of the flow map, $\varphi(t, x)$, is defined by
\[
\phiX{t,x} = \partial_x g(X^{x,t}_t) = g'(X_T^{x,t}) \partial_x X_T^{x,t}
\]
and the first variation of the path itself, $\partial_x X_s^{x,t}$, 
is the solution of the linear SDE
\begin{equation}\label{eq:sdeFlowFirstVar}
\begin{split}
d(\partial_x X_s^{x,t}) &= a_x(s,X_s^{x,t}) \partial_x X_s^{x,t} ds + b_x(s,X_s^{x,t}) \partial_x X_s^{x,t} dW_s,\quad s \in (t,T], \\
\partial_x X_t^{x,t} & =1.
\end{split}
\end{equation}
To describe conditions under which the terms $g'(X_s^{x,t})$ and 
$\partial_x X_s^{x,t}$ are well defined, let us first recall that if 
$X_s^{x,t}$ solves the SDE~\eqref{eq:sdeFlow} and
\[
\Exp{\int_t^T |X_s^{x,t}|^{2} ds } < \infty,
\]
then we say that there exists a solution to the SDE;
and if $\widetilde{X}^{x,t}_s$ is another solution 
of the SDE with the same initial condition, then 
we say the solution is pathwise unique provided that
\[
\Prob{\sup_{s \in [t, T]} \abs{X_s^{x,t} - \widetilde{X}^{x,t}_s} > 0} =0.
\] 

\begin{lemma}\label{lem:existUniquePath}
Assume the regularity assumptions (R.1), (R.2), (R.3), and (R.4) in
Theorem~\ref{thm:mseExpansion1D} hold, and that for any fixed $t \in [0,T]$,
$x \in \FCal_t$ and $\Exp{|x|^{2p}}<\infty$, for all $p\in \N$. Then
there exist pathwise unique solutions $X_s^{x,t}$ and $\partial_x
X_s^{x,t}$ to the respective SDE~\eqref{eq:sdeFlow}
and~\eqref{eq:sdeFlowFirstVar} for which
\[
\max\left\{ \Exp{\sup_{s \in [t, T]} \abs{X^{x,t}_s}^{2p}}, \Exp{\sup_{s \in [t, T]} \abs{\partial_x X^{x,t}_s}^{2p}} \right\}  
< \infty, \quad \forall p \in \N.
\]
Furthermore, $\phiX{t,x} \in \FCal_T$ and 
\[
\Exp{ |\phiX{t,x}|^{2p} } < \infty, \quad \forall p \in \N.
\]
\end{lemma}

\begin{proof}
By writing $(Y_s^{(1)}, Y^{(2)}_s ):= (X^{x,t}_s, \partial_x X^{x,t}_s)$, 
 \eqref{eq:sdeFlow} and~\eqref{eq:sdeFlowFirstVar} together form 
a system of SDE:
\begin{equation}\label{eq:sdeFlowSystem2}
\begin{split}
dY_s^{(1)} & = a(s, \yy 1 s ) ds + b(s, \yy 1 s ) dW_s \\
dY_s^{(2)} & = a_x(s, \yy 1 s) \yy 2 s ds + b_x(s, \yy 1 s)\yy 2 s  dW_s \\
\end{split}
\end{equation}  
for $s \in (t, T]$ and with initial condition $Y_t = (x, 1)$. 
By the Lipschitz continuity and the linear growth bound of this system's 
drift and diffusion coefficients, 
there exists a pathwise unique solution of the SDE~\eqref{eq:sdeFlowSystem2} for which
\[
 \Exp{\sup_{s \in [t, T]}|Y_s|^{2p}} < \infty, \quad \forall p \in \N,
\]
(cf.~\cite[Theorems 4.5.3 and 4.5.4 and Exercise 4.5.5]{Kloeden92}). 
As solutions of the \Ito SDE, $X^{x,t}_T, \partial_x X^{x,t}_T \in \FCal_T$, and 
since we assume that $g' \in C_b^3(\R)$, we know that $g'$ is Borel measurable and so is 
the mapping $f:\R^2 \to \R$ defined by $f(x,y) = xy$.
From this we conclude that $\phiX{x,t} = f(g'(X^{x,t}_T), \partial_x X^{x,t}_T) \in \FCal_T$ and,
by~\eqref{eq:growthG}, H\"older's and Minkowski's inequalities, that for any $p \in \N$, 
\[
\Exp{ |\phiX{t,x}|^{2p} } \leq 
C\sqrt{\Exp{ \parenthesis{1+\abs{X_T^{x,t}}^{k}}^{4p} }\Exp{ \abs{\partial_x X_T^{x,t}}^{4p} }} <  \infty. 
\]
\end{proof}

To obtain an all-terms-computable error expansion in
Theorem~\ref{thm:mseExpansion1D}, which will be needed to construct
an \aposteriori adaptive algorithm, the first variation of the flow map,
$\varphi_x$, is approximated by the first
variation of the Euler--Maruyama numerical solution,
\[
\phiBarX{n} :=  g'( \barX_T) \partial_{\barX_{t_n}} \barX_T.
\]
Here, for $k>n$, $\partial_{\barX_{t_n}} \barX_{t_k}$ is the
solution of the Euler--Maruyama scheme
\begin{equation}\label{eq:firstVarNum}
(\partial_{\barX_{t_n}} \barX)_{t_{j+1}} 
=  (\partial_{\barX_{t_n}} \barX)_{t_{j}} +  a_x(t_j, \barX_{t_j}) (\partial_{\barX_{t_n}} \barX)_{t_{j}} \Delta t_j 
+ b_x(t_j, \barX_{t_j}) (\partial_{\barX_{t_n}} \barX)_{t_{j}} \Delta W_j, 
\end{equation}
for $j =n,n+1, \ldots k-1$ and with the initial condition $\partial_{\barX_{t_n}} \barX_{t_n} = 1$,
which is coupled to the numerical solution of the SDE, $\barX_{t_j}$.
\begin{lemma}\label{lem:numSolConv}
If the assumptions (R.1), (R.2), (R.3), (R.4), (M.1) and (M.2) in 
Theorem~\ref{thm:mseExpansion1D} hold,
then the numerical solution $\barX$ of~\eqref{eq:eulerMaruyama} 
converges in mean square sense to the solution of the SDE~\eqref{eq:sdeProblem}, 
\begin{equation}\label{eq:strongConvPath}
\max_{1\leq n \leq \check{N} } \parenthesis{\Exp{ \abs{\barX_{t_n} - X_{t_n}}^{2p}}}^{1/2p} \leq 
C \check{N}^{-1/2},
\end{equation}
and
\begin{equation}\label{eq:boundNumSol}
\max_{1\leq n \leq \check{N} } \Exp{ \abs{\barX_{t_n}}^{2p} } < \infty, \quad \forall p \in \N.
\end{equation}
For any fixed $1\leq n \leq \check{N} $, 
the numerical solution $\partial_{\barX_{t_n}} \barX$ of~\eqref{eq:firstVarNum} 
converges in mean square sense to $\partial_{x} X^{X_{t_m}, t_m}$,
\begin{equation}\label{eq:strongConvFirstVar}
\max_{n \leq k \leq \check{N}  } \parenthesis{\Exp{ \abs{\partial_{\barX_{t_n}} \barX_{t_k} - \partial_{x} X_{t_k}^{X_{t_n}, t_n}}^{2p}}}^{1/2p} \leq 
C \check N^{-1/2}.
\end{equation} 
and 
\begin{equation}\label{eq:boundNumSolFirstVar}
\max_{n \leq k \leq \check{N}} \Exp{ \abs{\partial_{\barX_{t_n}} \barX_{t_k}}^{2p} } < \infty, \quad \forall p \in \N.
\end{equation}
Furthermore, $\phiBarX{n} \in \FCal_T$ and 
\begin{equation}\label{eq:phiBarBound}
\Exp{ |\phiBarX{n}|^{2p} } < \infty, \quad \forall p \in \N.
\end{equation}
\end{lemma}

\begin{proof}
The system $\overline Y_{t_{k}} := (\barX_{t_k}, \partial_{\barX_{t_n}} \barX_{t_k} )$
provides solutions approximating the SDE~\eqref{eq:sdeFlowSystem2} that
are generated by the Euler--Maruyama scheme  
\begin{equation}\label{eq:sdeNumFlowSystem2}
\begin{split}
Y_{t_{k+1}}^{(1)} & = Y_{t_{k}}^{(1)} + a(t_k, Y_{t_k}^{(1)})   \Delta t_k            + b(t_k, Y_{t_k}^{(1)}) \Delta W_k, \\
Y_{t_{k+1}}^{(2)} & = Y_{t_{k}}^{(2)} + a_x(t_k, Y_{t_k}^{(1)})Y^{(2)}_{t_k} \Delta t_k  + b_x(t_k, Y_{t_k}^{(1)})Y^{(2)}_{t_k} \Delta W_k ,
\end{split}
\end{equation}
for $k \ge n$ and with initial condition $\overline Y_{t_{n}} = (\barX_{t_n},  1)$. 
The assumptions of~\cite[Theorem 10.6.3 and the remark following it]{Kloeden92}
are fulfilled for the SDE~\eqref{eq:sdeFlowSystem2}, 
which implies the strong convergence of $\overline Y$ to $Y$ and 
that the inequalities~\eqref{eq:strongConvPath},~\eqref{eq:boundNumSol},~\eqref{eq:strongConvFirstVar}, 
and~\eqref{eq:boundNumSolFirstVar} hold.
That $\phiBarX{n} \in \FCal_T$ and that inequality~\eqref{eq:phiBarBound} holds 
can be shown by a similar argument as in the proof of Lemma~\ref{lem:existUniquePath}.
\end{proof}

From the SDE~\eqref{eq:sdeNumFlowSystem2}, it is clear that
\[
\partial_{\barX_n} \barX_T = \prod_{k=n}^{N-1} \parenthesis{1+ a_x(t_k, \barX_{t_k}) \Delta t_k + b_x(t_k, \barX_{t_k}) \Delta W_k},
\]
and this implies that $\phiBarX{n}$ solves the backward scheme
\begin{equation}\label{eq:backwardScheme1D}
\phiBarX{n} =   c_x(t_n,\barX_{t_n})\phiBarX{n+1}, \quad n = N-1, N-2, \ldots, 0,
\end{equation}
with the initial condition $\phiBarX{N} = g'(\barX_T)$ and the shorthand notation
\begin{equation*}\label{eq:cDefinition}
c(t_n, \barX_{t_n}) := \barX_{t_n} + a(t_n, \barX_{t_n}) \Delta t_n +  b(t_n, \barX_{t_n}) \Delta W_n.
\end{equation*}
The backward scheme~\eqref{eq:backwardScheme1D} is convenient from a computational perspective since
it implies that the set of points, $\{\phiBarX{n}\}_{n=0}^{N}$, can be computed at the same
cost as that of one-path realization, $\{ \barX_{t_n} \}_{n=0}^N$, which
can be verified as follows 
\[
\begin{split}
\phiBarX{n} &= g'(\barX_T) \prod_{k=n}^{N-1} c_x(t_k,\barX_{t_k}) \\
& = c_x(t_n,\barX_{t_n})  g'(\barX_T) \prod_{k=n+1}^{N-1} c_x(t_k,\barX_{t_k}) \\
& = c_x(t_n,\barX_{t_n})  g'(\barX_T)  \partial_{t_{n+1}} \barX_T\\
&=  c_x(t_n,\barX_{t_n}) \phiBarX{n+1}.
\end{split}
\]


\subsection{The Adaptive Algorithm}\label{subsec:adaptiveAlgorithm}
 
Having derived computable expressions for all terms in the error
expansion, we next introduce the error density
\begin{equation}\label{eq:rhoBar}
\rhoBar_n := \phiBarX{n}^2 \frac{(b_x b)^2}{2}(t_n, \barX_{t_n}), \quad n =0,1,\ldots, N-1,
\end{equation}
and, for representing the numerical solution's error contribution from
the time interval $(t_{n}, t_{n+1})$, the error indicators
\begin{equation}\label{eq:rBar}
\rBar_n  := \rhoBar_n \Delta t_n^2, \quad n=0,1,\ldots, N-1.
\end{equation}
The error expansion~\eqref{eq:mseExpansion1D} may then be written as
\begin{equation}\label{eq:errorIndExpansion}
\Exp{\parenthesis{g(X_T) -\gBarXT}^2} = \sum_{n=0}^{\check{N}-1} \Exp{    \rBar_n + \littleO{\Delta t_n^2} }.
\end{equation}
The final goal of the adaptive algorithm is minimization of the
leading order of the MSE in~\eqref{eq:errorIndExpansion}, namely, $\Exp{ \sum_{n=0}^{N-1} \rBar_n  }$,
which (for each realization) is approached by minimization of the
error expansion realization $\sum_{n=0}^{N-1} \rBar_n$.
An approximately optimal choice for the refinement procedure can be
derived by introducing the Lagrangian
\begin{equation}\label{eq:timeStepLagrangian}
\mathcal{L}(\Delta t, \lambda) = \int_0^T \rho(s) \Delta t(s) ds
+ \lambda ( \int_0^T \frac{1}{\Delta t(s)} ds  - \check{N} ),
\end{equation}
for which we seek to minimize the pathwise squared error   
\[
\parenthesis{g(X_T) -\gBarXT}^2 = \int_0^T \rho(s) \Delta t(s) ds
\]
under the constraint that 
\begin{equation*}\label{eq:mseConstraintLagrangian}
\int_0^T \frac{1}{\Delta t(s)} ds = \check{N},
\end{equation*}
for a fixed number of time steps, $\check{N}$, and the 
implicit constraint that the error indicators are equilibrated,
\begin{equation}\label{eq:errIndEquilibration}
 \rBar_n = \rhoBar_n \Delta t_n^2 = \frac{\parenthesis{g(X_T) -\gBarXT}^2}{\check{N}}, \quad n=0,1,\ldots, \check{N}-1.
\end{equation}

Minimizing~\eqref{eq:timeStepLagrangian} yields
\begin{equation}\label{eq:solTimeStepLagrangian}
\Delta t_n = \sqrt{\frac{\parenthesis{g(X_T) -\gBarXT}^2}{\check{N} \, \rho(t_n) } } \quad \text{and} \quad 
\text{MSE}_{\text{adaptive}}  \leq \frac{1}{\check{N}}\Exp{ \left( \int_{0}^T \sqrt{\rho(s)} \, ds  \right)^2 },
\end{equation}
where the above inequality follows from using H\"older's inequality,
\[
\begin{split}
\Exp{ \parenthesis{g(X_T) -\gBarXT}^2 }  & = \frac{1}{\sqrt{\check{N}}}  \Exp{  \abs{ g(X_T) -\gBarXT} \int_{0}^T \sqrt{\rho(s)} \, ds }  \\
& \leq \frac{1}{\sqrt{\check{N}}}  \sqrt{ \Exp{  \abs{g(X_T) -\gBarXT}^2} } \sqrt{ \Exp{ \parenthesis{\int_{0}^T \sqrt{\rho(s)} \, ds }^2} }.
\end{split}
\]
In comparison, we notice that if a uniform mesh is used, the MSE
becomes
\begin{equation}\label{eq:nUniformBound}
\text{MSE}_{\text{uniform}}  = \frac{T}{\check{N}} \Exp{ \int_{0}^T \rho(s) \, ds }.
\end{equation}
A consequence of observations~\eqref{eq:solTimeStepLagrangian}
and~\eqref{eq:nUniformBound} is that for many low-regularity problems,
for instance, if $\rho(s) = s^{-p}$ with $p \in [1,2)$, adaptive
  time-stepping Euler--Maruyama methods may produce more accurate solutions
  (measured in the MSE) than are obtained using the uniform
  time-stepping Euler--Maruyama method under the same computational
  budget constraints.

\subsubsection{Mesh Refinement Strategy}\label{subsec:meshRefinement}
To equilibrate the error indicators~\eqref{eq:errIndEquilibration}, 
we propose an iterative mesh refinement strategy to
identify the largest error indicator and then refining the
corresponding time step by halving it.
To compute the error indicators prior to refinement, the
algorithm first computes the numerical SDE solution, $\barX_{t_n}$, and the
corresponding first variation $\phiBarX{n}$ (using equations
\eqref{eq:eulerMaruyama} and \eqref{eq:backwardScheme1D} respectively)
on the initial mesh, $\Dt 0$. Thereafter, the error indicators $\rBar_n$
are computed by Equation~\eqref{eq:rBar} and the mesh is refined 
a prescribed number of times, $\nRefine$, as follows:
\begin{enumerate}

\item[(C.1)] Find the largest error indicator 
\begin{equation}\label{eq:errorIndMax}
n^* := \arg \max_n \rBar_n, 
\end{equation}
and refine the corresponding time step by halving
\begin{equation}\label{eq:intervalHalving}
(t_{n^*}, t_{n^*+1})  \to \Big( t_{n^*}, \underbrace{\frac{t_{n^*} + t_{n^*+1}}{2}}_{=t_{n^*+1}^{new}} , \underbrace{t_{n^*+1}}_{=t_{n^*+2}^{new}} \Big), 
\end{equation}
and increment the number of refinements by one.

\item[(C.2)] Update the values of the error indicators, either by
  recomputing the whole problem or locally by interpolation, cf.~Section~\ref{subsubsec:errorIndicators}.

\item[(C.3)] Go to step (C.4) if $\nRefine$ 
  mesh refinements have been made; otherwise, return to step (C.1).

\item[(C.4)] (Postconditioning) Do a last sweep over the mesh
  and refine by halving every time step that is strictly
  larger than $\DtMax$, where $\DtMax = \bigO{\check{N}^{-1} }$
  denotes the maximum allowed step size.

\end{enumerate}
The postconditioning step (C.4) ensures that all time steps become
infinitesimally small as the number of time steps $N \to \infty$
with such a rate of decay that condition (M.2) in Theorem~\ref{thm:mseExpansion1D}
holds and is thereby one of the necessary conditions from
Lemma~\ref{lem:numSolConv} to ensure strong convergence for the
numerical solutions of the MSE adaptive Euler--Maruyama algorithm.
However, the strong convergence result should primarily be interpreted as a motivation for
introducing the postconditioning step (C.4) since Theorem~\ref{thm:mseExpansion1D}'s
assumption (M.1), namely that the mesh points are stopping times for which $t_{n}
\in \FCal_{t_{n-1}}$, will not hold in general for our adaptive algorithm.

\subsubsection{Wiener Path Refinements}\label{subsec:incrementRefinement}

When a time step is refined, as described in~\eqref{eq:intervalHalving}, the
Wiener path must be refined correspondingly.
The value of the Wiener path at the midpoint between $W_{t_{n^*}}$ and $W_{t_{n^*+1}}$ 
can be generated by Brownian bridge interpolation,
\begin{equation}\label{eq:brownianBridge}
W_{t_{n^*+1}^{new}} = \frac{W_{t_{n^*}} +W_{t_{n^*+1}}}{2} + \xi \frac{\sqrt{\Delta t_{n^*}}}{2},
\end{equation}
where $\xi \sim N(0,1)$, cf.~\cite{Oksendal98}. See
Figure~\ref{fig:Brownian_Bridge} for an illustration of Brownian
bridge interpolation applied to numerical solutions of an
Ornstein-Uhlenbeck SDE.

\subsubsection{Updating the Error Indicators}\label{subsubsec:errorIndicators}
After the refinement of an interval, $(t_{n^*}, t_{n^*+1})$, and its
Wiener path, error indicators must also be updated before moving
on to determine which interval is next in line for refinement. There are
different ways of updating error indicators. One expensive but more
accurate option is to recompute the error indicators completely by
first solving the forward problem~\eqref{eq:eulerMaruyama} and the
backward problem~\eqref{eq:backwardScheme1D}. A less costly but also
less accurate alternative is to update only the error indicators
locally at the refined time step by one forward and
backward numerical solution step, respectively:
\begin{equation}\label{eq:errorIndInterpolation}
\begin{split}
 \barX_{t_{n^*+1}}^{new} &= \barX_{t_{n^*}} + a(t_{n^*}, \barX_{t_{n^*}}) \Delta t_{n^*}^{new} 
 + b(t_{n^*},\barX_{t_{n^*}}) \Delta W_{n^*}^{new},\\
 \phiBarX{n^*+1}^{new} &= c_x(t_{n^*}^{new},\barX_{t_{n^*}^{new}})\phiBarX{n^*+1}.
 \end{split}
\end{equation}
Thereafter, we compute the resulting error density, $\rhoBar_{n*+1}^{new}$,
by Equation~\eqref{eq:rhoBar}, and finally update the error locally
by 
\begin{equation}\label{eq:updateErrorInds}
\rBar_{n^*} = \rhoBar_{n^*} \left( \Delta t_{n^*}^{new}\right)^2,
\qquad \rBar_{n^*+1} 
= \rhoBar_{n^*+1}^{new} \left(\Delta t_{n^*+1}^{new}\right)^2.
\end{equation}
As a compromise between cost and accuracy, we here propose the
following mixed approach to updating error indicators post refinement:
With $\nRefine$ denoting the prescribed number of refinement 
iterations of the input mesh, let all error indicators be completely recomputed
every $\widetilde N = \bigO{\log(\nRefine)}$-th iteration, whereas
for the remaining $\nRefine-\widetilde N$ iterations, only
local updates of the error indicators are computed. Following this
approach, the computational cost of refining a mesh holding $N$ time
steps into a mesh of $2N$ time steps becomes $\bigO{N\log(N)^2}$.
Observe that the
asymptotically dominating cost is to sort the mesh's error indicators
$\bigO{\log(N)}$ times. To anticipate the computational cost for the
MSE adaptive MLMC algorithm, this implies that the cost of generating an
MSE adaptive realization pair is $\text{Cost}(\DlG) = \bigO{\ell^2 2^{\ell}}$.

\subsubsection{Pseudocode}\label{subsubsec:mseAdaptivePseudoCode}
The mesh refinement and the computation of 
error indicators are presented in Algorithms~\ref{alg:meshRefinement}
and~\ref{alg:computeErrorIndicators}, respectively.

\begin{algorithm}[H]
\caption{{\bf meshRefinement}}
\begin{algorithmic}\label{alg:meshRefinement}
  \STATE{\bf Input:} Mesh $\Delta t$, Wiener path $W$,
  number of refinements $\nRefine$, maximum time step $\DtMax$

  \STATE{\bf Output:} Refined mesh $\Delta t$ and Wiener path $W$.
  
  \medskip

    \STATE{Set the number of re-computations of all error indicators
      to a number $\widetilde N = \bigO{\log(\nRefine)}$ and compute the
      refinement batch size $\widehat N = \lceil \nRefine/ \widetilde N \rceil$.}
    
  \FOR{$i=1$ \TO $\widetilde N$ }
      
      \STATE{Completely update the error density by applying\\
	$
	[\rBar, \barX, \overline \varphi_x, \rhoBar ] = \textbf{computeErrorIndicators}(\Delta t, W ).
	$
      }

    \IF{ $\nRefine > 2 \widehat N$} 
    \STATE{Set the below for-loop limit to $J = \widehat N$.}
    \ELSE{}
    \STATE{Set $J = \nRefine$.}
    \ENDIF
    \FOR{$j=1$ \TO $J$ } 
	  \STATE{Locate the largest error indicator $\rBar_{n^*}$ using Equation~\eqref{eq:errorIndMax}.}
	  
	  \STATE{Refine the interval $(t_{n^*}, t_{n^*+1})$ by the halving~\eqref{eq:intervalHalving}, add a 
	    midpoint value $W_{n^*+1}^{new}$ to the Wiener path by the
            Brownian bridge interpolation~\eqref{eq:brownianBridge},
          and set $\nRefine = \nRefine -1$.}
		 
	  \STATE{Locally update the error indicators $r_{n^*}^{new}$ and $r_{n^*+1}^{new}$ by the steps~\eqref{eq:errorIndInterpolation} 
		and~\eqref{eq:updateErrorInds}.}
      \ENDFOR
      
  \ENDFOR
  
  \STATE Do a final sweep over the mesh and refine all time steps of the input mesh which
  are strictly larger than $\DtMax$.

  
\end{algorithmic}
\end{algorithm}

\begin{algorithm}[ht]
\caption{ {\bf computeErrorIndicators} }
\begin{algorithmic}\label{alg:computeErrorIndicators}
  \STATE{\bf Input:} mesh $\Delta t$, Wiener path $W$.
  \STATE{\bf Output:} error indicators $\rBar$, path solutions $\barX$ and $\overline \varphi_x$, error density $\rhoBar$.
  \STATE{Compute the SDE path $\barX$ using the Euler--Maruyama algorithm~\eqref{eq:eulerMaruyama}.}
  \STATE{Compute the first variation $\overline \varphi_x$ using the backward algorithm~\eqref{eq:backwardScheme1D}.}
  \STATE{Compute the error density $\rhoBar$ and error indicators $\rBar$ by the formulas~\eqref{eq:rhoBar} 
	  and~\eqref{eq:rBar}, respectively.}
\end{algorithmic}
\end{algorithm}

\subsection{Numerical Examples}
To illustrate the procedure for computing error indicators and the
performance of the adaptive algorithm, we now present four SDE example
problems. To keep matters relatively elementary, the dual solutions,
$\phiX{t}$, for these examples are derived not from \aposteriori but
\apriori analysis. This approach results in adaptively
generated mesh points which for all problems in this section will contain
mesh points which are stopping times for which $t_n \in \FCal_{t_{n-1}}$ 
for all $n\in \{1,2,\ldots,N\}$. In Examples~\ref{ex:mcGbm},~\ref{ex:example2}
and~\ref{ex:example3}, it is straightforward to verify that the other assumptions of the respective single- and multi-dimensional MSE error
expansions of Theorems~\ref{thm:mseExpansion1D} and~\ref{thm:mseExpansionMultiD} hold, 
meaning that the adaptive approach produces numerical solutions
whose MSE to leading order are bounded by
the respective error expansions~\eqref{eq:mseExpansion1DNotComp} 
and~\eqref{eq:mseExpansionMultiDExact}.


\begin{example}\label{ex:mcGbm}
We consider the classical geometric Brownian motion problem 
\[
dX_t = X_t dt + X_t dW_t, \quad X_0 =1,
\]
for which we seek to minimize the MSE
\begin{equation}\label{eq:mseGoalExamples}
\Exp{(X_T - \barX_T)^2 } = \min!,    \quad  N \text{ given},
\end{equation}
at the final time, $T=1$, (cf. the goal (B.1)). 
One may derive that the dual solution of this problem 
is of the form
\[
\phiX{X_t,t} = \partial_{X_t} X_T^{X_t,t} = \frac{X_T}{X_t},
\]
which leads to the error density
\[
\rho(t) = \frac{(b_x b)^2(X_t,t) \parenthesis{\phiX{X_t,t}}^2 }{2} = \frac{X_T^2}{2}.  
\]
We conclude that uniform time-stepping
is optimal. A further reduction of the MSE 
could be achieved by allowing the number of 
time steps to depend on the magnitude of $X_T^2$ for 
each realization. 
This is however outside the scope of the
considered refinement goal (B.1), where we assume
the number of time steps, $N$, is fixed for all realizations
and would be possible only to a very weak degree under 
the slight generalization of (B.1) given in assumption 
(M.2) of Theorem~\ref{thm:mseExpansion1D}. 
\end{example}

\begin{example}\label{ex:example2}
Our second example is the two-dimensional (2D) SDE problem 
\begin{align*}
dW_t  & = 1 dW_t, & W_0 &= 0,\\
dX_t   & =  W_t dW_t, & X_0 &=0.
\end{align*}
Here, we seek to minimize the MSE $\Exp{(X_T - \barX_T)^2 }$
for the observable 
\[
X_T = \int_0^T W_t dW_t
\]
at the final time $T=1$. With the diffusion matrix represented by
\[
b( (W_t, X_t) ,t) = \begin{bmatrix} 1\\ W_t\end{bmatrix},
\]
and observing that 
\[
\partial_{X_t} X_T^{X_t,t} =  \partial_{X_t} \left( X_t +
\int_t^T W_s dW_s \right) = 1,
\]
it follows from the error density in multi-dimensions in
Equation~\eqref{eq:rhoBarHigherDim} that $\rho(t) = \frac{1}{2}$.  We
conclude that uniform time-stepping is optimal for this problem as
well.
\end{example}

\begin{example}\label{ex:example3}
Next, we consider the three-dimensional (3D) SDE problem
\begin{align*}
dW_t^{(1)}  & = 1 dW_t^{(1)}, & W_0^{(1)}& = 0,\\
dW_t^{(2)}  & = 1 dW_t^{(2)}, & W_0^{(2)} & = 0,\\
dX_t   & =  W_t^{(1)} dW_t^{(2)} - W_t^{(2)} dW_t^{(1)}, & X_0 &=0,
\end{align*}
where $W_t^{(1)}$ and $W_t^{(2)}$ are independent Wiener processes.
Here, we seek to minimize the MSE $\Exp{(X_T - \barX_T)^2 }$
for the Levy area observable 
\[
X_T =\int_0^T (W_t^{(1)} dW_t^{(2)} - W_t^{(2)} dW_t^{(1)}),
\]
at the final time, $T=1$. Representing the diffusion matrix by
\[
b( (W_t, X_t) ,t) = \begin{bmatrix} 1& 0 \\ 0& 1\\ -W_t^{(1)} & W_t^{(2)} \end{bmatrix},
\]
and observing that 
\[
\partial_{X_t} X_T^{X_t,t} =  \partial_{X_t} \left( X_t + \int_t^T (W_s^{(1)} dW_s^{(2)} - W_s^{(2)} dW_s^{(1)}), \right) = 1,
\]
it follows from Equation~\eqref{eq:rhoBarHigherDim} that $\rho(t) = 1$.  We
conclude that uniform time-stepping is optimal for computing Levy
areas.
\end{example}

\begin{example}\label{ex:wTSquared}
As the last example, we consider the 2D SDE
\begin{align*}
dW_t  & = 1 dW_t, & W_0 &= 0,\\
dX_t   & =  3(W_t^2 - t) dW_t,  &\quad X_0 &=0.
\end{align*}
We seek to minimize the MSE~\eqref{eq:mseGoalExamples}
at the final time $T=1$. For this problem, it may be shown by
\Ito calculus that the pathwise exact solution is $X_T = W_T^3 -
3W_TT$.  Representing the diffusion matrix by
\[
b( (W_t, X_t) ,t) = \begin{bmatrix} 1\\ 3(W_t^2-t)\end{bmatrix},
\]
equation~\eqref{eq:rhoBarHigherDim} implies that $\rho(t) = 18W_t^2$.
This motivates the use of discrete error indicators, $\rBar_n = 18 W_{t_n}^2
\Delta t_n^2$, in the mesh refinement criterion.
For this problem, we may not directly conclude that the error 
expansion~\eqref{eq:mseExpansionMultiDExact} holds 
since the diffusion coefficient does not fulfill the assumption 
in Theorem~\ref{thm:mseExpansionMultiD}. Although we will not include 
the details here, it is easy to derive that
$\partial_{x}^j X_T^{x,t} = 0$ for all $j>1$ and to
prove that the MSE leading-order error expansion also holds for this 
particular problem by following the steps of the 
proof of Theorem~\ref{thm:mseExpansion1D}.
In Figure~\ref{fig:wTSquaredErrorVsTimesteps}, we compare the uniform and
adaptive time-stepping Euler--Maruyama algorithms in terms of MSE vs.
the number of time steps, $N$. Estimates for the MSE for both algorithms
are computed by MC sampling using $M=10^6$ samples. This is a sufficient
sample size to render the MC estimates' statistical error
negligible.
For the adaptive algorithm, we have used the following input parameter
in Algorithm~\ref{alg:meshRefinement}: uniform input mesh, $\Delta t$, 
with step size $2/N$ (and $\DtMax = 2/N$). The number of refinements is
set to $\nRefine =N/2$. We observe
that the algorithms have approximately equal convergence
rates, but, as expected, the adaptive algorithm is slightly
more accurate than the uniform time-stepping algorithm.
\begin{figure}[H]
    \centering
    \includegraphics[width=0.75\textwidth]{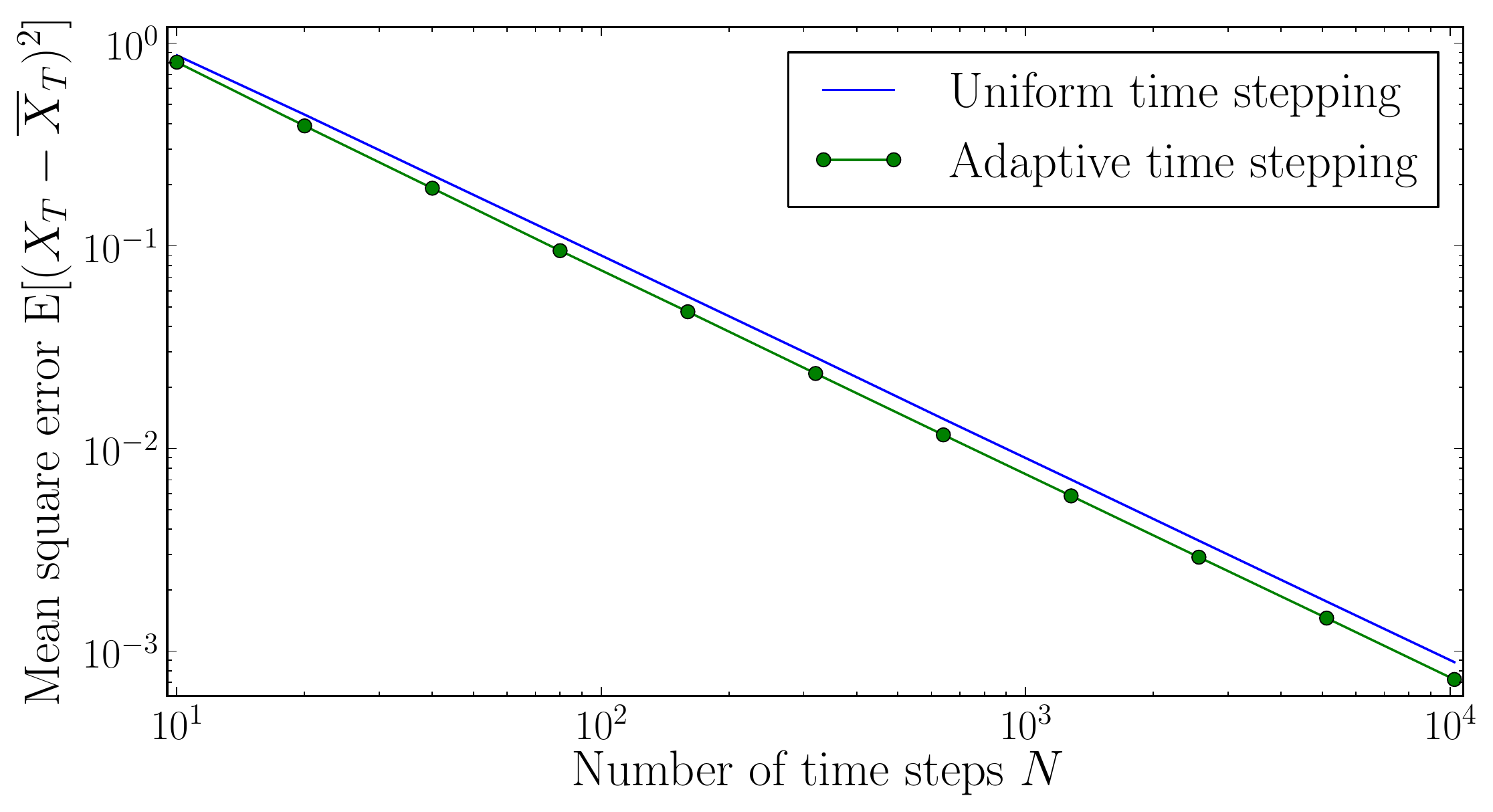}
    \caption{Comparison of the performance of uniform and adaptive
      time-stepping Euler--Maruyama numerical integration for
      Example~\ref{ex:wTSquared} in terms of MSE vs. number of time
      steps.}
    \label{fig:wTSquaredErrorVsTimesteps}
\end{figure}

\end{example}





\section{Extension of the Adaptive Algorithm to the Multilevel Setting}
\label{sec:adaptiveMLMC}
In this section, we incorporate the MSE adaptive time-stepping
algorithm presented in the preceding section into an MSE adaptive MLMC
algorithm for weak approximations.  First, we shortly recall the
approximation goal and important concepts for the MSE adaptive MLMC
algorithm, such as the structure of the adaptive mesh hierarchy and MLMC
error control. Thereafter, the MLMC algorithm is presented in
pseudocode form.

\subsection{Notation and Objective} \label{subsec:mlmcNotation2} For a
tolerance, $\tol>0$, and confidence, $0<1-\delta<1$, we recall that our
objective is to construct an adaptive time-stepping MLMC estimator,
$\AMLSimple$, which meets the approximation constraint
\begin{equation}\label{eq:objectiveRecall}
\Prob{\abs{ \Exp{ \gXT } - \AMLSimple} \leq \tol} \geq 1-\delta.
\end{equation}
We denote the multilevel estimator by
\begin{equation*} \label{eq:MLestimatorRecall}
\begin{split}
\AMLSimple:= \sum_{\ell=0}^L \underbrace{\sum_{i=1}^{M_\ell} \frac{ \DlGO  }{M_\ell}}_{=:\A{ \Delta_\ell g; M_\ell}},
\end{split}
\end{equation*}
where
\[
 \Delta_\ell g(\omega) := 
 \begin{cases} 
    \g{\barXL[0](\omega) }, &\text{if} \quad \ell = 0,\\
    \g{\barXL[\ell](\omega) } - \g{\barXL[\ell-1](\omega) }, & \text{else.}
 \end{cases}
\]
Section~\ref{subsubsec:adaptiveMlmcMesh} presents further details on
MLMC notation and parameters.

\subsubsection{The Mesh Hierarchy}\label{subsec:meshConstructionMLMC2}

A realization, $\DlGO$, is generated on a nested pair of mesh
realizations
\[
\ldots \subset \Dt{\ell-1}(\omega_{i,\ell}) \subset \Dt{\ell}(\omega_{i,\ell}).
\]
Subsequently, mesh realizations are generated step by step from 
a prescribed and deterministic input mesh, $\Dt{-1}$, holding $N_{-1}$ uniform
time steps. First, $\Dt{-1}$ is refined into a mesh, $\Dt{0}$, by
applying
Algorithm~\ref{alg:meshRefinement}, namely
\[
[\Dt{0}, \WL{0}] = \textbf{meshRefinement}\left(\Dt{-1}, \WL{-1}, \nRefine = N_{-1}, \DtMax= N_{0}^{-1} \right).
\]
The mesh refinement process is iterated until meshes $\Dt{\ell-1}$ and $\Dt{\ell-1}$
are produced, with the last couple of iterations being
\[
[\Dt{\ell-1}, \WL{\ell-1} ] =
\textbf{meshRefinement}\left(\Dt{\ell-2},\WL{\ell-2}, \nRefine = N_{\ell-2} , \DtMax= N_{\ell-1}^{-1} \right),
\]
and 
\[
[\Dt{\ell}, \WL{\ell}] = \textbf{meshRefinement}\left(\Dt{\ell-1},
\WL{\ell-1} , \nRefine = N_{\ell-1}, \DtMax= N_{\ell}^{-1}  \right).
\]
The output realization for the difference
$\DlGO = \g{\barXL(\omega_{i,\ell})} - \g{ \barXL[\ell-1](\omega_{i,\ell}) }$ is 
thereafter generated on the output temporal mesh and Wiener path pairs,
$(\Dt{\ell-1}, \WL{\ell-1})$ and $(\Dt{\ell}, \WL{\ell})$.

For later estimates of the computational cost of the MSE adaptive MLMC
algorithm, it is useful to have upper bounds on the growth of the number of time
steps in the mesh hierarchy, $\{\Dt{\ell}\}_\ell$, as $\ell$
increases. Letting $\abs{\Delta t}$ denote the number of time steps in
a mesh, $\Delta t$ (i.e., the
cardinality of the set $\Delta t = \{\Delta t_0, \Delta t_1, \ldots
\}$), the following bounds hold 
\begin{equation*}\label{eq:nLBounds}
 N_\ell \leq \abs{\Dt{\ell}} < 2 N_\ell \qquad \forall \ell \in \N_0.
\end{equation*}
The lower bound follows straightforwardly from the mesh hierarchy
refinement procedure described above. To show the upper bound, 
notice the maximum
number of mesh refinements going from a level $\ell-1$ mesh,
$\Dt{\ell-1}$ to a level $\ell$ mesh, $\Dt{\ell}$ is $2N_{\ell-1} -1$.
Consequently,
\[
\begin{split}
|\Dt{\ell}| &\leq |\Dt{-1} | + \sum_{j=0}^{\ell-1} \text{Maximum number of
  refinements going from $\Dt{j-1}$ to $\Dt{j}$ } \\
& \leq N_{-1} +  2 \sum_{j=0}^{\ell} N_{j-1} - (\ell+1) < 2 N_\ell.
\end{split}
\]


\begin{remark}
  For the telescoping property $\Exp{\AMLSimple}=\Exp{\g{\barXL}}$ to
  hold, it is not required that the adaptive mesh hierarchy is nested,
  but non-nested meshes make it more complicated to compute Wiener
  path pairs $(\WL{\ell-1},\WL{\ell})(\omega)$. In the numerical tests
  leading to this work, we tested both nested and non-nested adaptive
  meshes and found both options performing satisfactorily.
\end{remark}


\subsection{Error Control}\label{subsec:errorControl}
The error control for the adaptive MLMC algorithm follows the general framework 
of a uniform time-stepping MLMC, but for the sake of completeness, we recall  
the error control framework for the setting of weak approximations.
By splitting
\[
\begin{split}
 \abs{\Exp{\gXT} - \AMLSimple}
 \leq   \underbrace{\abs{ \Exp{\gXT - \g{ \barXL[L] }}}}_{=:\mathcal{E}_T}
+\underbrace{\abs{\Exp{ \g{ \barXL[L] } - \AMLSimple}}}_{=:\mathcal{E}_S}
\end{split}
\]
and 
\begin{equation}\label{eq:toleranceSplit}
\tol = \tolT + \tolS,
\end{equation}
we seek to implicitly fulfill~\eqref{eq:objectiveRecall} by imposing
the stricter constraints
\begin{align}
\mathcal{E}_T &\leq \tolT,  & &\text{the time discretization error},\\
P\parenthesis{\mathcal{E}_S \leq \tolS} & \geq 1-\delta, & &\text{the statistical error.}
\end{align}

\subsubsection{The Statistical Error}\label{subsec:statError}

Under the moment assumptions stated in~\cite{Durrett96}, 
Lindeberg's version of the Central Limit Theorem yields that as $\tol \downarrow 0$,
\begin{equation*}\label{eq:lindebergCltConv}
\frac{\AMLSimple- \Exp{\g{\barXL[L] } }}{\sqrt{\Var{\AMLSimple }}} \xrightarrow[]{D} N(0,1).
\end{equation*}
Here, $\xrightarrow[]{D}$ denotes convergence in distribution. By
construction, we have 
\[
 \Var{\AMLSimple} =  \sum_{\ell=0}^{L} \frac{\Var{\DlG}}{M_\ell}.
\]
This asymptotic result motivates the statistical error constraint
\begin{equation}\label{eq:statError2}
\Var{ \AMLSimple} \leq \frac{\tolS^2}{\Cc^2},
\end{equation}
where $\Cc(\delta)$ is the confidence parameter chosen such that  
\begin{equation}\label{eq:CcDefinition}
 1- \frac{1}{\sqrt{2 \pi}} \int_{-\Cc}^\Cc e^{-x^2/2} \, dx =  (1-\delta),
\end{equation}
for a prescribed confidence $(1-\delta)$. 

Another important question is how to distribute the number of samples,
$\{M_\ell\}_\ell$, on the level hierarchy such that both the
computational cost of the MLMC estimator is minimized and the
constraint~\eqref{eq:statError2} is met. Letting $C_\ell$ denote the
expected cost of generating a numerical realization $\DlGO$, the
approximate total cost of generating the multilevel estimator becomes
\[
\CML := \sum_{\ell=0}^L C_\ell M_\ell.
\]
An optimization of the number of samples at each level can then be
found through minimization of the Lagrangian
\[
 \mathcal{L}(M_0, M_1, \ldots, M_L, \lambda) 
= \lambda \left( \sum_{\ell=0}^L \frac{\Var{\DlG}}{M_\ell} - \frac{\tolS^2}{\Cc^2} \right)
+ \sum_{\ell=0}^L C_\ell M_\ell,
\]
yielding
\[
 M_\ell = \roundUp{ \frac{\Cc^2}{\tolS^2}
   \sqrt{\frac{\Var{\DlG}}{C_\ell}} 
\sum_{\ell=0}^L \sqrt{C_\ell \Var{\DlG } }  }, \quad \ell = 0, 1, \ldots, L.
\]
Since the cost of adaptively refining a mesh, $\Dt{\ell}$, is
$\bigO{N_\ell \log(N_\ell)^2}$, as noted in
Section~\ref{subsubsec:errorIndicators}, the cost of generating an SDE
realization, is of the same order: $C_\ell = \bigO{N_\ell
  \log(N_\ell)^2}$. Representing the cost by its leading-order term and
  disregarding the logarithmic factor, an approximation to the 
  level-wise optimal number of samples becomes
\begin{equation}\label{eq:determineMl}
 M_\ell = \roundUp{ \frac{\Cc^2}{\tolS^2} \sqrt{\frac{\Var{\DlG}}{N_\ell}} 
 \sum_{\ell=0}^L \sqrt{N_\ell \Var{\DlG } }  }, \quad \ell = 0, 1, \ldots, L.
\end{equation}
\begin{remark}\label{rem:sampleVarDef}
In our MLMC implementations, the variances, $\Var{\DlG}$, 
in equation~\eqref{eq:determineMl} are approximated by sample variances.
To save memory in our parallel computer implementation, 
the maximum permitted batch size for a set of realizations, $\{\DlGO \}_i$, is 
set to 100,000.  
For the initial batch consisting of $M_\ell = \widehat M$ samples,
the sample variance is computed by the standard approach,
\begin{equation*}\label{eq:sampleVarDef}
 \V{\DlG ; M_\ell} = \frac{1}{M_\ell-1} \sum_{i=1}^{M_\ell}
  (\DlGO - \A{\DlG; M_\ell} )^2.
\end{equation*}
Thereafter, for every new batch of realizations, 
$\{\DlGO \}_{i=M_{\ell}+1}^{M_{\ell}+M}$ 
($M$ here denotes an arbitrary natural number smaller or equal to 100,000),
we incrementally update the sample variance,
\[
\begin{split}
\V{\DlG ; M_\ell +M } & =  \frac{ M_\ell  }{ M_\ell+M } \times \V{\DlG ; M_\ell}
\\
&\quad + \frac{1}{(M_\ell+M-1)}  \sum_{i=M_\ell +1}^{M_\ell + M}
  (\DlGO - \A{\DlG; M_\ell + M} )^2, 
   \end{split}
\]
and update the total number of samples on level $\ell$ accordingly,
$M_\ell = M_\ell + M$.
\end{remark}

\subsubsection{The Time Discretization Error}

To control the time discretization error, we assume that a weak order 
convergence rate, $\alpha>0$, holds for the given SDE problem when
solved with the Euler--Maruyama method, i.e.,
\[
 \abs{ \Exp{\gXT - \g{ \barXL[L] }}} =  \bigO{N_L^{-\alpha}},
\]
and we assume that the asymptotic rate is reached at level
$L-1$. Then
\[
\abs{ \Exp{\gXT - \g{ \barXL[L] }}} = \abs{ \sum_{\ell=L+1}^\infty \Exp{\DlG}} \leq \abs{\Exp{\Delta_Lg}} \sum_{\ell=1}^\infty
2^{-\alpha\ell} 
 = \frac{\abs{\Exp{\Delta_Lg}}}{2^\alpha -1}.
\]
In our implementation, we assume the weak
convergence rate, $\alpha$, is known prior to sampling and,
replacing $\Exp{\Delta_Lg}$ with a sample average approximation in the 
above inequality, we determine $L$ by the following stopping criterion:
\begin{equation}\label{eq:determineL}
\frac{ \max \left( 2^{-\alpha} \abs{ \A{ \Delta_{L-1} g; M_{L-1}}} ,  \abs{\A{
        \Delta_{L} g; M_{L}} }  \right) }{2^\alpha -1}  \leq \tolT ,
\end{equation}
(cf.~Algorithm~\ref{alg:mlmcEstimator}). Here we implicitly assume that the
statistical error in estimating the bias condition is not prohibitively large.

A final level $L$ of order $\log(\tolT^{-1})$ will thus control the discretization error.

\subsubsection{Computational Cost}
Under the convergence rate assumptions stated
in Theorem~\ref{thm:cliffeComplexity}, it follows that the cost of generating an
adaptive MLMC estimator, $\AMLSimple$, fulfilling the MSE approximation goal
$\Exp{(\AMLSimple - \Exp{\gXT})^2 } \leq \tol^2$ is bounded by
\begin{equation}\label{eq:cmlBound}
\CML = \sum_{\ell=0}^L M_\ell C_\ell \leq 
\begin{cases} 
\bigO{\tol^{-2}}, & \text{if } \beta> 1,\\
 \bigO{\tol^{-2}\log(\tol)^4}, & \text{if }  \beta = 1,\\
  \bigO{\tol^{-2 + \frac{\beta-1}{\alpha}} \, \log(\tol)^2 } , &\text{if }  \beta < 1.
 \end{cases}
\end{equation}
Moreover, under the additional higher moment approximation rate assumption 
\[
\Exp{\abs{\g{ \barXL} - \gXT }^{2+\nu} } = \bigO{2^{-\beta+\nu/2}}, 
\]
the complexity bound~\eqref{eq:cmlBound} also holds for
fulfilling criterion~\eqref{eq:mlmcGoal} asymptotically as $\tol
\downarrow 0$, (cf.~\cite{Collier14}).

\subsection{MLMC Pseudocode} \label{subsec:mlmcPseudoCode} In this
section, we present pseudocode for the implementation of the MSE
adaptive MLMC algorithm. In addition to
Algorithms~\ref{alg:meshRefinement}
and~\ref{alg:computeErrorIndicators}, presented in
Section~\ref{subsubsec:mseAdaptivePseudoCode}, the implementation consists of
Algorithms~\ref{alg:mlmcEstimator} and~\ref{alg:adaptiveRealization}.
Algorithm~\ref{alg:mlmcEstimator} describes how the stopping criterion
for the final level $L$ is implemented and how the multilevel
estimator is generated, and Algorithm~\ref{alg:adaptiveRealization}
describes the steps for generating a realization $\Delta_\ell g$.


\begin{algorithm}[!ht]
\caption{{\bf mlmcEstimator}}
\begin{algorithmic}\label{alg:mlmcEstimator}
  \STATE{\bf Input:} $\tolT$, $\tolS$, confidence $\delta$, initial mesh
  $\Dt{-1}$, initial number of mesh steps $N_{-1}$,
  input weak rate $\alpha$, initial number of samples $\widehat M$.
  \STATE{\bf Output:} Multilevel estimator $\AMLSimple$.
    
\medskip 

\STATE{Compute the confidence parameter $\Cc(\delta)$ by~\eqref{eq:CcDefinition}.}

\STATE{Set $L=-1$.}

\WHILE{$L<2$ \OR ~\eqref{eq:determineL}, using the input $\alpha$ for the weak rate, is violated}

  \STATE{Set $L=L+1$.}
  
  \STATE {Set $M_L = \widehat M$, generate a set of realizations $\{ \DlGO \}_{i=1}^{M_L}$
  by applying 
  $
    \textbf{adaptiveRealizations}(\Dt{-1}).
  $
   }
  \FOR{$\ell=0$ \TO $L$}
    \STATE{Compute the sample variance $\V{ \DlG ; M_l}$.}
  \ENDFOR

  \FOR{$\ell=0$ \TO $L$}
    \STATE{Determine the number of samples $M_\ell$ by~\eqref{eq:determineMl}.}
      \IF{new value of $M_\ell$ is larger than the old value}{
	\STATE{Compute additional realizations $\{ \DlGO \}_{i=M_\ell+1}^{M_\ell^{new}}$ by applying 
$
	    \textbf{adaptiveRealizations}(\Dt{-1}).
$
	}}
      \ENDIF
  \ENDFOR
 
\ENDWHILE

Compute $\AMLSimple$ from the generated samples by using formula~\eqref{eq:MLestimator}.

\end{algorithmic}
\end{algorithm}

\begin{algorithm}[!h]
\caption{{\bf adaptiveRealization}}
\begin{algorithmic}\label{alg:adaptiveRealization}
  \STATE{\bf Input:} Mesh $\Dt{-1}$.
  \STATE{\bf Outputs:} One realization $\DlG(\omega)$
    
\medskip 

\STATE{Generate a Wiener path $\WL{-1}$ on the initial mesh $\Dt{-1}$.}

\FOR{$j=0$ \TO $\ell$ }

\STATE{Refine the mesh by applying 
\[
 [\Dt{j}, \WL{j} ] = \textbf{meshRefinement}(\Dt{j-1},\WL{j-1},
 \nRefine = N_{j-1}, \DtMax = N_{j}^{-1} ).
\]
}
\ENDFOR
\STATE{Compute Euler--Maruyama realizations $(\barXL[\ell-1],
  \barXL[\ell])(\omega)$ using the mesh pair $(\Dt{\ell-1},
  \Dt{\ell})(\omega)$ and Wiener path pair
  $(\WL{\ell-1},\WL{\ell})(\omega)$, cf.~\eqref{eq:eulerMaruyama}, and
  return the output
\[
\DlG(\omega) = \g{\barXL(\omega)} - \g{\barXL[\ell-1](\omega)}.
\]
}

\end{algorithmic}
\end{algorithm}

\begin{remark}
  For each increment of $L$ in Algorithm~\ref{alg:mlmcEstimator}, 
  all realizations $\Delta_\ell g$ that have been generated up to that point are reused in
  later computations of the multilevel estimator. 
  This approach, which is common in MLMC,
  (cf.~\cite{Giles08}), seems to work fine in practice although
  the independence between samples is then lost. Accounting for
  the lack of independence complicates the convergence analysis.
\end{remark}

\section{Numerical Examples for the MLMC Algorithms} \label{sec:MlmcExamples} To
illustrate the implementation of the MSE adaptive MLMC algorithm and to show
its robustness and potential efficiency gain over the uniform MLMC algorithm, we present
two
numerical examples in this section. 
The first example
considers a geometric Brownian motion SDE problem with sufficient
regularity, such that there is very little (probably nothing) to gain by
introducing adaptive mesh refinement. The example is
included to show that in settings where adaptivity is not required,
the MSE adaptive MLMC algorithm is not excessively more expensive than
the uniform MLMC algorithm. In the second example, we consider an SDE
with a random time drift coefficient blow-up of order $t^{-p}$ with $p
\in [0.5, 1)$. The MSE adaptive MLMC algorithm performs
progressively more efficiently than does the uniform MLMC algorithm
as the value of the blow-up exponent $p$ increases.
We should add, however, that although we observe numerical evidence for the numerical solutions
converging for both examples, all of the assumptions 
in Theorem~\ref{thm:mseExpansion1D} are not fulfilled for our adaptive algorithm,
when applied to
either of the two examples. We are therefore not able to prove theoretically that our adaptive algorithm
converges in these examples.

For reference, the implemented MSE adaptive MLMC algorithm is described in
Algorithms~\ref{alg:meshRefinement}--\ref{alg:adaptiveRealization}, the standard form 
of the uniform time-stepping MLMC algorithm that we 
use in these numerical comparisons is presented in
Algorithm~\ref{alg:uniformMlmcEstimator}, Appendix~\ref{sec:uniformMlmcAlg}, 
and a summary of the parameter values used in the examples is given 
in Table~\ref{tab:parameters}. Furthermore, all average properties 
derived from the MLMC algorithms that we plot for the considered examples in 
Figures~\ref{fig:mlmcGbmAvgNl}--\ref{fig:pSingularityError} below 
are computed from 100 multilevel estimator realizations,
and, when plotted, error bars are scaled to one sample standard deviation.

\begin{table}[h!]
  \centering
  \caption{List of parameter values used by the MSE adaptive MLMC
      algorithm and (when required) the uniform MLMC algorithm for the
      numerical examples in Section~\ref{sec:MlmcExamples}.
      }
  \label{tab:parameters}
  \begin{tabular}{p{1.4cm}p{5cm}rrp{2.2cm} }
    \hline \noalign{\smallskip}
    Parameter & Description of parameter& Example~\ref{ex:mlmcGbm} & Example~\ref{ex:mlmcSingularityGbm} \\
    \noalign{\smallskip}\hline\noalign{\smallskip}
    $\delta$ & Confidence parameter, cf.~\eqref{eq:objectiveRecall}. & 0.1 & 0.1\\
    $\tol$ & Accuracy parameter, cf.~\eqref{eq:objectiveRecall}. & $[10^{-3},
    10^{-1}]$ & $[10^{-3}, 10^{-1}]$\\
    $\tolS$& Statistical error tolerance, cf.~\eqref{eq:toleranceSplit}. & \tol/2 &
    \tol/2 \\
    $\tolT$ & Bias error tolerance, cf.~\eqref{eq:toleranceSplit}. & \tol/2 & \tol/2 \\
    $\Dt{-1}$ & Pre-initial input uniform mesh having the following
    step size. & 1/2 & 1/2 \\
    $N_{0}$  & Number of time steps in the initial mesh $\Dt{0}$. & 4
    & 4\\
    $\tilde N(\ell)$ & The number of complete updates of the error
    indicators in the MSE adaptive algorithm,
    cf. Algorithm~\ref{alg:meshRefinement}. & 
    $\Big \lfloor \frac{\log(\ell+2)}{\log(2)} \Big \rfloor $ 
    & $\Big \lfloor \frac{\log(\ell+2)}{\log(2)} \Big \rfloor $\\
    $\DtMax(\ell)$ & Maximum permitted time step size. &
    $N_\ell^{-1}$ & $N_\ell^{-1}$\\
    $\Delta t_{\text{min}}$ & Minimum permitted time step size (due to
    the used double-precision binary floating-point format). &
    $ 2^{-51}$ & $2^{-51}$\\
    $\widehat M$ & Number of first batch samples for a (first) estimate
    of the variance $\Var{\DlG}$. & 100 & 20\\
    $\alpha_U$ & Input weak convergence rate used in the stopping
    rule~\eqref{eq:determineL} for uniform time step Euler--Maruyama numerical integration.& 1 & $(1-p)$\\
    $\alpha_A$ & Input weak convergence rate used in the stopping
    rule~\eqref{eq:determineL} for the MSE adaptive time step Euler--Maruyama
    numerical integration.
    & 1 & 1\\
    \noalign{\smallskip}\hline\noalign{\smallskip}
        \end{tabular}
\end{table}

\begin{example}\label{ex:mlmcGbm}
We consider the geometric Brownian motion 
\begin{equation*}\label{eq:gbmProblem2}
dX_t = X_t dt + X_t dW_t, \quad X_0 =1,
\end{equation*}
where we seek to fulfill the weak approximation
goal~\eqref{eq:mlmcGoal} for the observable, $g(x)=x$, at the final time,
$T=1$. The reference solution is $\Exp{\g{X_T}} = e^T$. From
Example~\ref{ex:mcGbm}, we recall that the MSE minimized in this
problem by using uniform time steps. However, 
our \aposteriori MSE adaptive MLMC algorithm
computes error indicators from numerical solutions of the
path and the dual solution, which may lead 
to slightly non-uniform output meshes. In Figure~\ref{fig:mlmcGbmAvgNl},
we study how close to uniform the MSE adaptive meshes are by plotting the level-wise 
ratio, $\Exp{\abs{\Dt{\ell}}}/N_\ell$, where we recall that $\abs{\Dt{\ell}}$
denotes the number of time steps in the mesh, $\Dt{\ell}$, and that a 
uniform mesh on level $\ell$ has $N_\ell$ time steps.
As the level, $\ell$, increases, $\Exp{\abs{\Dt{\ell}}}/N_\ell$ converges
to $1$, and to interpret this result, we recall from the construction of the adaptive mesh 
hierarchy in Section~\ref{sec:adaptiveMLMC} that if $\abs{\Dt{\ell}}=N_\ell$, 
then the mesh, $\Dt{\ell}$, is uniform. We thus conclude that for this 
problem, the higher the level, the more uniform the 
MSE adaptive mesh realizations generally become.
\begin{figure}[h]
    \centering
    \includegraphics[width=0.9\textwidth]{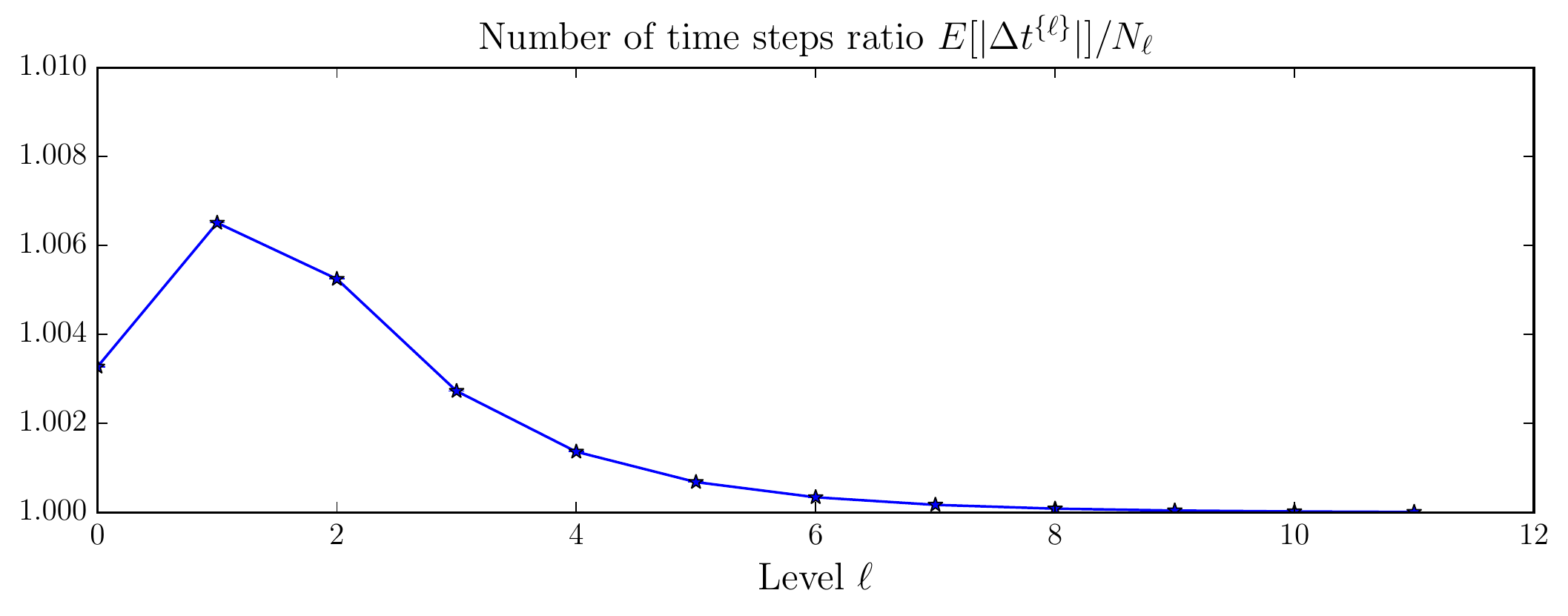}
    \caption{
    The ratio of the level-wise mean number of time steps $\Exp{\abs{\Dt{\ell}}}/N_\ell$, 
    of MSE adaptive mesh realizations to uniform mesh realizations for Example~\ref{ex:mlmcSingularityGbm}.
    }
    \label{fig:mlmcGbmAvgNl}
\end{figure}

Since adaptive mesh refinement is costly and since this problem has
sufficient regularity for the first-order weak and MSE convergence
rates~\eqref{eq:weakRate} and~\eqref{eq:strongRate} to hold,
respectively, one might expect that MSE adaptive MLMC 
will be less efficient than the uniform MLMC. This is verified in
Figure~\ref{fig:gbmComplexity}, which shows that the
runtime of the MSE adaptive MLMC algorithm grows
slightly faster than the uniform MLMC algorithm and that the cost
ratio is at most roughly $3.5$, in favor of uniform MLMC.
In Figure~\ref{fig:gbmAccuracy}, the accuracy of the
MLMC algorithms is compared, showing that both algorithms fulfill the
goal~\eqref{eq:mlmcGoal} reliably. 
Figure~\ref{fig:mlmcGbmLevelComparisons} 
further shows that both
algorithms have roughly first-order convergence rates for the weak error
$\abs{\Exp{\DlG}}$ and the variance $\Var{\DlG}$, and that the decay
rates for $M_l$ are close to identical.  We conclude that although MSE
adaptive MLMC is slightly more costly than uniform MLMC, the algorithms 
perform comparably in terms of runtime for 
this example.
\begin{figure}[h]
    \centering
    \includegraphics[width=0.9\textwidth]{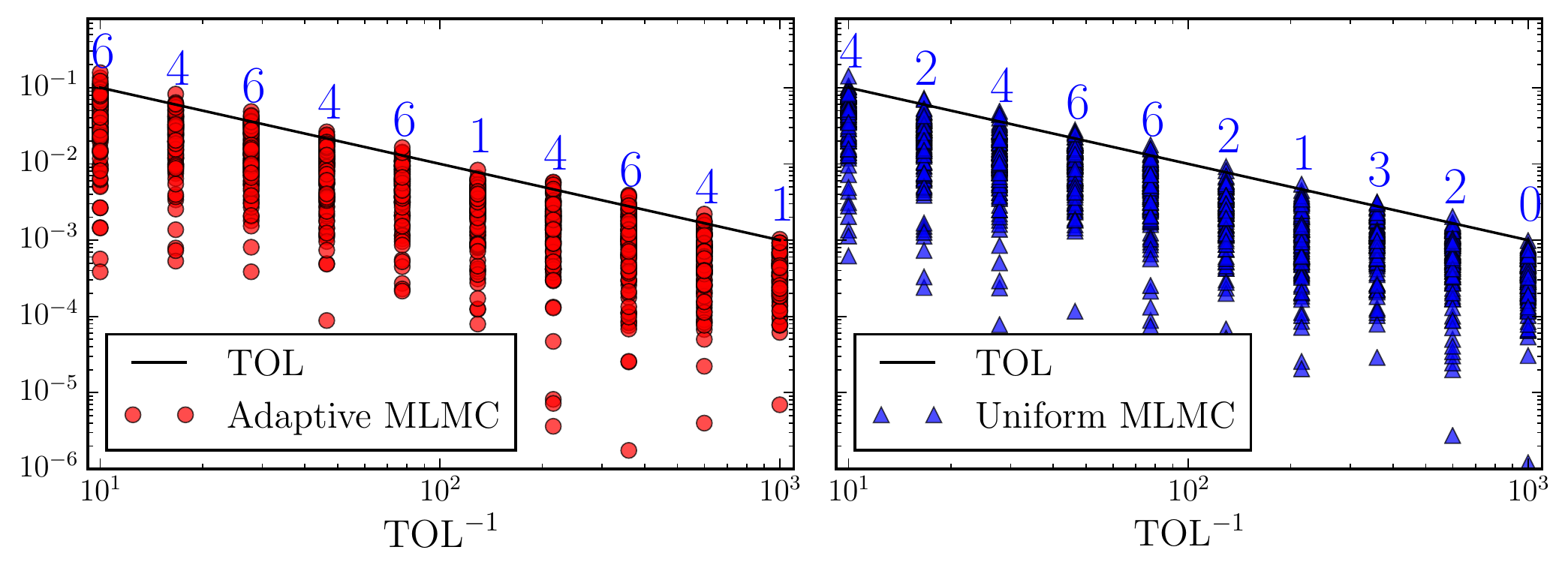}
     \caption{For a set of $\tol$ values, 100 realizations of the MSE adaptive multilevel
       estimator are computed using both MLMC algorithms for Example~\ref{ex:mlmcGbm}. 
       The errors $|\AMLSimple(\omega_i; \tol, \delta) - \Exp{\g{X_T}}|$ 
       are respectively plotted as circles (adaptive MLMC) and 
       triangles (uniform MLMC), and the
       number of multilevel estimator realizations failing the
       constraint $|\AMLSimple(\omega_i; \tol, \delta) -
       \Exp{\g{X_T}}|<\tol$ is written above the $(\tol^{-1},\tol)$
       line. Since the confidence parameter is set to $\delta = 0.1$
       and less than 10 realizations fail for any of the tested $\tol$
       values, both algorithms meet the approximation
       goal~\eqref{eq:objectiveRecall}.}
       \label{fig:gbmAccuracy}
\end{figure}

\begin{figure}[h]
    \centering
     \includegraphics[width=.87\textwidth]{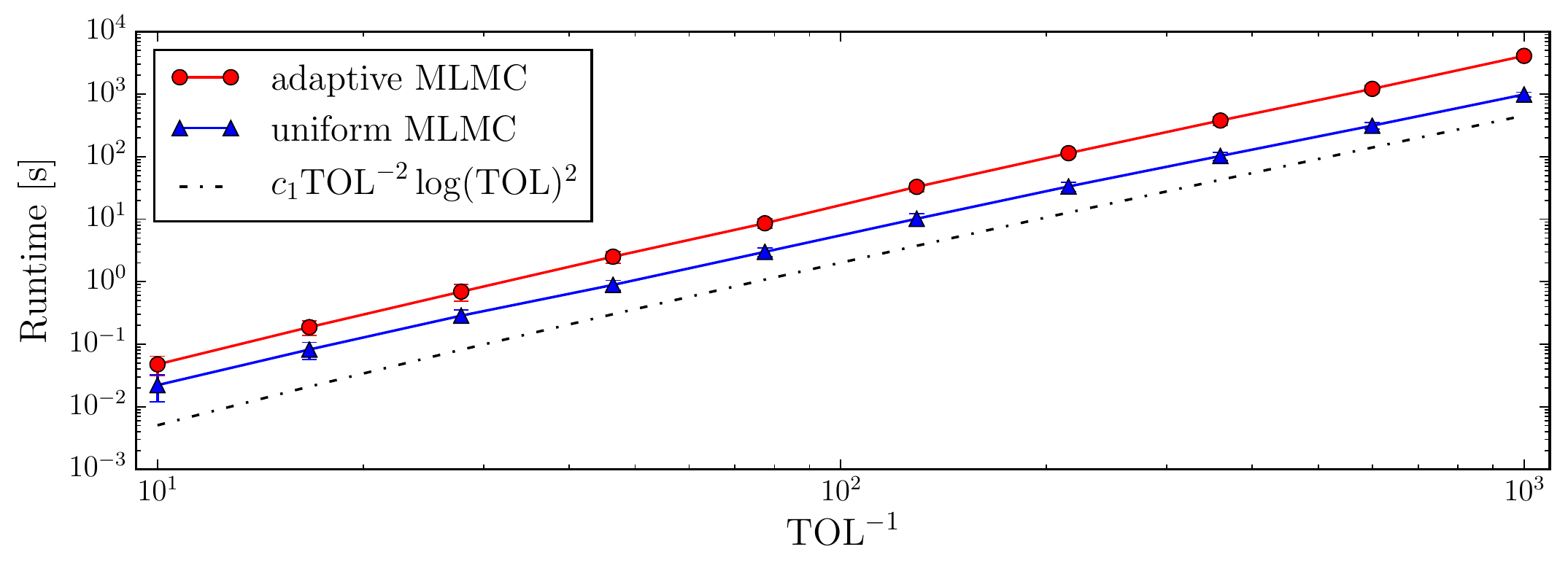}
     \caption{Average runtime vs. $\tol^{-1}$ for the two MLMC
       algorithms solving Example~\ref{ex:mlmcGbm}.}
    \label{fig:gbmComplexity}
\end{figure}

\begin{figure}[h]
    \centering
    \includegraphics[width=0.9\textwidth]{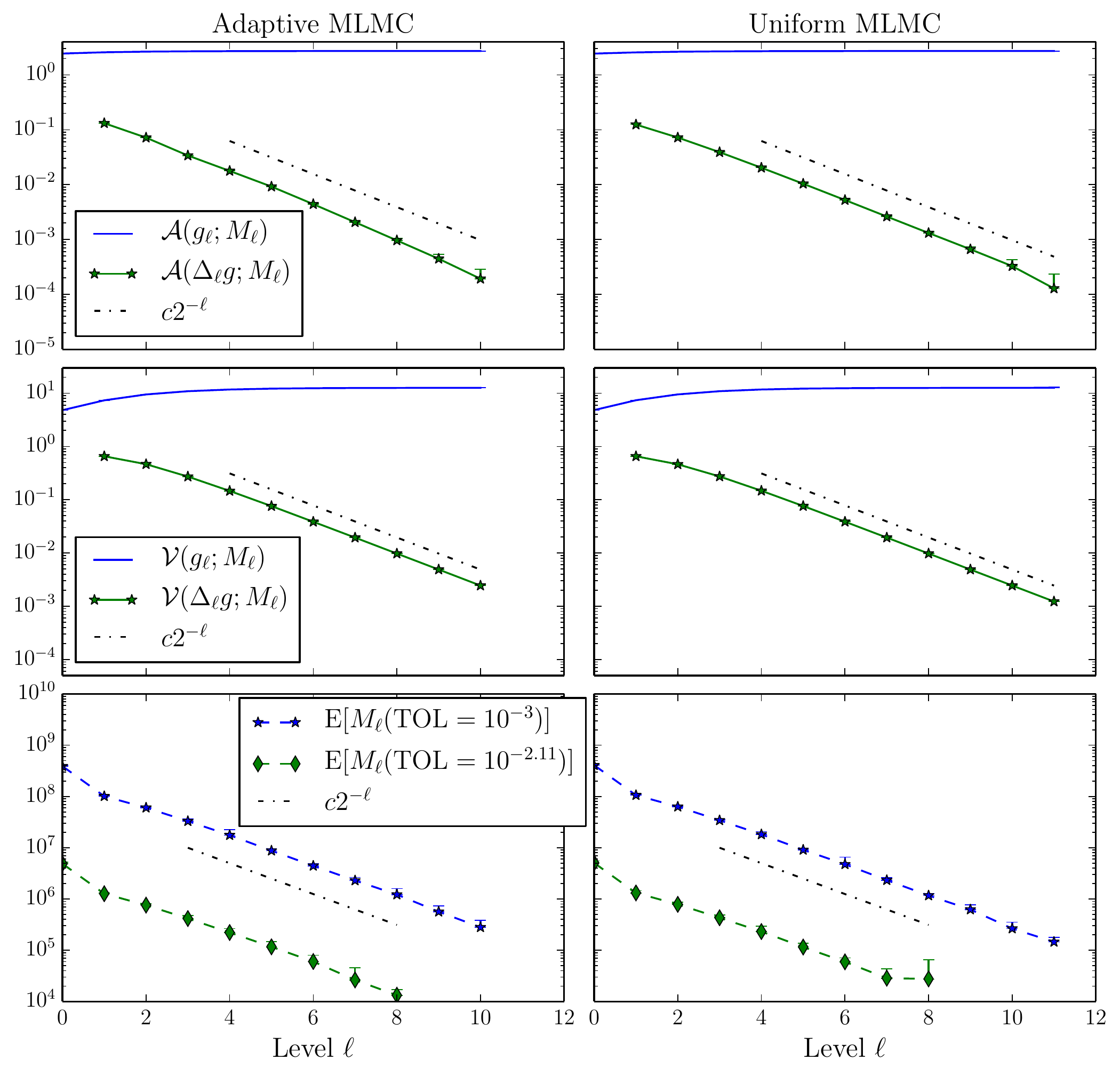}
    \caption{Output for Example~\ref{ex:mlmcGbm}
    solved with the MSE adaptive and uniform time-stepping MLMC algorithms.
    {\bf (Top)} Weak error $\abs{\Exp{\DlG}}$ for solutions at $\tol=10^{-3}$.
    {\bf (Middle)} Variance $\Var{\DlG}$ for solutions at $\tol=10^{-3}$. 
    {\bf (Bottom)} Average number of samples $E[M_l]$.}
    \label{fig:mlmcGbmLevelComparisons}
\end{figure}

\begin{remark}
The reason why we are unable to prove theoretically that the numerical solution of this 
problem computed with our adaptive algorithm asymptotically converges 
to the true solution is slightly subtle. The required smoothness conditions 
in Theorem~\ref{thm:mseExpansion1D} are obviously fulfilled, 
but due to the local update of the error indicators in our mesh refinement procedure, 
(cf. Section~\ref{subsubsec:errorIndicators}), we cannot prove that the mesh points 
will asymptotically be stopping times for which $t_n \in \FCal_{t_{n-1}}$ for all $n \in \{1,2,\ldots, N\}$.
If we instead were to use the version of our adaptive algorithm that recomputes 
all error indicators for each mesh refinement, 
the definition of the error density~\eqref{eq:rhoBar} implies that, for this particular problem, it would 
take the same value, $\rhoBar_n = \prod_{k=0}^{N-1} c_x(t_k,\barX_{t_k})^2/2$, for all indices, $n \in \{0,1,\ldots,N\}$. 
The resulting adaptively refined mesh would then become uniform and we could verify convergence, for instance,
by using Theorem~\ref{thm:mseExpansion1D}.
Connecting this to the numerical results for the adaptive algorithm that we have implemented here, we notice that 
the level-wise mean number of time steps ratio, $\Exp{\abs{\Dt{\ell}}}/N_\ell$,
presented in Figure~\ref{fig:mlmcGbmAvgNl} seems to tend towards $1$ as $\ell$ increases, a
limit ratio that is achieved only if $\Dt{\ell}$ is indeed a uniform mesh.
\end{remark}

\end{example}

\begin{example}\label{ex:mlmcSingularityGbm}

We next consider the SDE
\begin{equation}\label{eq:driftSingularity}
\begin{split}
dX_t & = \underbrace{r f(t; \xi) X_t}_{=: a(t,X_t; \xi)} dt + \underbrace{\sigma X_t}_{=:b(t,X_t; \xi)} dW_t\\                              X_0 & = 1,
\end{split}
\end{equation}
with the low-regularity drift coefficient, $f(t; \xi) = |t-\xi|^{-p}$,
interest rate, $r=1/5$, volatility, $\sigma = 0.5$, 
and observable, $g(x)=x$, at the final time
$T=1$. A new singularity point, $\xi \in U(1/4,3/4)$, is sampled for each
path, and it is independent from the Wiener paths, $W$. 
Three different blow-up exponent test cases are considered, $p=(1/2,2/3,3/4)$,
and to avoid blow-ups in the numerical integration of the drift function
component, $f(\cdot;\xi)$, we replace the fully explicit Euler--Maruyama integration
scheme with the following semi-implicit scheme:
\begin{equation}\label{eq:replacedScheme}
\barX_{t_{n+1}} = \barX_{t_n} + \begin{cases}
  r f(t_n; \xi) \barX_{t_n} \Delta t_n + \sigma \barX_{t_n} \Delta W_n, &
  \text{if} \quad f(t_n; \xi) < 2 f(t_{n+1}; \xi),\\ 
 r f(t_{n+1}; \xi) \barX_{t_n} \Delta t_n + \sigma \barX_{t_n} \Delta W_n, & \text{else.}
 \end{cases}
\end{equation}
For $p\in [1/2, 3/4]$ it may be shown that for any singularity point,
any path integrated by 
the scheme~\eqref{eq:replacedScheme} will have at most one drift-implicit
integration step.
The reference mean for the exact solution is given by
\begin{equation*}\label{eq:referenceMeanEx1}
\Exp{X_T} = 2 \int_{1/4}^{3/4} \exp\left (  \frac{r( x^{1-p} + (1-x)^{1-p})}{1-p}   \right) dx,
\end{equation*}
and in the numerical experiments, we approximate this integral value by
quadrature to the needed accuracy. 

\subsubsection*{The MSE Expansion for the Adaptive Algorithm}
Due to the low-regularity drift present in this problem, the resulting MSE
expansion will also contain drift-related terms that formally are of higher order. 
From the proof of Theorem~\ref{thm:mseExpansion1D},
equation~\eqref{eq:boundCheckDa}, we conclude 
that, to leading order the MSE is bounded by
\begin{equation*}\label{eq:errorExpansionPSingularity}
\Exp{ \abs{\barX_T - X_T}^2 } \leq \Exp{ \sum_{n=0}^{N-1}
  \phiBarX{n}^2 \frac{ \left(N (a_t + a_xa)^2\Delta t_n^2 +
      (b_xb)^2\right)(t_n,\barX_{t_n}; \xi) }{2} \Delta t_n^2}.
\end{equation*}
This is the error expansion we use for the adaptive mesh
refinement (in Algorithm~\ref{alg:meshRefinement}) in this example.
In Figure~\ref{fig:pSingularityPaths}, we illustrate the 
effect that the singularity exponent, $p$, has on SDE and adaptive mesh realizations. 
\begin{figure}[h]
    \centering
    \includegraphics[width=0.95\textwidth]{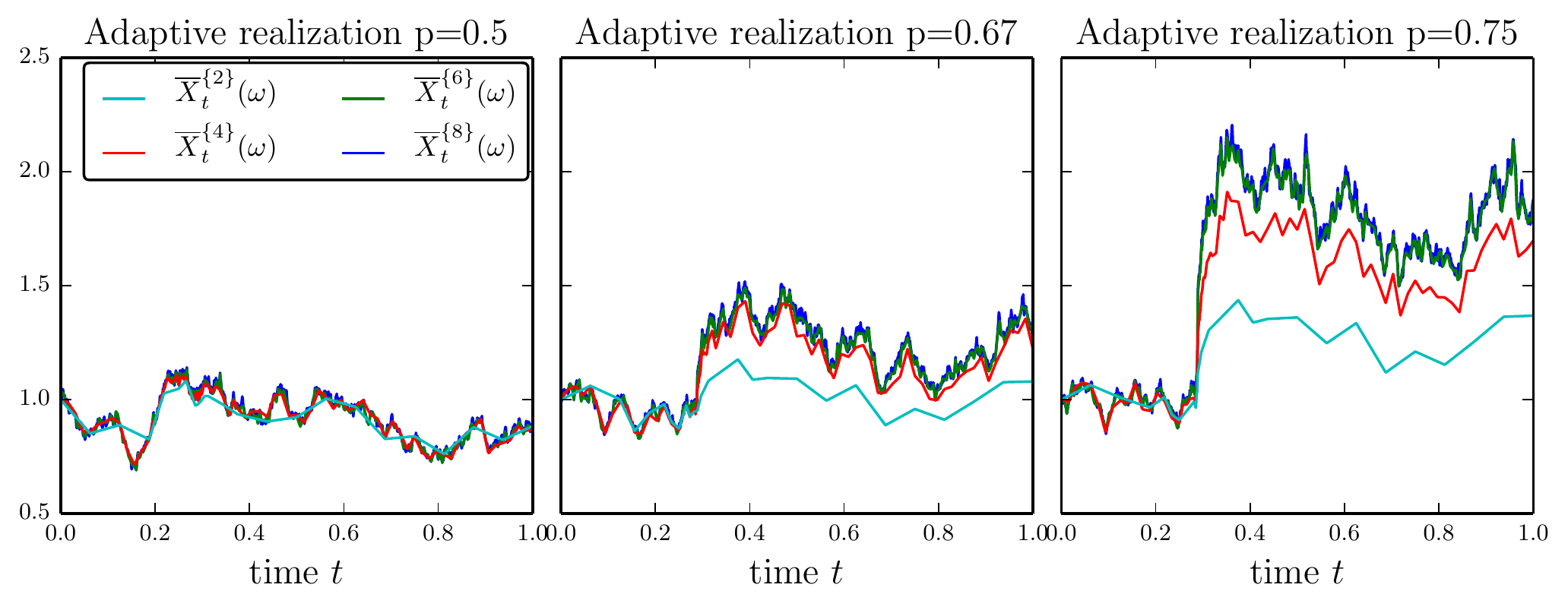}
     \hspace*{-0.2cm}\includegraphics[width=0.95\textwidth]{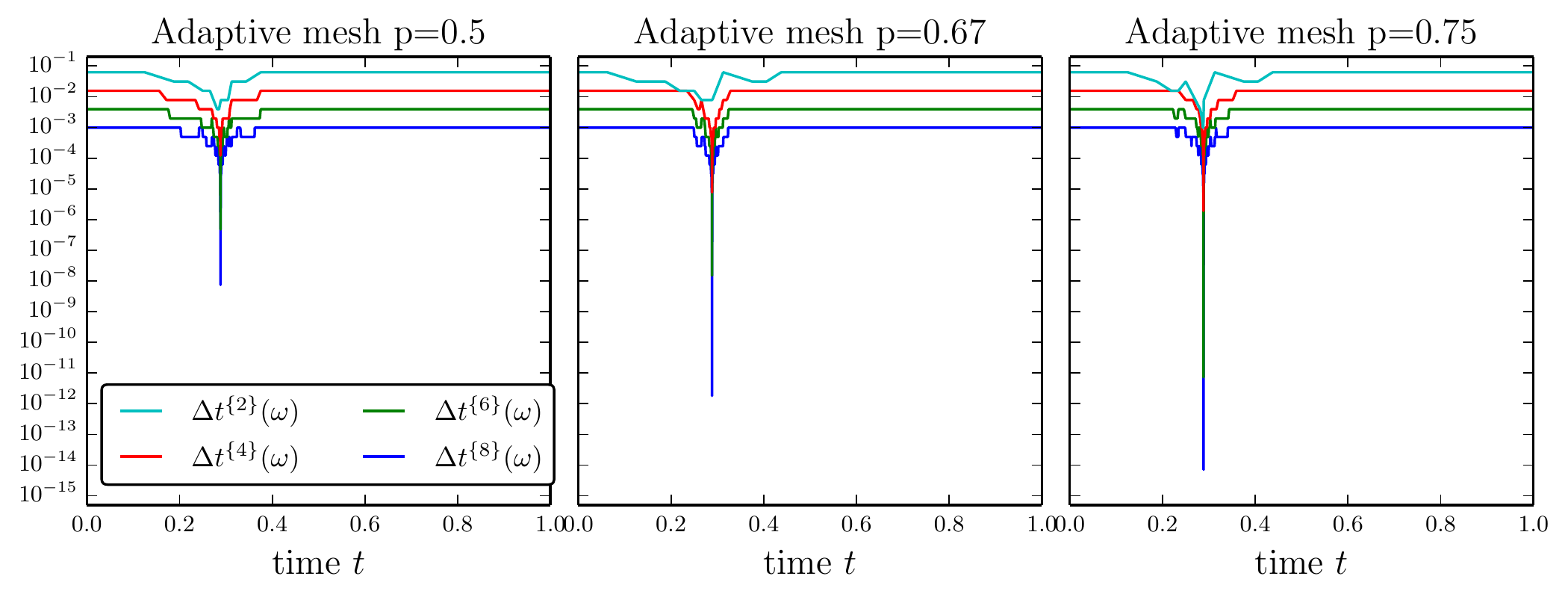}
     \caption{{\bf (Top)} One MSE adaptive numerical realization of
       the SDE problem~\eqref{eq:driftSingularity} at
       different mesh hierarchy levels. The blow-up singularity point
       is located at $\xi \approx 0.288473$ and the realizations are
       computed for three singularity exponent
       values. We observe that as the exponent, $p$, increases, the more
       jump at $t=\xi$ becomes more pronounced.
       {\bf (Bottom)} Corresponding MSE adaptive mesh realizations for
       the different test cases.  }
    \label{fig:pSingularityPaths}
\end{figure}

\subsubsection*{Implementation Details and Observations}

Computational tests for the uniform and MSE adaptive MLMC algorithms are implemented
with the input parameters summarized in Table~\ref{tab:parameters}.
The weak convergence rate, $\alpha$, which is needed in the MLMC implementations' stopping
criterion~\eqref{eq:determineL}, is estimated experimentally as $\alpha(p) =
(1-p)$ when using the Euler--Maruyama integrator with uniform time steps, and roughly 
$\alpha =1$ when using the Euler--Maruyama integrator with adaptive time steps, 
(cf.~Figure~\ref{fig:pSingularityWeakRate}). 
\begin{figure}[h]
    \centering    \includegraphics[width=0.9\textwidth]{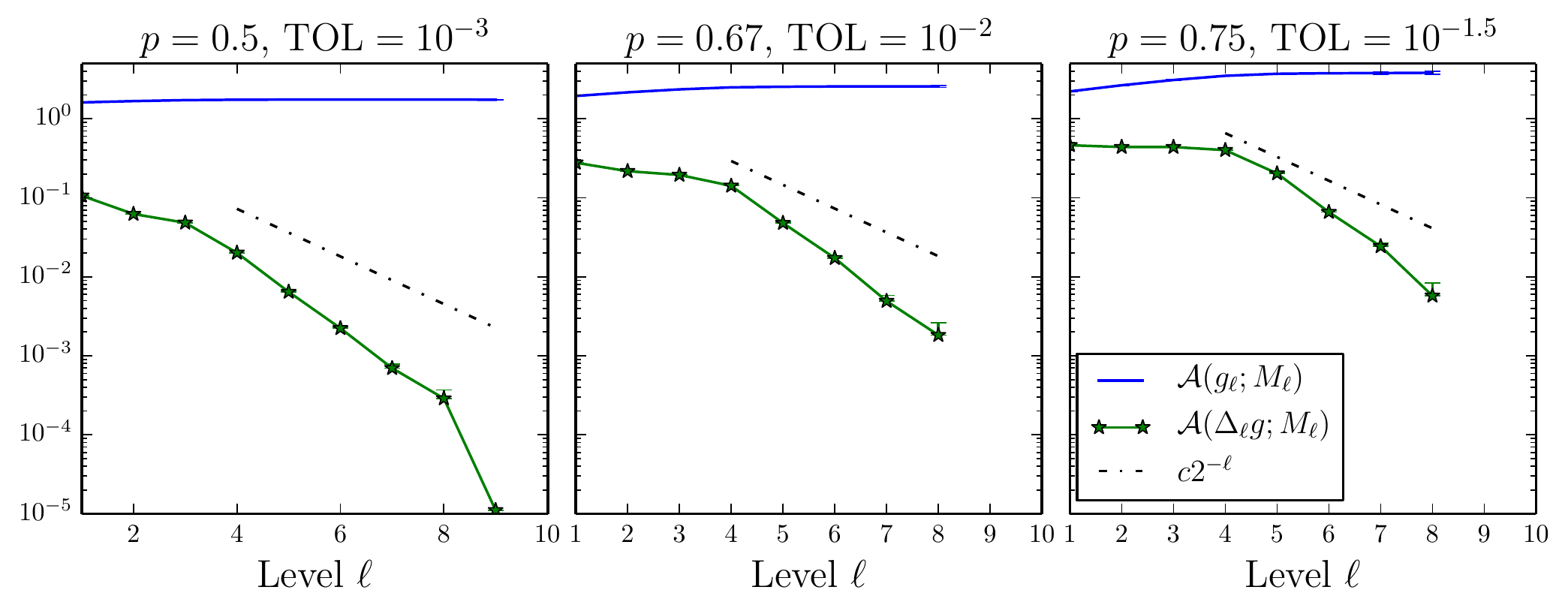}
\includegraphics[width=0.9\textwidth]{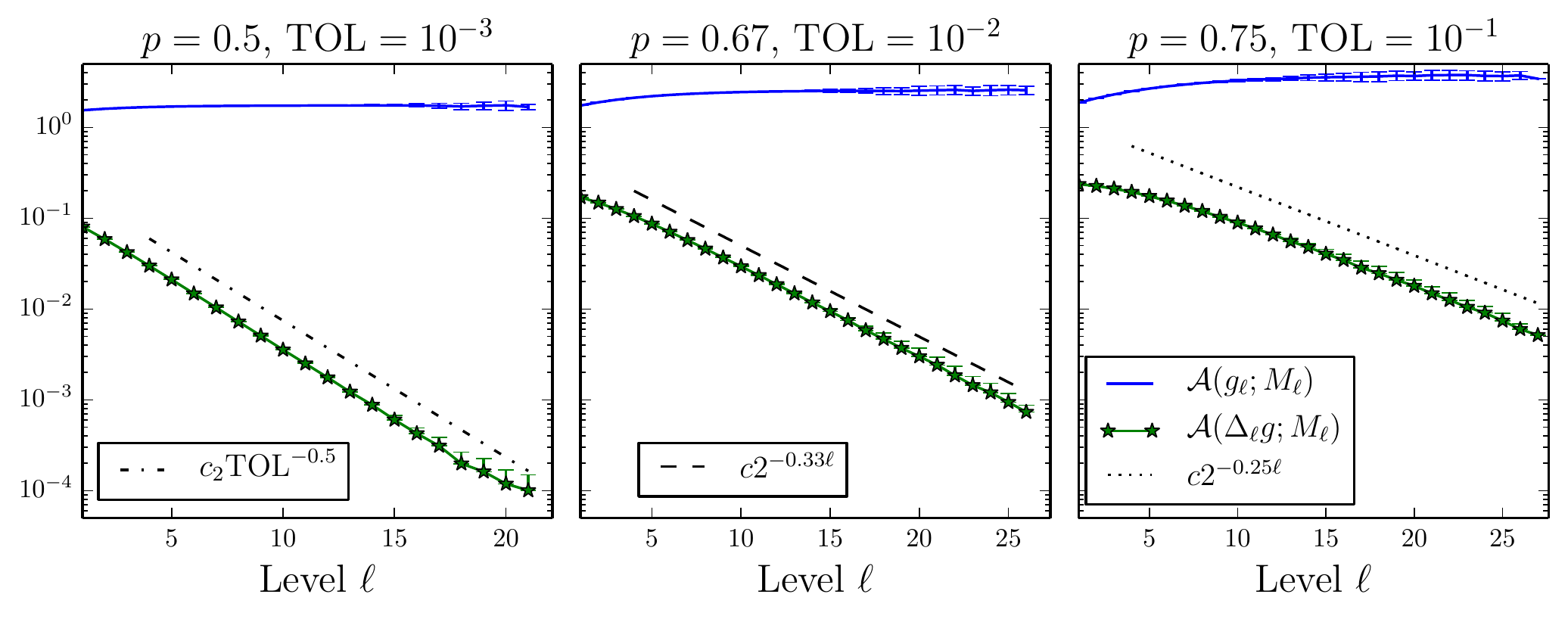}
    \caption{{\bf (Top)} Average errors $\abs{\Exp{\DlG}}$ for 
      Example~\ref{ex:mlmcSingularityGbm}
      solved with the MSE adaptive MLMC algorithm for three singularity exponent values. 
       {\bf (Bottom)} Corresponding average errors for the uniform MLMC algorithm.
    }
    \label{fig:pSingularityWeakRate}
\end{figure}
We further estimate the variance convergence rate to $\beta(p) =
2(1-p)$, when using uniform time-stepping, and roughly to $\beta
= 1$ when using MSE adaptive time-stepping, (cf.~Figure~\ref{fig:pSingularityVarRate}).
\begin{figure}[h]
    \centering
    \includegraphics[width=0.9\textwidth]{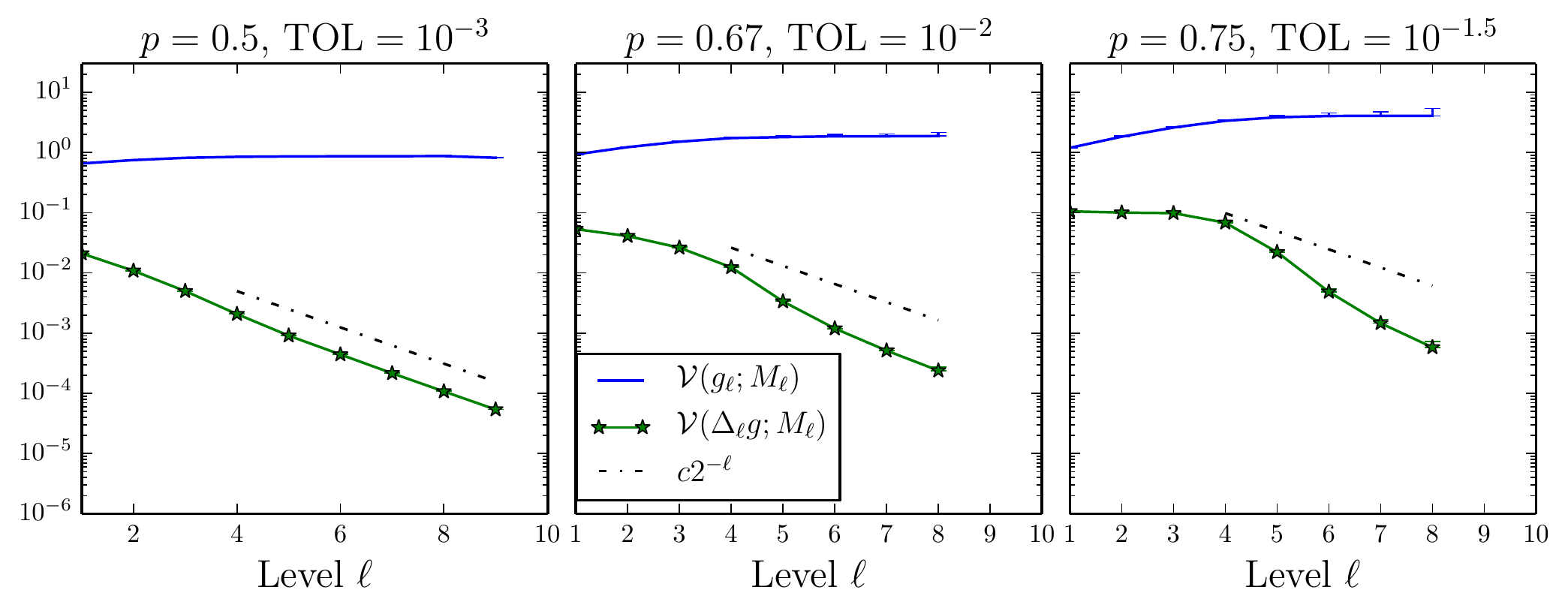}
    \includegraphics[width=0.9\textwidth]{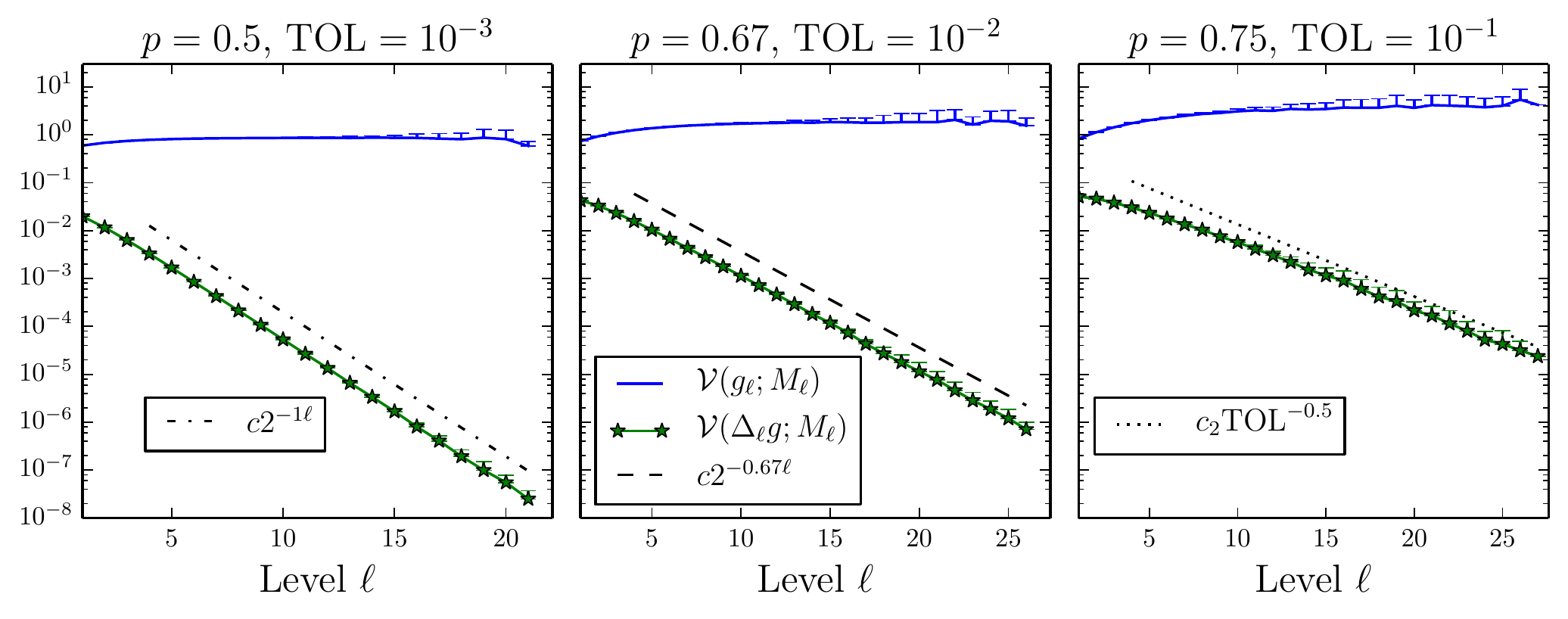}
   \caption{{\bf (Top)} Variances $\Var{\DlG}$ for 
   for Example~\ref{ex:mlmcSingularityGbm} solved with
    the MSE adaptive MLMC algorithm for three singularity exponent values. 
      {\bf (Bottom)} Corresponding variances for the uniform MLMC
      algorithm. The more noisy data on the highest levels is due to 
      the low number used for the initial samples, $\hat M=20$, and only 
	a subset of the generated 100 multilevel estimator realizations reached the last levels.
      }
    \label{fig:pSingularityVarRate}
\end{figure}
The low weak convergence rate for uniform MLMC implies that the number
of levels $L$ in the MLMC estimator will be become very large, even
with fairly high tolerances. Since computations of realizations on high
levels are extremely costly, we have, for the sake of computational
feasibility, chosen a very low value, $\widehat M =20$, for the initial number of samples
in both MLMC algorithms. 
\begin{figure}[h]
    \centering
    \includegraphics[width=0.9\textwidth]{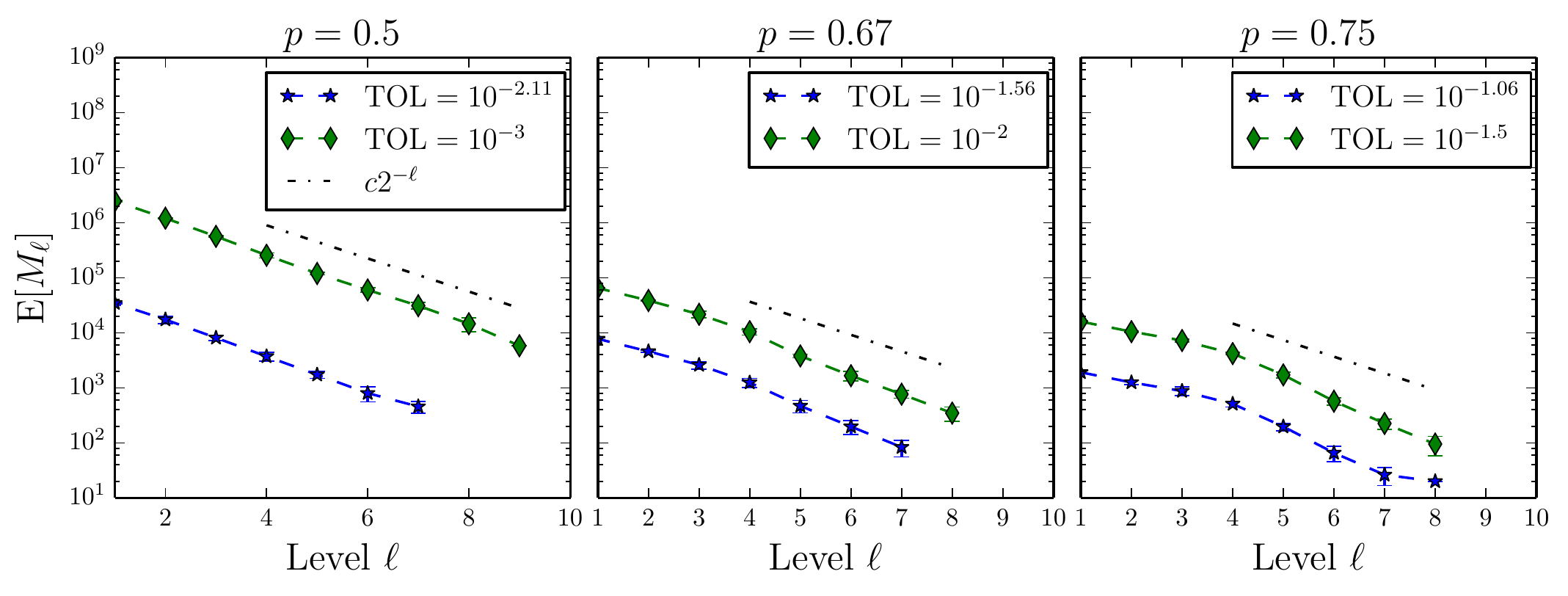}
    \includegraphics[width=0.9\textwidth]{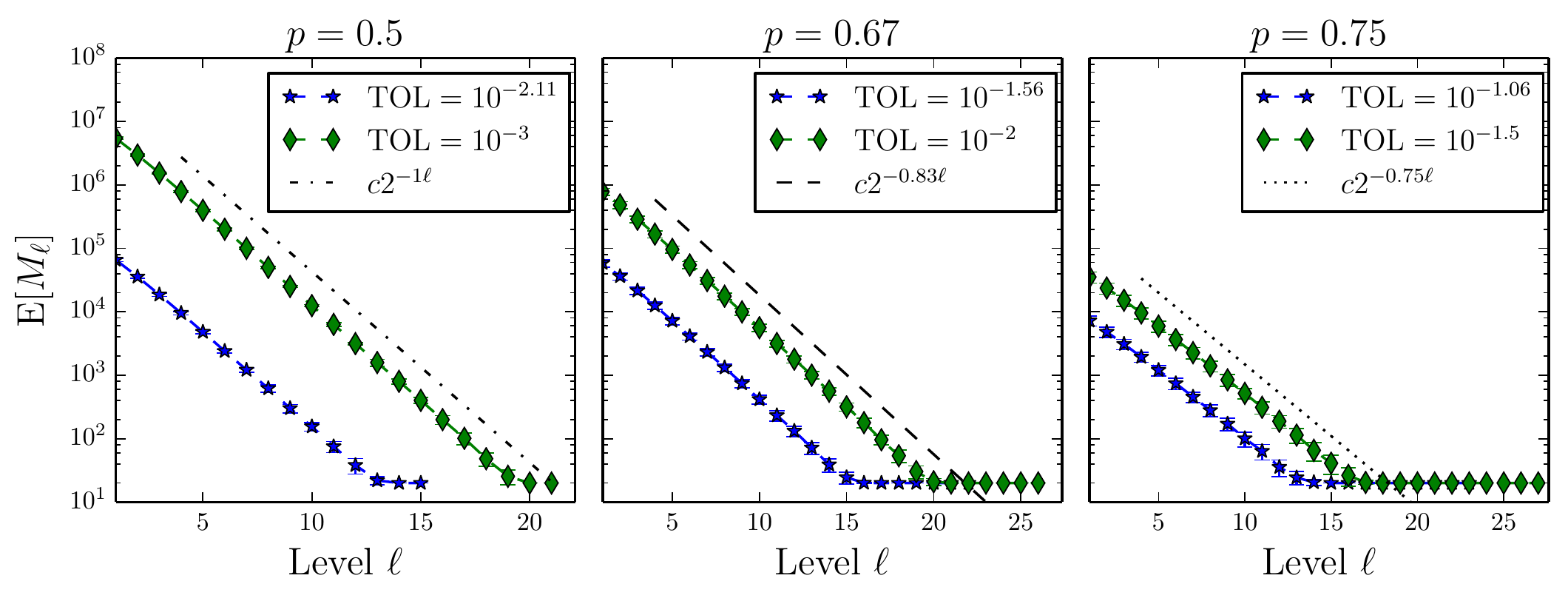}
   \caption{{\bf (Top)} Average number of samples $M_\ell$ for 
   for Example~\ref{ex:mlmcSingularityGbm} solved with
    the MSE adaptive MLMC algorithm for three singularity exponent values. 
      {\bf (Bottom)} Corresponding average number of samples
      for the uniform MLMC algorithm. The plotted decay rate reference 
      lines, $c2^{-((\beta(p)+1)/2)\ell}$, for $M_\ell$ follow implicitly 
      from equation~\eqref{eq:determineMl} (assuming that $\beta(p)= 2(1-p)$ 
      is the correct variance decay rate).}
    \label{fig:pSingularityMl}
\end{figure}
The respective estimators' use of samples, $M_\ell$, 
(cf.~Figure~\ref{fig:pSingularityMl}), shows that the low number of initial
samples is not strictly needed for the the adaptive MLMC algorithm, but
for the sake of fair comparisons, we have chosen to use the same
parameter values in both algorithms. 

From the rate estimates of $\alpha$ and $\beta$, we predict the computational
cost of reaching the approximation goal~\eqref{eq:objectiveRecall}
for the respective MLMC algorithms to be
\[
\text{Cost}_{\mathrm{adp}} (\AMLSimple) = \bigO{ \log(\tol)^4 \tol^{-2 }} \quad
\text{and} \quad 
\text{Cost}_{\mathrm{unf}}(\AMLSimple) = \bigO{\tol^{-\frac{1}{1-p}}},
\]
by using the estimate~\eqref{eq:cmlBound}
and Theorem~\ref{thm:cliffeComplexity} respectively.
These predictions fit well with the observed computational runtime for the
respective MLMC algorithms, (cf.~Figure~\ref{fig:pSingularityRuntime}). 
\begin{figure}[h]
    \centering
    \includegraphics[width=0.9\textwidth]{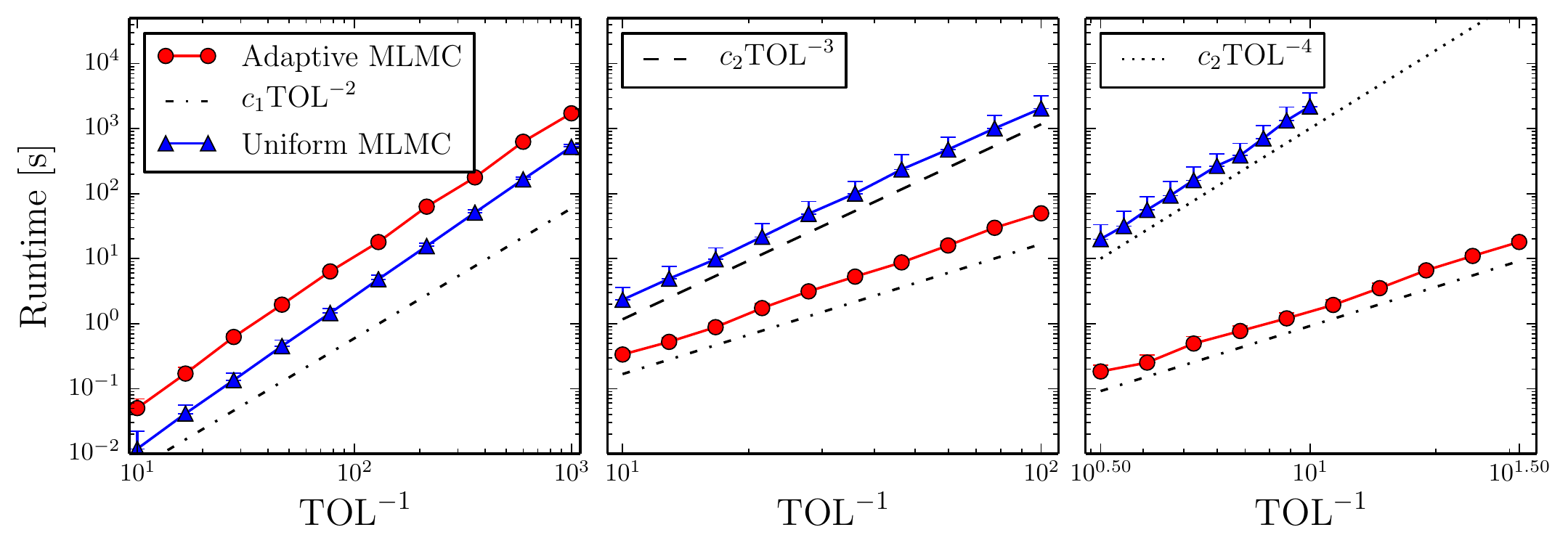}
    \caption{ Average runtime vs. $\tol^{-1}$ for the two MLMC
   algorithms for three singularity exponent values 
   in Example~\ref{ex:mlmcSingularityGbm}.}
    \label{fig:pSingularityRuntime}
\end{figure}
Lastly, we observe that the numerical results are
consistent with both algorithms fulfilling the
goal~\eqref{eq:objectiveRecall} in Figure~\ref{fig:pSingularityError}.
\begin{figure}[h]
    \centering
    \includegraphics[width=0.9\textwidth]{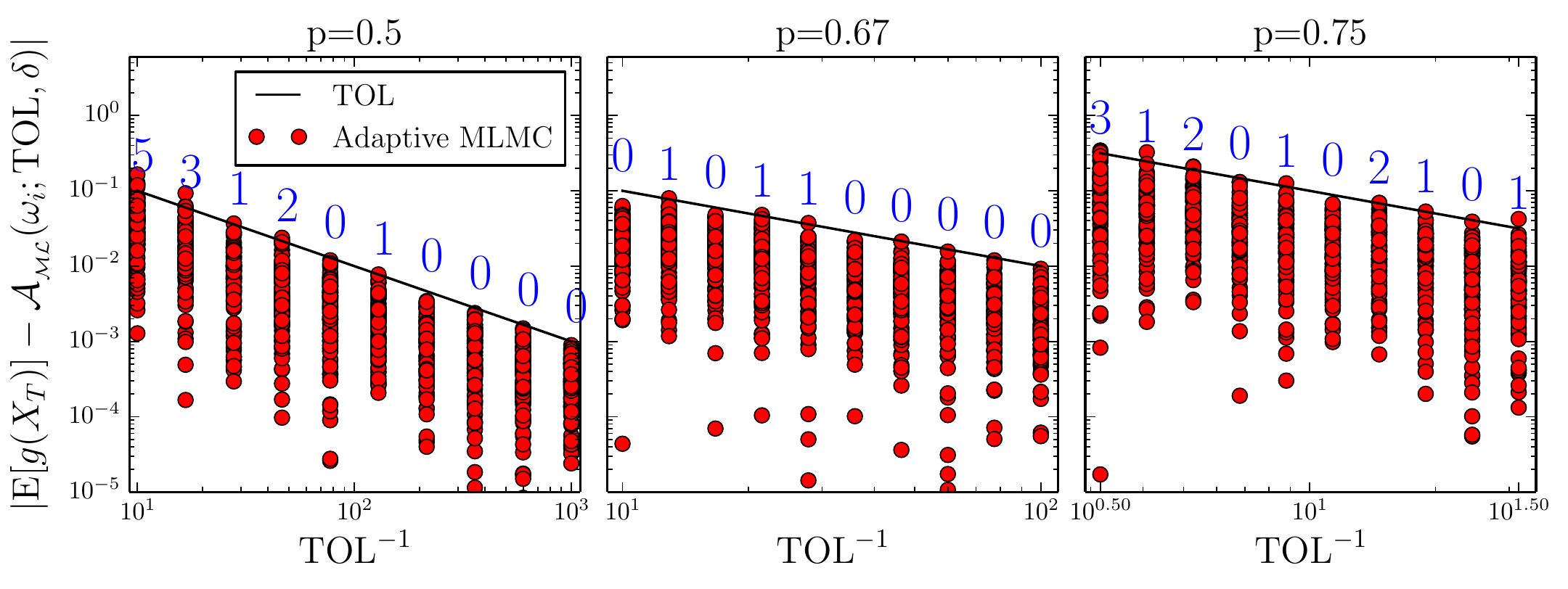}
    \includegraphics[width=0.9\textwidth]{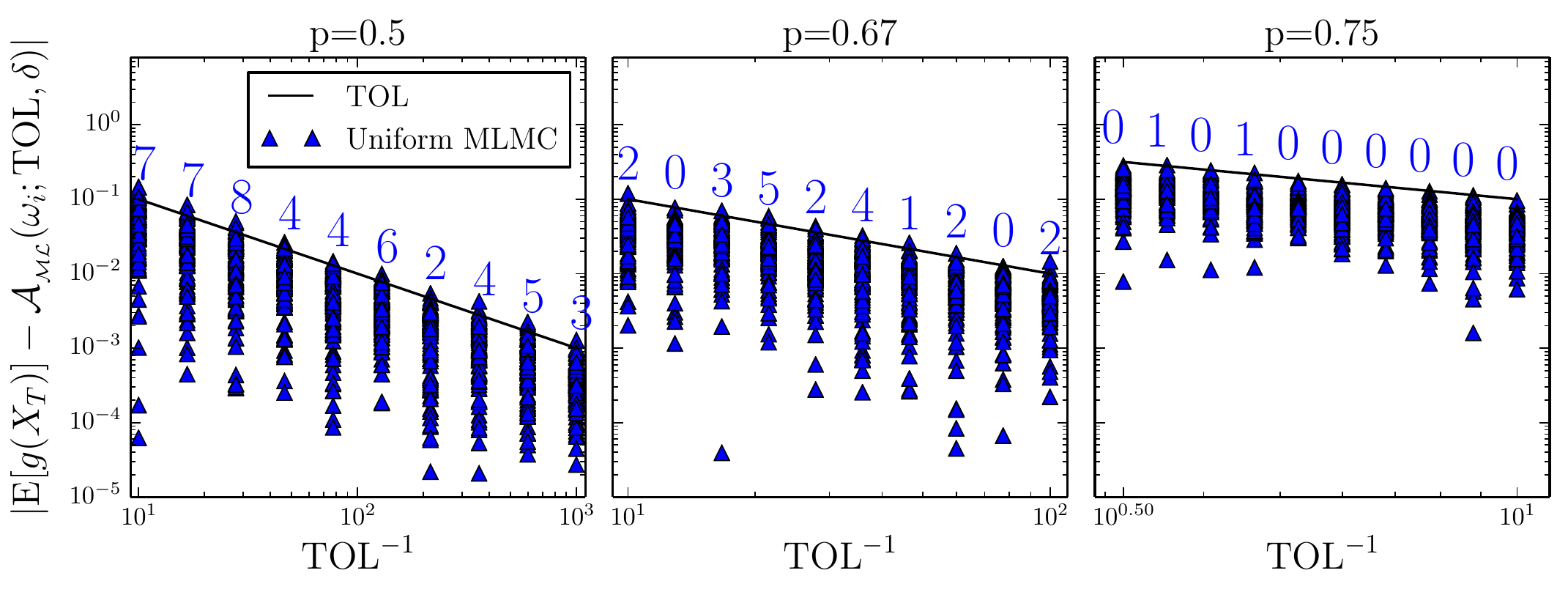}
   \caption{Approximation errors for both of the MLMC algorithms solving
     Example~\ref{ex:mlmcSingularityGbm}. At every $\tol$ value, circles
     and triangles represent the errors from 100 independent
     multilevel estimator realizations of the respective 
     algorithms.}
               \label{fig:pSingularityError}
\end{figure}

\end{example}

\addtocontents{toc}{\protect\setcounter{tocdepth}{1}}

\subsection*{Computer Implementation}

The computer code for all algorithms was
written in Java and used the ``Stochastic Simulation in Java'' library
to sample the random variables in parallel from thread-independent {\bf MRG32k3a} pseudo random
number generators, ~\cite{Lecuyer05}. The experiments were run on
multiple threads on Intel Xeon(R) CPU X5650, $2.67$GHz processors and the 
computer graphics were made using the open source plotting library Matplotlib,
\cite{Hunter07}. 

\addtocontents{toc}{\protect\setcounter{tocdepth}{2}}

\section{Conclusion} \label{sec:conclusion}

We have developed an {\aposteriori}, MSE adaptive Euler--Maruyama time-stepping algorithm
and incorporated it into an MSE adaptive MLMC algorithm. The MSE error expansion presented in
Theorem~\ref{thm:mseExpansion1D} is fundamental to the adaptive algorithm. 
Numerical tests have shown that MSE adaptive time-stepping may
outperform uniform time-stepping, both in the single-level MC
setting and in the MLMC setting, (Examples~\ref{ex:wTSquared}
and~\ref{ex:mlmcSingularityGbm}). 
Due to the complexities of implementing adaptive time-stepping, the
numerical examples in this work were restricted to quite simple,
low-regularity SDE problems with singularities in the temporal
coordinate. In the future, we aim to study SDE problems with
low-regularity in the state coordinate (preliminary tests and analysis
do however indicate that then some ad hoc molding of the adaptive
algorithm is required).

Although \aposteriori adaptivity has proven to be a very effective method for 
deterministic differential equations, the use of
information from the future of the numerical solution of the dual problem
makes it a somewhat unnatural method to extend to
\Ito SDE: It can result in numerical solutions that are not
$\FCal_t$-adapted, which consequently may introduce a bias in the
numerical solutions. \cite{Gaines97} provides an example of a failing 
adaptive algorithm for SDE. A rigorous analysis of the convergence properties 
of our developed MSE adaptive algorithm would strengthen the 
theoretical basis of the algorithm further. We leave this for future work.

\newpage 
~\newpage

\appendix

\section{Theoretical Results} 

\subsection{Error Expansion for the MSE in 1D}

In this section, we derive a leading-order error expansion for the
MSE~\eqref{eq:adaptivityGoal} in the 1D setting when the drift and
diffusion coefficients are respectively mappings {of} the form $a: [0,T] \times \R
$ and $b: [0,T] \times \R \to \R$.
We begin by deriving a representation of the MSE in terms of products
of local errors and weights.

Recalling the definition of the flow map, $\varphi(x,t):=
g(X_T^{x,t})$, and the first variation of the flow map
and the path itself given in Section~\ref{subsec:firstVar},
we use the Mean Value Theorem to deduce that  
\begin{equation}\label{eq:obsDifference}
\begin{split}
\g{X_T}- \gBarXT & =  \varphi(0,x_0) - \varphi(0,\barXT) \\
& = \sum_{n=0}^{N-1}  \varphi(t_{n},\barX_{t_n}) - \varphi(t_{n+1},\barX_{t_{n+1}}) \\
& = \sum_{n=0}^{N-1}  \varphi\parenthesis{t_{n+1},X_{t_{n+1}}^{\barX_{t_n},t_{n}}} - \varphi(t_{n+1},\barX_{t_{n+1}}) \\
& = \sum_{n=0}^{N-1}  \phiX{t_{n+1},\barX_{t_{n+1}} + s_n \Delta e_n } \Delta e_n,
\end{split}
\end{equation}
where the \emph{local error} is given by $\Delta e_n :=  X_{t_{n+1}}^{\barX_{t_n},t_{n}} -\barX_{t_{n+1}}$ and $s_n \in [0,1]$.
\Ito expansion of the local error gives the following representation:
\begin{equation}\label{eq:localErrorTerm}
\begin{split}
&\Delta e_n  = \underbrace{\int_{t_{n}}^{t_{n+1}}  a(t,X_{t}^{\barX_{t_n},t_{n}}) - a(t_n,\barX_{t_n}) \,dt}_{\Delta a_n}  
+ \underbrace{\int_{t_{n}}^{t_{n+1}} b(t,X_{t}^{\barX_{t_n},t_{n}}) - b(t_n, \barX_{t_n}) \,dW_t}_{\Delta b_n}\\
&= \underbrace{\int_{t_{n}}^{t_{n+1}} \int_{t_{n}}^t (a_t + a_x a + \frac{a_{xx}}{2} b^2)(s,X_{s}^{\barX_{t_n},t_{n}}) \,ds \,dt}_{=: \checkDa_{n} } + 
\underbrace{ \int_{t_{n}}^{t_{n+1}} \int_{t_{n}}^t (a_x b)(s,X_{s}^{\barX_{t_n},t_{n}}) \, dW_s \,dt }_{ =: \tildeDa_{n}} \\
& + \underbrace{\int_{t_{n}}^{t_{n+1}} \int_{t_{n}}^t (b_t+b_x a + \frac{b_{xx}}{2} b^2 )(s,X_{s}^{\barX_{t_n},t_{n}}) ds \,dW_t}_{=: \checkDb_n  } 
+  \underbrace{ \int_{t_{n}}^{t_{n+1}} \int_{t_{n}}^t (b_xb)(s,X_{s}^{\barX_{t_n},t_{n}}) dW_s \,dW_t}_{=: \tildeDb_{n} }.
\end{split}
\end{equation}
By equation~\eqref{eq:obsDifference} we may express the MSE by
the following squared sum
\begin{equation*}\label{eq:meanSquareErrorExpansion1}
\begin{split}
&\Exp{\parenthesis{g(X_T) -\gBarXT}^2}  = \Exp{ \left(\sum_{n=0}^{\check
      N-1} \phiX{t_{n+1}, \barX_{t_{n+1}} + s_n \Delta e_n} \Delta e_n \right)^2 }\\
& \qquad \quad = \sum_{n,k=0}^{\check N-1} \Exp{ \phiX{t_{k+1},\barX_{t_{k+1}} + s_k
    \Delta e_k} \phiX{t_{n+1}, \barX_{t_{n+1}} + s_n \Delta e_n} \Delta e_k \Delta e_n  }.
\end{split}
\end{equation*}
This is the first step in deriving the error expansion in
Theorem~\ref{thm:mseExpansion1D}. The remaining steps follow in
the proof below. 

\begin{proof}[Proof of Theorem~\ref{thm:mseExpansion1D}]

The main tools used in proving this theorem are Taylor and It{\^o}--Taylor expansions, \Ito isometry, and truncation of
higher order terms. For errors attributed to the leading-order local error term, $\tildeDb_n$, 
(cf.~equation~\eqref{eq:localErrorTerm}), we do detailed 
calculations, and the remainder is bounded by stated higher order terms.

We begin by noting that under the assumptions in Theorem~\ref{thm:mseExpansion1D}
Lemmas~\ref{lem:existUniquePath} and~\ref{lem:numSolConv} respectively verify
then the existence
and uniqueness of the solution of the SDE $X$ and the numerical solution $\barX$, 
and provide higher order moment bounds for both. Furthermore, due to the assumption
of the mesh points being stopping times for which $t_n \in \FCal_{t_{n-1}}$ for all $n$, 
it follows also that the numerical solution is adapted to the filtration, i.e., $\barX_{t_n} \in \FCal_{t_n}$
for all $n$.

We further need to extend the flow map and the first variation 
notation from Section~\ref{subsec:firstVar}. Let   
$\barX^{x, t_k}_{t_n}$ for $n\ge k$ denote the numerical solution of the Euler--Maruyama  
scheme
\begin{equation}
\barX_{t_{j+1}}^{x,t_k} 
=  \barX_{t_{j}}^{x,t_k} + a(t_j, \barX_{t_j}^{x,t_k})  \Delta t_j 
+ b(t_j, \barX_{t_j}^{x,t_k}) \Delta W_j, \quad j \ge k,
\end{equation}
with initial condition $X_{t_k} = x$. The first variation of $\barX^{x, t_k}_{t_n}$ 
is defined by $\partial_x \barX_{t_n}^{x,t_k}$.
Provided that $\Exp{|x|^{2p}}< \infty$ for all $p \in \N$, $x \in \FCal_{t_k}$ 
and provided the assumptions of Lemma~\ref{lem:numSolConv} hold, 
it is straightforward to extend the proof of the lemma  
to verify that $(\barX^{x, t_k}, \partial_x \barX^{x,t_k})$ 
converges strongly to $(X^{x, t_k}, \partial_x X^{x,t_k})$ for $t \in [t_k, T]$, 
\begin{equation*}\label{eq:strongConvPath2}
\begin{split}
\max_{k \leq n \leq \check{N} }  \parenthesis{  \parenthesis{\Exp{ \abs{\barX_{t_n}^{x,t_k} - X_{t_n}^{x,t_k}}^{2p}}}^{1/2p} }
& \leq C \check{N}^{-1/2}, \quad \forall p \in \N \\
\max_{k \leq n \leq \check{N}  }  \parenthesis{ \parenthesis{ \Exp{ \abs{\partial_{x} \barX_{t_n}^{x,t_k} - \partial_{x} X_{t_n}^{x,t_k}}^{2p} } }^{1/2p} }
& \leq C \check{N}^{-1/2}, \quad \forall p \in \N 
\end{split}
\end{equation*}
and
\begin{equation}
\max_{ k \leq n \leq \check{N} } 
\parenthesis{ \max \parenthesis{ \Exp{ \abs{\barX_{t_n}^{x,t_k}}^{2p}}, \Exp{ \abs{\partial_x\barX_{t_n}^{x,t_k}}^{2p}  } } } < \infty, 
\quad \forall p \in \N.
\end{equation}
In addition to this, we will also make use of moment bounds for the second
and third variation of the flow map in the proof, i.e.,  $\phiXX{t,x}$
and $\phiXXX{t,x}$. The second 
variation is described in Section~\ref{subsec:higherOrderVariations},
where it is shown in Lemma~\ref{lem:existenceUniquenessHigherVariations} that provided that $x\in \FCal_t$ and $\Exp{|x|^{2p}}<\infty$ 
for all $p \in \N$, then
\[
\max \parenthesis{ \Exp{ \abs{\phiXX{t, x}}^{2p} }, \Exp{ \abs{\phiXXX{t, x}}^{2p} }, \Exp{ \abs{\phiXXXX{t, x}}^{2p} }} < \infty, \quad \forall p \in \N.
\]

Considering the MSE error contribution from the leading order local
error terms $\tildeDb_n$, i.e., 
\begin{equation}\label{eq:boundTildeDb}
\Exp{ \phiX{t_{k+1}, \barX_{t_{k+1}} + s_k \Delta e_k  } \phiX{t_{n+1},\barX_{t_{n+1}} + s_n \Delta e_n} \tildeDb_k \tildeDb_{n} },
\end{equation}
we have for $k=n$,
\begin{align*}\label{eq:dbTaylor1}
& \Exp{ \left(\phiX{t_{n+1}, \barX_{t_{n+1}} } + \phiXX{t_{n+1}, \barX_{t_{n+1}} + \hat s_n \Delta e_n  } s_n\Delta e_n \right)^2 \tildeDb_{n}^2 } \\
 = & \Exp{ \phiX{t_{n+1}, \barX_{t_{n+1}} }^2   \tildeDb_{n}^2  +  \littleO{\Delta t_n^2}}.
\end{align*}
The above $\littleO{\Delta t_n^2}$ follows from Young's and H{\"o}lder's inequalities,
\begin{equation}\label{eq:boundDiagonalTerms}
\begin{split}
& \Exp{ 2\phiX{t_{n+1}, \barX_{t_{n+1}}} \phiXX{t_{n+1},\barX_{t_{n+1}} + \hat s_n \Delta e_n  } s_n \Delta e_n \tildeDb_{n}^2 }\\
& \leq C \parenthesis{ \Exp{\parenthesis{   \phiX{t_{n+1}, \barX_{t_{n+1}}} \phiXX{t_{n+1}, \barX_{t_{n+1}} 
+ \hat s_n \Delta e_n }}^2 \Delta t_n^3} +  \Exp{ \frac{\Delta e_n^2 \tildeDb_{n}^4}{\Delta t_n^3} } } \\
&  \leq C \Bigg(  \Exp{\Exp{ \parenthesis{   \phiX{t_{n+1}, \barX_{t_{n+1}} } \phiXX{t_{n+1}, \barX_{n+1} + \hat s_n \Delta e_n  }}^2 \, \Big| \FCal_{t_n} } \Delta t_n^3}  \\
&  \quad +  \Exp{ \frac{\checkDa_n^2 \tildeDb_{n}^4}{\Delta t_n^3}}   
+  \Exp{ \frac{\tildeDa_n^2 \tildeDb_{n}^4}{\Delta t_n^3}} 
    +  \Exp{ \frac{\checkDb_n^2 \tildeDb_{n}^4}{\Delta t_n^3} }  + \Exp{ \frac{\tildeDb_{n}^6}{\Delta t_n^{3} }}  \Bigg)\\
& \leq C \Bigg\{ \Exp{\Delta t_n^3} + \Bigg(\sqrt{\Exp{\Exp{\checkDa_n^4 | \FCal_{t_n}} \frac{1}{\Delta t_n}}} 
+ \sqrt{\Exp{\Exp{\tildeDa_n^4 | \FCal_{t_n}} \frac{1}{\Delta t_n}}} \\
&  \quad  + \sqrt{\Exp{\Exp{\checkDb_n^4 | \FCal_{t_n}} \frac{1}{\Delta t_n}}} + 
\sqrt{\Exp{\Exp{\tildeDb_n^4 | \FCal_{t_n}} \frac{1}{\Delta t_n}}}\Bigg) \sqrt{\Exp{\Exp{\tildeDb_n^8 | \FCal_{t_n}} \frac{1}{\Delta t_n^5}}} \Bigg\}\\
&  =\Exp{ o(\Delta t_n^2)} 
\end{split}
\end{equation}
where the last inequality is derived by applying the moment bounds for multiple \Ito integrals described in~\cite[Lemma 5.7.5]{Kloeden92}
and under the assumptions (R.1), (R.2), (M.1), (M.2) and (M.3). This yields
\begin{equation}\label{eq:multiItoBounds}
 \begin{split}
  \Exp{\checkDa_n^4|\FCal_{t_n}}    &    \leq C \Exp{\sup_{s \in [t_n,t_{n+1}) } 
          \abs{a_t + a_x a + \frac{a_{xx}}{2} b^2}^4(s,X_{s}^{\barX_{t_n},t_{n}}) \, \Big| \, \FCal_{t_n} }  \Delta t_n^8,  \\  
  \Exp{\tildeDa_n^4 | \FCal_{t_n}}  &    \leq   C \Exp{ \sup_{s \in [t_n,t_{n+1}) } \abs{a_x b}^4(s,X_{s}^{\barX_{t_n},t_{n}})\, \Big| \,\FCal_{t_n}} \Delta t_n^6,\\
  \Exp{\checkDb_n^4| \FCal_{t_n}}   &    \leq C \Exp{ \sup_{s \in [t_n,t_{n+1})} \abs{b_t + b_x a + \frac{b_{xx}}{2} b^2 }^4(s,X_{s}^{\barX_{t_n},t_{n}}) \, \Big| \,\FCal_{t_n}} 
\Delta t_n^6, \\ 
\Exp{\tildeDb_n^{4} | \FCal_{t_n}}   &    \leq C \Exp{ \sup_{s \in [t_n,t_{n+1})  } \abs{b_x b}^{4}(s,X_{s}^{\barX_{t_n},t_{n}})\, \Big| \,\FCal_{t_n}} \Delta t_n^4, \\
  \Exp{\tildeDb_n^{8} | \FCal_{t_n}} &    \leq C  \Exp{ \sup_{s \in [t_n,t_{n+1})} \abs{b_x b}^{8}(s,X_{s}^{\barX_{t_n},t_{n}})\, \Big| \,\FCal_{t_n}} \Delta t_n^8.
 \end{split}
\end{equation}
And by similar reasoning, 
\[
\Exp{ \phiXX{\barX_{t_{n+1}} + \hat s_n \Delta e_n ,t_{n+1} }^2s_n^2  \Delta e_n^2 \tildeDb_{n}^2 } \leq C \Exp{\Delta t_n^4}.
\]
For achieving independence between forward paths and dual solutions in the expectations,
an It{\^o}--Taylor expansion of $\varphi_x$ leads to the equality
\begin{equation*}\label{eq:dbTaylor2}
 \Exp{ \phiX{t_{n+1}, \barX_{t_{n+1}} }^2 \tildeDb_{n}^2 } =  
\Exp{ \phiX{t_{n+1},\barX_{t_{n}} }^2  \tildeDb_{n}^2  +  \littleO{\Delta t_n^{2}} }.
\end{equation*}
Introducing the null set completed $\sigma-$algebra 
\begin{equation*}\label{eq:FnDef}
\hatF^{n} = \overline{\sigma \left( \sigma( \{ W_{s}\}_{ 0 \leq s \leq
      t_n})  \lor \sigma( \{ W_s-W_{t_{n+1}}\}_{ t_{n+1} \leq s \leq
      T} ) \right) \lor \sigma(X_0)},
\end{equation*}
we observe that
$\phiX{t_{n+1}, \barX_{t_n}}^2$ is $\hatF^{n}$ measurable by
construction, (cf.~\cite[App.~B]{Oksendal98}).
Moreover, by conditional expectation,
\[
\begin{split}
\Exp{ \phiX{t_{n+1}, \barX_{t_{n}}}^2   \tildeDb_{n}^2  }  & = 
\Exp{ \phiX{t_{n+1}, \barX_{t_{n}} }^2 \Exp{ \tildeDb_{n}^2 |\hatF^{n} }  }\\
& = \Exp{  \phiX{t_{n+1}, \barX_{t_{n}} }^2 (b_xb)^2(t_n,\barX_{t_n}) \frac{\Delta t_n^2}{2}   + \littleO{\Delta t_n^{2} } },
\end{split}
\]
where the last equality follows from using It{\^o}'s formula,
\[
\begin{split}
(b_x b)^2 (t,X_{t}^{\barX_{t_n},t_{n}}) &= (b_x b)^2 (t_n,\barX_{t_n}) + \int_{t_n}^t \parenthesis{\Big(\partial_t + a \partial_x + \frac{b^2}{2} \partial_x^2 \Big)(b_x b)^2 }(s,X_{s}^{\barX_{t_n},t_{n}}) \, ds \\
& \quad + \int_{t_n}^t  \parenthesis{b \partial_x (b_x b)^2 }(s,X_{s}^{\barX_{t_n},t_{n}}) \, dW_s, \quad  t\in [t_n, t_{n+1}),
\end{split}
\]
to derive that 
\[
\begin{split}
\Exp{ \tildeDb_{n}^2 |\hatF^{n} } & = \Exp{\left(\int_{t_{n}}^{t_{n+1}} \int_{t_{n}}^t (b_xb)(s,X_s^{\barX_{t_n},t_n}) dW_s \,dW_t \right)^2 \Big | \barX_{t_n} }\\
 & =  \frac{(b_xb)^2( t_{n},\barX_{t_n} )}{2} \Delta t_n^2 + \littleO{\Delta t_{n}^{2}}.
\end{split}
\]
Here, the higher order $\littleO{\Delta t_{n}^{2}}$ terms are bounded
in a similar fashion as the terms in
inequality~\eqref{eq:boundDiagonalTerms}, by using~\cite[Lemma 5.7.5]{Kloeden92}.


For the terms in~\eqref{eq:boundTildeDb} for which $k <n$, we will show that
\begin{equation}\label{eq:boundDb2}
\begin{split}
&
\sum_{k,n=0}^{\check{N}-1} \Exp{ \phiX{t_{k+1}, \barX_{t_{k+1}} + s_k \Delta e_k  } \phiX{t_{n+1}, \barX_{t_{n+1}} + s_n \Delta e_n} \tildeDb_k \tildeDb_{n} } 
= \sum_{n=0}^{\check{N}-1}
\Exp{ \littleO{\Delta t_n^2 }},
\end{split}
\end{equation}
which means that the contribution to the MSE from these terms is
negligible to leading order. For the use in later expansions, let us first
observe by use of the chain rule that for any $y \in \FCal_{t_n}$ with
bounded second moment,
\[
\begin{split}
\phiX{t_{k+1}, y  }  &= g'(X_T^{y, t_{k+1}} ) \partial_x X_{T}^{y ,t_{k+1}}\\
& = g'(X_T^{\barX_{t_{k+1}} + s_m \Delta e_k, t_{k+1}} ) 
\partial_x X_{T}^{X_{t_{n+1}}^{y,t_{k+1}}  ,t_{n+1}}
\partial_x X_{t_{n+1}}^{y,t_{k+1}}\\
&= \phiX{t_{n+1}, X_{t_{n+1}}^{y,t_{k+1}}  } \partial_x X_{t_{n+1}}^{y,t_{k+1}},
\end{split}
\] 
and that
\begin{multline*}
\partial_x  X_{t_{n+1}}^{\barX_{t_{k+1}} + s_k \Delta e_k, t_{k+1}} = 
\partial_x  X_{t_{n}}^{\barX_{t_{k+1}} + s_k \Delta e_k, t_{k+1}}  \\
+ \int_{t_n}^{t_{n+1}} a_x(s,X_{s}^{\barX_{t_{k+1}} + s_k \Delta e_k, t_{k+1}}) \partial_x X_{s}^{\barX_{t_{k+1}} + s_k \Delta e_k, t_{k+1}} ds \\
+\int_{t_n}^{t_{n+1}} b_x(s,X_{s}^{\barX_{t_{k+1}} + s_k \Delta e_k, t_{k+1}}) \partial_x X_{s}^{\barX_{t_{k+1}} + s_k \Delta e_k, t_{k+1}} dW_s.
\end{multline*}
We next introduce the $\sigma$-algebra
\[
\hatF^{k,n} := \overline{\sigma( \{ W_{s }\}_{ 0 \leq s \leq t_k})  \lor \sigma( \{ W_s-W_{t_{k+1} }\}_{ t_{k+1} \leq s \leq t_n} )
\lor \sigma( \{ W_s-W_{t_{n+1} }\}_{ t_{n+1} \leq s \leq T} ) \lor \sigma(X_0)},
\] 
and It{\^o}--Taylor expand the $\varphi_x$ functions in~\eqref{eq:boundDb2} 
about center points that are $\hatF^{k,n}$-measurable:
\begin{multline}\label{eq:crossTerm1}
\phiX{t_{k+1}, \barX_{t_{k+1}} + s_k \Delta e_k  } =
\phiX{t_{n+1}, X_{t_{n+1}}^{ \barX_{t_{k+1}} + s_k \Delta e_k,t_{k+1}  }} 
\partial_x X_{t_{n+1}}^{\barX_{t_{k+1}} + s_k \Delta e_k, t_{k+1}} \\
=  \Bigg[\phiX{ t_{n+1}, X_{t_{n}}^{\barX_{t_{k}},t_{k+1}}} + \phiXX{t_{n+1}, X_{t_{n}}^{\barX_{t_{k}},t_{k+1}}} 
\parenthesis{X_{t_{n+1}}^{ \barX_{t_{k+1}} + s_k \Delta e_k, t_{k+1}  } - X_{t_{n}}^{\barX_{t_{k}},t_{k+1}}} \\
+ \phiXXX{t_{n+1}, X_{t_{n}}^{\barX_{t_{k}},t_{k+1}}} 
\frac{\parenthesis{X_{t_{n+1}}^{ \barX_{t_{k+1}} + s_k \Delta e_k, t_{k+1}  } - X_{t_{n}}^{\barX_{t_{k}},t_{k+1}}}^2}{2} \\
 + \phiXXXX{t_{n+1}, (1-\check{s}_n) X_{t_{n}}^{\barX_{t_{k}},t_{k+1}} + \check{s}_n X_{t_{n+1}}^{ \barX_{t_{k+1}} + s_k \Delta e_k, t_{k+1}  }}\\
\times \frac{(X_{t_{n+1}}^{ \barX_{t_{k+1}} + s_k \Delta e_k, t_{k+1}  } - X_{t_{n}}^{\barX_{t_{k}},t_{k+1}})^2}{2}\Bigg]\\
\times\Bigg[ \partial_x X_{t_{n}}^{\barX_{t_{k}} , t_{k+1}}
+ \partial_{xx} X_{t_{n}}^{\barX_{t_{k}} , t_{k+1}} (a(t_k,\barX_{t_k}) \Delta t_k + b(t_k,\barX_{t_k}) \Delta W_k + s_k \Delta e_k)\\
+ \partial_{xxx} X_{t_{n}}^{\barX_{t_{k}} + \grave{s}_k (a(t_k,\barX_{t_k}) \Delta t_k + (b(t_k,\barX_{t_k}) \Delta W_k + s_k \Delta e_k) , t_{k+1}}\\
\times \frac{(a(t_k,\barX_{t_k}) \Delta t_k +b(t_k,\barX_{t_k}) \Delta W_k + s_k \Delta e_k)^2}{2}\\
+ \int_{t_n}^{t_{n+1}} a_x(s,X_{s}^{\barX_{t_{k+1}} + s_k \Delta e_k, t_{k+1}}) \partial_x X_{s}^{\barX_{t_{k+1}} + s_k \Delta e_k, t_{k+1}} ds \\
+ \int_{t_n}^{t_{n+1}} b_x(s,X_{s}^{\barX_{t_{k+1}} + s_k \Delta e_k, t_{k+1}}) \partial_x X_{s}^{\barX_{t_{k+1}} + s_k \Delta e_k, t_{k+1}} dW_s
\Bigg],
\end{multline}
where 
\begin{multline*}
X_{t_{n+1}}^{ \barX_{t_{k+1}} + s_k \Delta e_k, t_{k+1}  } - X_{t_{n}}^{\barX_{t_{k}},t_{k+1}}\\
= \int_{t_n}^{t_{n+1}} a(s,X_{s}^{\barX_{t_{k+1}} + s_k \Delta e_k, t_{k+1}})  ds 
+\int_{t_n}^{t_{n+1}} b(s,X_{s}^{\barX_{t_{k+1}} + s_k \Delta e_k, t_{k+1}}) dW_s \\
+ \partial_x X_{t_{n}}^{\barX_{t_{k}} + \tilde{s}_k(a(t_k,\barX_{t_k}) \Delta t_k+ b(t_k, \barX_{t_k}) \Delta W_k + s_k \Delta e_k ) ,t_{k+1}} (a(t_k,\barX_{t_k}) \Delta t_k+ b(t_k, \barX_{t_k}) \Delta W_k + s_k \Delta e_k ),
\end{multline*}
and
\begin{multline}\label{eq:crossTerm2}
\phiX{t_{n+1}, \barX_{t_{n+ 1}} + s_n \Delta e_n }  =  \phiX{t_{n+1}, \barX_{t_{n}}^{\barX_{t_{k}}, t_{k+1} }}\\
+ \phiXX{ t_{n+1}, \barX_{t_{n}}^{\barX_{t_{k}}, t_{k+1} } }\Delta \nu_{k,n}
+ \phiXXX{ t_{n+1}, \barX_{n}^{\barX_{k}, t_{k+1} } }\frac{\Delta \nu_{k,n}^2}{2}\\
+ \phiXXXX{ t_{n+1}, (1-\acute{s}_n)\barX_{t_n}^{\barX_{t_k}, t_{k+1} } + \acute{s}_n(\barX_{t_{n+1}} +s_n \Delta e_n)}
\frac{\Delta \nu_{k,n}^3}{6},
\end{multline}
with 
\begin{multline*}
\Delta \nu_{k,n} :=  a( t_{n}, \barX_{t_n}) \Delta t_n +  b( t_{n}, \barX_{t_n}) \Delta W_n +s_n \Delta e_n\\
+ \partial_x \barX_{t_{n}}^{\barX_{t_{k}}+\hat{s}_k( a(t_{k},\barX_{t_{k}})\Delta t_k + b(t_{k},\barX_{t_{k}})\Delta W_k),t_{k+1}} 
(a(t_k,\barX_{t_k}) \Delta t_k+ b(t_k, \barX_{t_k}) \Delta W_k + s_k \Delta e_k ).
\end{multline*}
Plugging the expansions~\eqref{eq:crossTerm1} and~\eqref{eq:crossTerm2} 
into the expectation
\begin{equation*}\label{eq:crossTermFull}
\Exp{ \phiX{t_{k+1}, \barX_{k+1} + s_k \Delta e_k  } \phiX{t_{n+1}, \barX_{n+1} + s_n \Delta e_n} \tildeDb_k \tildeDb_{n} },  
\end{equation*}
the summands in the resulting expression that
only contains products of the first variations vanishes,
\[
\begin{split}
&\Exp{\phiX{ t_{n+1}, X_{t_{n}}^{\barX_{t_{k}},t_{k+1}}}
    \partial_x X_{t_{n}}^{\barX_{t_{k}} , t_{k+1}} \phiX{
      t_{n+1}, \barX_{t_{n}}^{\barX_{t_{k+1}}, t_{k+1} }} \tildeDb_k
    \tildeDb_{n} }\\ & = \Exp{ \Exp{ \tildeDb_{n} \tildeDb_k |
      \hatF^{k,n} } \phiX{ t_{n+1},
      X_{t_{n}}^{\barX_{t_{k}},t_{k+1}}} \partial_x
    X_{t_{n}}^{\barX_{t_{k}} , t_{k+1}} \phiX{ t_{n+1},
      \barX_{t_{n}}^{\barX_{t_{k}}, t_{k+1} }} } = 0.
\end{split}
\]
One can further deduce that all of the the summands 
in which the product of multiple \Ito integrals
$\tildeDb_k$ and $\tildeDb_{n}$ are multiplied only with one 
additional \Ito integral of first-order vanish by
using the fact that the inner product of the resulting multiple \Ito integrals 
is zero, cf.~\cite[Lemma 5.7.2]{Kloeden92}, and by 
separating the first and second variations from
the \Ito integrals by taking a conditional expectation with respect to the
suitable filtration. We illustrate this with a couple of examples,
\begin{multline*}
\mathrm{E}\Bigg[ \phiX{ t_{n+1}, X_{t_{n}}^{\barX_{t_{k}},t_{k+1}}} 
\partial_{xx} X_{t_{n}}^{\barX_{t_{k}} , t_{k+1}} b(t_k,\barX_{t_k}) \Delta W_k    
\phiX{ t_{n+1}, \barX_{t_{n}}^{\barX_{t_{k}}, t_{k+1} }} \tildeDb_k \tildeDb_{n}  \Bigg]\\
 = \mathrm{E}\Bigg[ \phiX{ t_{n+1}, X_{t_{n}}^{\barX_{t_{k}},t_{k+1}}} 
\partial_{xx} X_{t_{n}}^{\barX_{t_{k}} , t_{k+1}} b(t_k,\barX_{t_k}) \Delta W_k    
\phiX{ t_{n+1}, \barX_{t_{n}}^{\barX_{t_{k}}, t_{k+1} }} \tildeDb_k
\\
\times \Exp{\tildeDb_{n} | \hatF^n}  \Bigg]=0,
\end{multline*}
and
\begin{multline*}
\mathrm{E}\Bigg[ \phiX{ t_{n+1}, X_{t_{n}}^{\barX_{t_{k}},t_{k+1}}} 
\partial_{x} X_{t_{n}}^{\barX_{t_{k}} , t_{k+1}} b( t_{n}, \barX_{t_n}) \Delta W_n    
\phiX{ t_{n+1}, \barX_{t_{n}}^{\barX_{t_{k}}, t_{k+1} }}  \tildeDb_k \tildeDb_{n}  \Bigg]\\
 = \mathrm{E}\Bigg[ \phiX{ t_{n+1}, X_{t_{n+1}}^{\barX_{t_{k}},t_{k+1}}} 
\phiX{ t_{n+1}, \barX_{t_{n}}^{\barX_{t_{k}}, t_{k+1} }} \tildeDb_k  b( t_{n}, \barX_{t_n}) 
\Exp{\tildeDb_{n}  \Delta W_n | \hatF^n}  \Bigg]=0.
\end{multline*}
From these observations, assumption (M.3), inequality~\eqref{eq:multiItoBounds},
and, when necessary, additional expansions of integrands to render the
leading order integrand either $\hatF^{k}$- or $\hatF^{n}$-measurable and thereby sharpen the bounds (an example of
such an expansion is
\begin{align*}
\tildeDb_n &= \int_{t_{n}}^{t_{n+1}} \int_{t_{n}}^t (b_xb)(s,X_{s}^{\barX_{t_n},t_{n}}) dW_s \,dW_t\\
&= \int_{t_{n}}^{t_{n+1}} \int_{t_{n}}^t (b_xb)\parenthesis{s,X_{s}^{\barX_{t_n}^{\barX_{t_{k}},t_{k+1}},t_{n}}} dW_s \,dW_t + \text{h.o.t.}).
\end{align*}
We derive after a laborious computation which we will not include here that
\begin{multline*}
\abs{\Exp{\phiX{t_{k+1}, \barX_{t_{k+1}} + s_k \Delta e_k  } \phiX{t_{n+1}, \barX_{t_{n+1}} + s_n \Delta e_n} \tildeDb_k \tildeDb_{n} }}
 \leq C \check{N}^{-3/2} \sqrt{\Exp{\Delta t_{k}^2} \Exp{\Delta t_{n}^2}}.
\end{multline*}
This further implies that
\[
\begin{split}
&\sum_{k,n=0, k\neqq n}^{\check{N}-1} \Exp{ \phiX{t_{k+1}, \barX_{t_{k+1}} + s_k \Delta e_k  } \phiX{t_{n+1}, \barX_{t_{n+1}} + s_n \Delta e_n} \tildeDb_k \tildeDb_{n} }\\
&\leq C \check{N}^{-3/2} \sum_{k,n=0, k\neqq n}^{\check{N}-1} \sqrt{\Exp{\Delta t_{k}^2} \Exp{\Delta t_{n}^2}} \\
&\leq C \check{N}^{-3/2} \parenthesis{\sum_{n=0}^{\check{N}-1} \sqrt{\Exp{\Delta t_{n}^2}}}^2 \\
&\leq C \check{N}^{-1/2 }\sum_{n=0}^{\check{N}-1} \Exp{\Delta t_{n}^2},
\end{split}
\]
such that inequality~\eqref{eq:boundDb2} holds.

So far, we have shown that 
\begin{multline}\label{eq:boundTildeDb2}
 \Exp{ \parenthesis{ \sum_{n=0}^{N-1} \phiX{t_{n+1},\barX_{t_{n+1}} + s_n \Delta e_n } \tildeDb_{n} }^2 } \\
 = \Exp{ \sum_{n=0}^{N-1} \phiX{t_{n+1}, \barX_{t_{n}}}^2  \frac{(b_x b)^2}{2}(t_n,\barX_{t_n})  \Delta t_n^2  + \littleO{\Delta t_n^2} }.
\end{multline}
The MSE contribution from the other local error terms, $\checkDa_n,
\tildeDa_n$ and $\checkDb_n$, can also be bounded using the above
approach with It{\^o}--Taylor expansions, $\hatF^{m,n}$-conditioning
and \Ito isometries. This yields that
\begin{equation}\label{eq:boundCheckDa}
\begin{split}
& \Exp{ \phiX{t_{k+1}, \barX_{t_{k+1}} + s_k \Delta e_k }  \phiX{t_{n+1}, \barX_{t_{n+1}} + s_n \Delta e_n } \checkDa_{k} \checkDa_{n}} \\
& = \mathrm{E} \Bigg[ \phiX{\barX_{t_{k}},t_{k}} \phiX{t_{n},\barX_{t_{n}}} \Big( \frac{a_t+ a_x a + a_{xx} b^2/2 }{2}\Big)(t_{k}, \barX_{t_k}) \times\\
& \qquad \Big( \frac{ a_t+ a_x a + a_{xx} b^2/2  }{2}\Big) (t_{n}, \barX_{t_n})  \Delta t_{k}^2 \Delta t_{n}^2  +\littleO{\Delta t_k^2 \Delta t_n^2}\Bigg],
\end{split}
\end{equation}
\begin{equation*}\label{eq:boundTildeDa*} 
\begin{split}
& \Exp{ \phiX{t_{k+1}, \barX_{t_{k+1}} + s_k \Delta e_k }  \phiX{t_{n+1}, \barX_{t_{n+1}} + s_n \Delta e_n } \tildeDa_{k} \tildeDa_{n}} \\
 & \qquad = \begin{cases}
		\Exp{ \phiX{t_{n}, \barX_{t_{n}} }^2 \frac{ (a_x b)^2 }{2}(t_n, \barX_{t_n}) \Delta t_n^3  + \littleO{\Delta t_n^3} }, & \text{if } k=n,\\
	        \bigO{\check{N}^{-3/2}\parenthesis{\Exp{\Delta t_k^3} \Exp{\Delta t_n^3} }^{1/2}}, & \text{if } k \neqq n,
	    \end{cases}
\end{split}
\end{equation*}
and 
\begin{equation*}\label{eq:boundCheckDb*}
\begin{split}
& \Exp{ \phiX{t_{k+1}, \barX_{t_{k+1}} + s_k \Delta e_k }  \phiX{t_{n+1}, \barX_{t_{n+1}} + s_n \Delta e_n } \checkDb_{k} \checkDb_{n}} \\
& \qquad = \begin{cases}
		\Exp{ \phiX{t_{n}, \barX_{t_{n}}}^2 \frac{ (b_t + b_x a + b_{xx}b^2/2)^2 }{3}(t_n, \barX_{t_n}) \Delta t_n^3  + \littleO{\Delta t_n^3} }, & \text{if } k=n,\\
		\bigO{\check{N}^{-3/2}\parenthesis{\Exp{\Delta t_k^3} \Exp{\Delta t_n^3} }^{1/2}}, & \text{if } k \neqq n.
	    \end{cases}
\end{split}
\end{equation*}
Moreover, conservative bounds for error contributions involving
products of different local error terms, e.g., $\checkDa_k
\tildeDb_n$, can be induced from the above bounds and H\"older's
inequality. For example,
\[
\begin{split}
&\abs{\Exp{ \sum_{k,n=0}^{\check{N}-1}  \phiX{t_{k+1}, \barX_{t_{k+1}} + s_k \Delta e_k }  \checkDa_{k} \phiX{t_{n+1}, \barX_{t_{n+1}} + s_n \Delta e_n }  \tildeDb_{n}} }\\
& = \abs{\Exp{ \parenthesis{ \sum_{k=0}^{\check{N}-1}\phiX{t_{k+1}, \barX_{t_{k+1}} + s_k \Delta e_k } \checkDa_{k}}
\parenthesis{ \sum_{k=0}^{\check{N}-1} \phiX{t_{n+1}, \barX_{t_{n+1}} + s_n \Delta e_n }  \tildeDb_{n}} }} \\
& \leq \sqrt{\Exp{ \parenthesis{
      \sum_{k=0}^{\check{N}-1}\phiX{t_{k+1}, \barX_{t_{k+1}} + s_k
        \Delta e_k } \checkDa_{k}}^2}
\Exp{\parenthesis{ \sum_{n=0}^{\check{N}-1} \phiX{t_{n+1}, \barX_{t_{n+1}} + s_n \Delta e_n }  \tildeDb_{n}}^2 }}\\
& = \bigO{\check{N}^{-1/2} \sum_{n=0}^{\check{N}-1} \Exp{\Delta t_n^2}}.
\end{split}
\]


The proof is completed in two replacement steps applied to $\varphi_x$
on the right-hand side of equality~\eqref{eq:boundTildeDb2}. First, we
replace $\phiX{t_{n+1},\barX_{t_{n}}}$ by $\phiX{t_{n}, \barX_{t_n}}$.
Under the regularity assumed in this theorem, the replacement is
possible without introducing additional leading order error terms as
\[
\begin{split}
& \Exp{|\phiX{t_{n+1},\barX_{t_{n}}} -\phiX{t_{n},\barX_{t_{n}}}| }   =
\Exp{\abs{g'(X_T^{\barX_{t_{n}},t_{n+1}})\partial_x X_T^{\barX_{t_{n}},t_{n+1}}  - g'(X_T^{\barX_{t_{n}},t_{n}})\partial_x X_T^{\barX_{t_{n}},t_{n}}} }\\
&\leq \Exp{\abs{(g'(X_T^{\barX_{t_{n}},t_{n+1}}) -
    g'(X_T^{\barX_{t_{n}},t_{n}}))\partial_x
    X_T^{\barX_{t_{n}},t_{n+1}}}} \\
& \qquad +
\Exp{\abs {g'(X_T^{\barX_{t_{n}},t_{n}})(\partial_x X_T^{\barX_{t_{n}},t_{n+1}}- \partial_x X_T^{\barX_{t_{n}},t_{n}})}} \\
&= \bigO{\check{N}^{-1/2}}.
\end{split}
\]
Here, the last equality follows from the assumptions (M.2), (M.3), (R.2), and (R.3), 
and Lemmas~\ref{lem:existUniquePath} and~\ref{lem:numSolConv},
\[
\begin{split}
&\Exp{\abs{\Big(g'(X_T^{\barX_{t_{n}},t_{n+1}}) - g'(X_T^{\barX_{t_{n}},t_{n}})\Big)\partial_x X_T^{\barX_{t_{n}},t_{n+1}}}} \\
& \leq C \sqrt{\Exp{\abs{X_T^{\barX_{t_{n}},t_{n+1}} - X_T^{X_{t_{n+1}}^{\barX_{t_{n}},t_n},t_{n+1}}}^2} \Exp{ \abs{\partial_x X_T^{\barX_{t_{n}},t_{n+1}}}^2} }\\
& \leq C \parenthesis{\Exp{\abs{\partial_x X_T^{(1-s_n)\barX_{t_{n}}+ s_nX_{t_{n+1}}^{\barX_{t_{n}},t_n} ,t_{n+1}}}^4}}^{1/4}\\
& \qquad \times 
\parenthesis{\Exp{\abs{\int_{t_n}^{t_{n+1}} a(s,X_{s}^{\barX_{t_{n}},t_n}) ds + \int_{t_n}^{t_{n+1}} b(s,X_{s}^{\barX_{t_{n}},t_n})dW_s}^4  }}^{1/4}\\
& \leq C \parenthesis{\Exp{\sup_{ t_n\le s \le t_{n+1} } |a(s,X_{s}^{\barX_{t_{n}},t_n})|^4 \Delta t_n^4  + 
 \sup_{ t_n\le s \le t_{n+1} } |b(s,X_{s}^{\barX_{t_{n}},t_n})|^4 \Delta t_n^2}}^{1/4} \\
& = \bigO{\check{N}^{-1/2}},
\end{split}
\]
and that
\[
\begin{split}
& \Exp{\abs {g'(X_T^{\barX_{t_{n}},t_{n}})(\partial_x X_T^{\barX_{t_{n}},t_{n+1}}- \partial_x X_T^{\barX_{t_{n}},t_{n}})}}
\leq C \sqrt{\Exp{\abs{\partial_x X_T^{\barX_{t_{n}},t_{n+1}}- \partial_x X_T^{\barX_{t_{n}},t_{n}}}^2}}\\
& = C \sqrt{\Exp{\abs{\partial_x X_T^{\barX_{t_{n}},t_{n+1}}- \partial_x X_T^{X_{t_{n+1}}^{\barX_{t_{n}},t_n},t_{n+1}}\partial_x X_{t_{n+1}}^{\barX_{t_{n}},t_n}}^2}}\\
& \leq C\Bigg(\sqrt{\Exp{\abs{\partial_x X_T^{\barX_{t_{n}},t_{n+1}}- \partial_x X_T^{X_{t_{n+1}}^{\barX_{t_{n}},t_n},t_{n+1}}}}}\\
& \qquad  + 
\sqrt{\Exp{\abs{\partial_x X_T^{X_{t_{n+1}}^{\barX_{t_{n}},t_n},t_{n+1}} \parenthesis{\int_{t_n}^{t_{n+1}} a_x(s,X_{s}^{\barX_{t_{n}},t_n}) ds + \int_{t_n}^{t_{n+1}} b_x(s,X_{s}^{\barX_{t_{n}},t_n})dW_s}}^2}}\Bigg)\\
& \leq C\sqrt{\Exp{\abs{\partial_{xx} X_T^{(1-\hat s_n) \barX_{t_{n}} + \hat{s}_n X_{t_{n+1}}^{\barX_{t_{n}},t_n},t_{n+1}}\parenthesis{\int_{t_n}^{t_{n+1}} a_x(s,X_{s}^{\barX_{t_{n}},t_n}) ds + \int_{t_n}^{t_{n+1}} b_x(s,X_{s}^{\barX_{t_{n}},t_n})dW_s}}^2}}\\
& \qquad + \bigO{\check{N}^{-1/2}} \\
& = \bigO{\check{N}^{-1/2}}.
\end{split}
\]
The last step is to replace the first variation of the exact path $\phiX{t_{n},\barX_{t_{n}}}$
with the first variation of the numerical solution $\phiBarX{n} = g'(\barX_T) \partial_x
\barX_T^{\barX_{t_{n}},t_{n}}$. This is also possible without
introducing additional leading order error terms by the same assumptions and
similar bounding arguments as in the two preceding bounds as
\begin{equation*}\label{eq:phiBarErrorBound}
\begin{split}
& \Exp{\abs{\phiBarX{n} -\phiX{t_{n},\barX_{t_{n}}}} }  =
\Exp{ \abs{ g'(\barX_T) \partial_x \barX_T^{\barX_{t_{n}},t_{n}}  - g'(X_T^{\barX_{t_{n}},t_{n}}) \partial_x X_T^{\barX_{t_{n}},t_{n}} } }\\
& \leq \Exp{  |g'(\barX_T)|\abs{\partial_x \barX_T^{\barX_{t_{n}},t_{n}}  -  \partial_x X_T^{\barX_{t_{n}},t_{n}} } } + 
\Exp{ \abs{ g'(\barX_T) - g'(X_T^{\barX_{t_{n}},t_{n}})} \abs{ \partial_x X_T^{\barX_{t_{n}},t_{n}} } }\\
& = \bigO{\check{N}^{-1/2}}.
\end{split}
\end{equation*}

\end{proof}

\subsection{Second, Third and Fourth Variations}\label{subsec:higherOrderVariations}

The proof of Theorem \ref{thm:mseExpansion1D} relies on
bounded moments of higher order variations of the flow map
$\varphi$. In this section, we we will verify that these higher
order variations are indeed well defined random variables with
all required moments bounded.

To this end, we define the following set of coupled SDE
\begin{equation}\label{eq:extendedSDE}
\begin{split}
d\yy{1}{u} =& a (u,\yy{1}{u}) du  + b(u,\yy{1}{u}) dW_u,\\
d\yy{2}{u} =& a_x (u,\yy{1}{u}) \yy{2}{u} du + b_x (u,\yy{1}{u}) \yy{2}{u} dW_u,\\
d\yy{3}{u} = & \parenthesis{a_{xx} (u,\yy{1}{u}) \parenthesis{\yy{2}{u}}^2 + a_x(u,\yy{1}{u}) \yy{3}{u}} du  \\
& \quad + \parenthesis{b_{xx} (u,\yy{1}{u}) \parenthesis{\yy{2}{u}}^2 + b_x(u,\yy{1}{u}) \yy{3}{u}} dW_u,\\
d\yy{4}{u} =& \parenthesis{a_{xxx} (u,\yy{1}{u}) \parenthesis{\yy{2}{u}}^3 
+ 3a_{xx} (u,\yy{1}{u}) \yy{2}{u} \yy{3}{u}  +  a_x(u,\yy{1}{u}) \yy{4}{u}} du  \\
& \quad +
\parenthesis{ b_{xxx} (u,\yy{1}{u}) \parenthesis{\yy{2}{u}}^3 + 3b_{xx} (u,\yy{1}{u}) \yy{2}{u} \yy{3}{u}  
+  b_x(u,\yy{1}{u}) \yy{4}{u}} dW_u,\\
d\yy{5}{u} = & 
\parenthesis{a_{xxxx}( u,\yy{1}{u}) \parenthesis{\yy{2}{u}}^4 + 6 a_{xxx} (u,\yy{1}{u}) \parenthesis{\yy{2}{u}}^2 \yy{3}{u}} du
\\
&\quad + \parenthesis{ a_{xx} (u,\yy{1}{u}) \parenthesis{3 \parenthesis{\yy{3}{u}}^2 + 4 \yy{2}{u}\yy{4}{u}}
+  a_x(u,\yy{1}{u}) \yy{5}{u} }du
\\
&\quad  +
\parenthesis{ b_{xxxx} (u,\yy{1}{u}) \parenthesis{\yy{2}{u}}^4 + 6 b_{xxx} (u,\yy{1}{u}) \parenthesis{\yy{2}{u}}^2 \yy{3}{u}} dW_u
\\
 & \quad + \parenthesis{ b_{xx} (u,\yy{1}{u}) \parenthesis{3 \parenthesis{\yy{3}{u}}^2 + 4 \yy{2}{u} \yy{4}{u}}  +  b_x(u,\yy{1}{u}) \yy{5}{u} } dW_u,
\end{split}
\end{equation}
defined for $u \in (t,T]$ with the initial condition $Y_t = (x,1,0,0,0)$.
The first component of the vector coincides with equation \eqref{eq:sdeFlow},
whereas the second one is the first variation of the path from equation \eqref{eq:sdeFlowFirstVar}.
The last three components can be understood as the second, third and fourth variations
of the path, respectively.

Making use of the solution of SDE \eqref{eq:extendedSDE}, we also define the second, third and
fourth variations as
\begin{align}
\phiXX{t,x} &= g ' (X^{x,t}_T) \partial_{xx}X^{x,t}_T + g'' (X^{x,t}_T) (\partial_{x}X^{x,t}_T)^2,
\nonumber
\\
\phiXXX{t,x} &= g ' (X^{x,t}_T) \partial_{xxx }X^{x,t}_T + \dots + g''' (X^{x,t}_T) (\partial_{x}X^{x,t}_T)^3,
\label{eq:varDef}
\\
\phiXXXX{t,x} &= g ' (X^{x,t}_T) \partial_{xxxx}X^{x,t}_T + \dots + g'''' (X^{x,t}_T) (\partial_{x}X^{x,t}_T)^4.
\nonumber
\end{align}

In the sequel, we prove that the solution to equation \eqref{eq:extendedSDE} when
understood in the integral sense that extends \eqref{eq:sdeFlow} is a well defined
random variable with bounded moments. Given sufficient differentiability of the
payoff $g$, this results in the boundedness of the higher order variations as required in 
Theorem~\ref{thm:mseExpansion1D}.

\begin{lemma}\label{lem:existenceUniquenessHigherVariations}

Assume that (R.1), (R.2), and (R.3) in Theorem~\ref{thm:mseExpansion1D}
hold and that for any fixed
$t \in [0,T]$ and $x\in \FCal_t$
such that $\Exp{ \abs{x}^{2p}} < \infty$ for all $ p \in \N$.
Then, equation \eqref{eq:extendedSDE} has pathwise unique
solutions with finite moments. That is,
\begin{equation*}
\max_{i \in \{1,2,\ldots,5\}} \parenthesis{\sup_{u \in [t,T]}\Exp{ \abs{\yy i u}^{2p}}}
< \infty, \qquad  \forall p \in \mathbb N.
\end{equation*}
Furthermore, the higher variations as defined by
equation~\eqref{eq:varDef} satisfy
\begin{equation*}
\phiX{t,x},\phiXX{t,x},\phiXXX{t,x},\phiXXXX{t,x}
\in \FCal_T
\end{equation*}
and for all $p \in \N$,
\begin{equation*}
\mathrm{max} 
\left\{
\Exp{ \abs{\phiX{t,x}}^{2p}}, \Exp{\abs{\phiXX{t,x}}^{2p}},\Exp{\abs{\phiXXX{t,x}}^{2p}},\Exp{\abs{\phiXXXX{t,x}}^{2p}}
\right\}
<
\infty.
\end{equation*}
\end{lemma}

\begin{proof}
We note that the system of SDE \eqref{eq:extendedSDE} can be trivially
truncated to its first $d_1 \leq 5$ elements. That is, the truncated
SDE for $\{\yy{j}{u}\}_{j=1}^{d_1}$ for $d_1 < 5$ has drift and diffusion
functions $\hat a: [0,T] \times \mathbb R^{d_1} \rightarrow \mathbb R^{d_1}$ and
$\hat b: [0,T] \times \mathbb R^{d_1} \rightarrow \mathbb R^{d_1\times d_2}$ that do
not depend on $\yy{j}{u}$ for $j \geq d_1$.

This enables verifying existence of solutions for the SDE in stages:
first for $(\yy{1}{}, \yy{2}{})$, thereafter for $(\yy{1}{},
\yy{2}{}, \yy{3}{})$, and so forth, proceeding iteratively to add the
next component $\yy{d_1+1}{}$ to the SDE. We shall also
exploit this structure for proving the result of bounded moments for
each component. The starting point for our proof is
Lemma~\ref{lem:existUniquePath}, which guarantees existence,
uniqueness and the needed moment bounds for the first two components
$\yy{1}{}$, and $\yy{2}{}$. It will turn out in the sequel that,
thanks to the regularity in the drift and diffusion functions of the
SDE~\eqref{eq:extendedSDE}, this regularity will cascade further to
$\yy{j}{}$ for $j \in \{3,4,5\}$.

Kloeden and Platen~\cite[Theorems 4.5.3 and 4.5.4]{Kloeden92} note
that their existence and uniqueness theorems for SDE cannot be
modified in order to account for looser regularity conditions, and the
proof below is a case in point. Our approach here follows closely the
presentation of Kloeden and Platen, with slight modifications on the
inequalities that are used to achieve bounds at various intermediate
stages of the proof.

As a beginning stage for the proof, let us note that 
Theorem~\cite[Thorem 5.2.5]{Karatzas91} guarantees that the solutions of
\eqref{eq:extendedSDE} are pahtwise unique and focus on verifying
the claimed results for $\yy{3}{u}$. 

We define a successive set of approximations $\yyi{3}{n}{u}$, $n \in
\mathbb N$ by
\begin{multline*}
\yyi{3}{n+1}{u} 
=
\int_t^u a_{xx} (s,\yy{1}{s}) \parenthesis{\yy{2}{s}}^2 + a_x(s,\yy{2}{s}) \yyi{3}{n}{s}  ds  
\\+
\int_t^u b_{xx} (s,\yy{1}{s}) \parenthesis{\yy{2}{s}}^2 + b_x(s,\yy{2}{s}) \yyi{3}{n}{s}  dW_s,
\end{multline*}
with the initial approximation defined by $\yyi{3}{1}{u} = 0$, for all $u \in [t,T]$. 
Let us denote by
\begin{equation}
Q=\int_t^u a_{xx} (s,\yy{1}{s}) \parenthesis{\yy{1}{s}}^2  ds + \int_t^u b_{xx} (s,\yy{1}{s}) \parenthesis{\yy{2}{s}}^2  dW_s
\label{eq:triangleSeparation}
\end{equation}
the terms that do not depend on the, for the time being, highest order
variation $\yyi{3}{n}{u}$.
We then have, using Young's inequality, that
\begin{align*}
\Exp{\abs{\yyi{3}{n+1}{u}}} \leq &
3 \Exp{\abs{Q}^2}
+
3 \Exp{\abs{\int_t^u  a_x(s,\yy{1}{s}) \yyi{3}{n}{s} ds  }^2} 
+
3 \Exp{\abs{\int_t^u b_x(s,\yy{1}{s}) \yyi{3}{n}{s} dW_s }^2} 
\\
\leq &
3 \Exp{\abs{Q}^2}
+
3 (u-t) \Exp{\int_t^u \abs{ a_x(s,\yy{1}{s}) \yyi{3}{n}{s}  }^2 ds } 
+
3 \Exp{\int_t^u \abs{b_x(s,\yy{1}{s}) \yyi{3}{n}{s}  }^2 ds} .
\end{align*}
The boundedness of the partial derivatives
of the drift and diffusion terms in~\eqref{eq:sdeFlow} gives us
\begin{align*}
\Exp{\abs{\yyi{3}{n+1}{u}}^2} \leq & 3 \Exp{\abs{Q}^2} + C (u-t+1)
\Exp{\int_t^u \parenthesis{1+ \abs{\yyi 3 n s}^2 } ds }.
\end{align*}
By induction, we consequently obtain that
\begin{align*}
\sup_{t\leq u \leq T} \Exp{\abs{\yyi{3}{n}{u}}^2} < \infty, \qquad \forall n \in \N.
\end{align*}
Now, we set $\ydi{3}{n}{u} = \yyi{3}{n+1}{u} -\yyi{3}{n}{u}$. Then
\begin{align*}
\Exp{\abs{\ydi 3 n u}} \leq & 2 \Exp{\abs{ \int_t^u  a_x(s,\yy{1}{s})  \ydi{3}{n-1}{s}  ds }^2 } 
+
2 \Exp{\abs{ \int_t^u  b_x(s,\yy{1}{s}) \ydi {3}{n-1}s dW_s }^2 } 
\\
\leq&
2(u-t) \int_t^u \Exp{\abs{a_x(s,\yy{1}{s}) \ydi {3}{n-1}s }^2} ds
+
2\int_t^u \Exp{\abs{b_x(s,\yy{1}{s})  \ydi {3}{n-1}s }^2} ds
\\
\leq &
 C_1 \int_t^u \Exp{\abs{ \ydi{3}{n-1}{s}}^2} ds.
\end{align*}
Thus, by Gr\"onwall's inequality, 
\begin{align*}
\Exp{\abs{ \ydi 3 n u}^2} \leq \frac{C_1^{n-1}}{(n-1)!} \int_t^u (u-s)^{n-1} \Exp{\abs{ \ydi 3 1 s }^2} ds .
\end{align*}
Let us next show that $\Exp{\abs{ \ydi 3 1 s }^2}$ is bounded. First,
\begin{equation*}
\begin{split}
\Exp{\abs{\ydi 3 1 u }^2} 
&= \Exp{\abs{\int_t^u a_{x} (s,\yy{1}{s})\yyi{3}{2}{s} ds + \int_t^u b_{x}(s,\yy{1}{s} )\yyi{3}{2}{u} dW_s }^2}\\
&\leq  C(u-t + 1) \sup_{s \in[t, u]} \Exp{\abs{\yyi{3}{2}{s}}^2}.
\end{split}
\end{equation*}
Consequently, there exists a $C \in \R$ such that 
\begin{align*}
\Exp{\abs {\ydi 3 n u }^2} \leq \frac{C^{n} (u-t)^n}{n!}, \qquad 
\sup_{u \in [t, T]}  \Exp{\abs {\ydi 3 n  u}^2}  \leq \frac{C^{n} (T-t)^n}{n!}.
\end{align*}
We define
\begin{align*}
Z_n = \sup_{t \leq u \leq T} \abs{ \ydi 3  n u },
\end{align*}
and note that 
\begin{align*}
Z_n \leq &\int_t^T \abs{a_{x}(s,\yy  1 s) \yyi 3 {n+1} s - a_{x}(s,\yy 1 s) \yyi{3}{n}{s} } ds
\\
&+ \sup_{t\leq u \leq T} \abs{\int_t^u b_{x}(s,\yy 1  s) \yyi 3 {n+1} s - b_{x}(s,\yy 1 s) \yyi{3}{n}{s}  dW_s}.
\\
\end{align*}
Using Doob's and Schwartz's inequalities, as well as the
boundedness of $a_x$ and $b_x$,
\begin{align*}
\Exp{\abs{Z_n}^2}
\leq 
&
2 (T-t) \int_t^T \Exp{ \abs{a_{x}(s,\yy 1 s) \yyi 3 {n+1} s - a_{x}(s,\yy 1 s) \yyi{3}{n}{s}}^2 }ds
\\
&+ 8 \int_t^T \Exp{ \abs{b_{x}(s,\yy 1 s)\yyi 3 {n+1} s  - b_{x}(s,\yy 1 s) \yyi{3}{n}{s}} ^2 } ds
\\
\leq 
&
\frac{ C^{n} (T-t)^{n}}{n!},
\end{align*}
for some $C \in \R$.
Using the Markov inequality, we get
\begin{align*}
\sum_{n=1}^\infty  \Prob{ Z_n > n^{-2}} \leq \sum_{n=1}^{\infty} \frac{n^4  C^{n} (T-t)^{n}}{n!}.
\end{align*}
The right-hand side of the equation above converges by the ratio test,
whereas the Borel-Cantelli Lemma guarantees the (almost sure) existence of $K^*\in \N$,
such that $Z_k < k^2, \forall k>K^*$.
We conclude that $\yyi 3 n u$ converges uniformly in $L^2(\mathrm{P})$ 
to the limit $\yy 3 u = \sum_{n=1}^{\infty} \ydi 3 nu$ 
and that since $\{\yyi 3 n u\}_n$ is a sequence of continuous and $\FCal_u$-adapted processes, $\yy 3 u$ is 
also continuous and $\FCal_u$-adapted.
Furthermore, as $n\to \infty$,
\begin{align*}
\abs{\int_t^u a_{x}(s,\yy{1}{s}) \yyi 3 n s ds - \int_t^u a_{x}(s,\yy{1}{s}) \yy 3  s ds }
\leq C \int_t^u \abs{\yyi 3 n s -\yy 3 s} ds \rightarrow 0, \quad \text{a.s.},
\end{align*}
and, similarly,   
\[
\abs{\int_t^u b_{x}(s,\yy{1}{s}) \yyi 3 n s dW_s - \int_t^u b_{x}(s,\yy{1}{s}) \yy 3  s dW_s } \to 0, \quad \text{a.s.}
\]
This implies that $\yy 3 u$ is a solution to the SDE \eqref{eq:extendedSDE}.

Having established that $\yy 3 u$ solves the
relevant SDE and that it has a finite second moment, we may follow the principles laid out 
in  \cite[Theorem 4.5.4]{Kloeden92} and show that all even moments of 
\begin{align*}
\yy 3 u = Q + \int_t^u  a_x (t,\yy 1 s) \yy 3 s ds +  \int_t^u  b_x (t,\yy 1 s) \yy 3 s dW_s
 \end{align*}
are finite. By It\^o's Lemma, we get that for any even integer $l$,
\begin{align*}
\abs{\yy 3 u}^{l}
= &
\int_t^u
\abs{\yy 3 s}^{l-2} \yy 3 s \parenthesis{a_{xx}(s,\yy{1}{s} ) \parenthesis{\yy{2}{s}}^2 + a_x(s,\yy{1}{s}) \yy 3 s }
ds
\\
&+
\int_t^u
\frac{l (l-1)}{2} \abs{\yy 3 s}^{l-2} \parenthesis{b_{xx}(s,\yy{1}{s} ) \parenthesis{\yy{2}{s}}^2 + b_x(s,\yy{1}{s}) \yy 3 s }^2 
ds
\\
&+
\int_t^u
\abs{\yy 3 s}^{l-2} \yy 3 s \parenthesis{b_{xx}(s,\yy{1}{s} ) \parenthesis{\yy{2}{s}}^2 + b_x(s,\yy{1}{s}) \yy 3 s }
dW_s.
\end{align*}
Taking expectations, the It\^o integral vanishes,
\begin{multline*}
\Exp{\abs{\yy 3 s}^{l}}
=
\Exp{
\int_t^u
\abs{\yy 3 s}^{l-2} \yy 3 s \parenthesis{a_{xx}(s,\yy{1}{s} ) \parenthesis{\yy 2 s }^ 2 + a_x(s,\yy{1}{s})\yy 3 s }
ds
}
\\
+
\Exp{
 \int_t^u
 \frac{l (l-1)\abs{\yy 3 s}^{l-2} }{2}
 \parenthesis{ 
 b_{xx}(s,\yy{1}{s} )\parenthesis{\yy 2 s}^2 + b_x(s,\yy{1}{s}) \yy{3}{s} }^2 
ds
}.
\end{multline*}
Using Young's inequality, denoting the term that does not depend on $\yy 3 s$ in the
first integral by $Q$, and exploiting the boundedness of $a_x$, we have that
\begin{align*}
\Exp{
\abs{
\yy 3u
}^l
} - Q
\leq &
C
\int_t^u
\Exp{|Y_{3,u}|^{l} } 
ds
\\
&+
\Exp{
\int_t^u
\frac{l(l-1) \abs{\yy 3 s}^{l-2} }{2} 
\parenthesis{b_{xx}\parenthesis{s, \yy 1 s} \parenthesis{\yy 2 s}^2+ b_x \parenthesis{s,\yy 1 s} \yy 3 s }^2
ds
}.
\end{align*}
By the same treatment for the latter integral, lumping together
all terms independent of $\yy 3 s$ and using that $b_x$ is bounded,
\begin{align*}
\Exp{\abs{\yy 3 u}^l} - \tilde Q
\leq &
C \int_t^u
\Exp{\abs{\yy 3 u}^{l} }
ds.
\end{align*}
Thus, by Gr\"onwall's inequality, 
$\Exp{\abs{\yy 3 u}^{l}} < \infty$.

Having established the existence and pathwise uniqueness of $\yy 3s$,
as well as the finiteness of its moments, verifying the same
properties for $\yy 4 u$ and $\yy 5 u$ can be done using similar
arguments relying, most importantly, on the boundedness of the relevant derivatives
of the drift and diffusion functions.
Finally, the $\FCal_T$-measurability and moment bounds for the 
variations of the flow map $\varphi$ up to order four
(cf.~equation~\eqref{eq:varDef}) can be verified by a straightforward
extension of the argument in the proof of
Lemma~\ref{lem:existUniquePath}.
\end{proof}


\subsection{Error Expansion for the MSE in Multiple Dimensions}\label{subsec:mseHihgerDim}

In this section, we extend the 1D MSE error expansion presented in
Theorem~\ref{thm:mseExpansion1D} to the multi-dimensional setting.


Consider the SDE
\begin{equation}\label{eq:sdeProblemHigerDimensions}
 \begin{split}
dX_t & = a\parenthesis{t, X_t} dt + b\parenthesis{t,X_t} dW_t, \qquad t \in (0,T]\\
X_0 & = x_0,
\end{split}
\end{equation}
where $X: [0,T] \to \mathbb \R^{d_1}$, $W : [0,T] \to \R^{d_2}$, $a:
 [0,T] \times \R^{d_1}  \to \R^{d_1}$ and $b:  [0,T] \times \R^{d_1} \to \R^{d_1 \times d_2}$.  
Let further $x_i$ denote the $i$-th component of 
$x \in \mathbb R^{d_1}$, $a^{(i)}$, the $i$-th component of a
drift coefficient and $b^{(i,j)}$ and $b^\T$ denote the $(i,j)$-th
element and the transpose of the diffusion matrix $b$, respectively. (To avoid
confusion, this derivation does not make use of any MLMC notation,
particularly not the multilevel superscript ${\cdot}^{\{\ell\}}$.)

Using the Einstein summation convention to sum over repeated
indices, but not over the time index $n$, the 1D local error terms 
in equation~\eqref{eq:localErrorTerm} generalize into
\begin{align*}
\checkDA_{n}^{(i)} & =  \int_{t_{n}}^{t_{n+1} } \int_{t_{n}}^t \parenthesis{ a^{(i)}_t +  a_{x_j}^{(i)} a^{(j)} + \frac{1}{2} a_{x_j x_k }^{(i)} (b b^T)^{(j,k)}}  \, ds  \, dt,\\
\tildeDA_n^{(i)}  & =  \int_{t_{n}}^{t_{n+1} } \int_{t_{n}}^t  a_{x_j}^{(i)} b^{(j,k)}   \, dW_s^{(k)}  \, dt, \\
\checkDB_n^{(i)}  & =  \int_{t_{n}}^{t_{n+1} } \int_{t_{n}}^t b_t^{(i,j)} + b_{x_k}^{(i,j)} a^{(k)} + \frac{1}{2} b^{(i,j)}_{x_k x_\ell} (b b^T)^{(k,\ell)}  \, ds \, dW_t^{(j)}, \\
\tildeDB_n^{(i)}  & =  \int_{t_{n}}^{t_{n+1} } \int_{t_{n}}^t  b_{x_k}^{(i,j)} b^{(k,\ell)}   \, dW_s^{(\ell)} \, dW_t^{(j)},
\end{align*}
where all the above integrand functions in all equations implicitly depend on the state argument $X_s^{\barX_{t_n},t_n}$. In flow notation,
$a^{(i)}_{t}$ is shorthand for $a^{(i)}_{t}(s,X_s^{\barX_{t_n},t_n})$.

Under sufficient regularity, a tedious calculation similar to the proof of 
Theorem~\ref{thm:mseExpansion1D}
verifies that, for a given smooth payoff, $g:\R^{d_1} \to \R$,
\begin{equation*}\label{eq:errorExpansionHigherDim}
\Exp{ \parenthesis{\gXT -\gBarXT}^2}  \leq   \Exp{ \sum_{n=0}^{N-1} \rhoBar_n \Delta t_{n}^2  + \littleO{\Delta t_n^2} },
\end{equation*}
where 
\begin{equation}\label{eq:rhoBarHigherDim}
\rhoBar_n := \frac{1}{2} \phiBarXi{x_i}{n}   \parenthesis{ (bb^\T)^{(k, \ell)} (b_{x_k} b_{x_\ell}^\T)}^{(i,j)} (t_n,\barX_{t_n})  \phiBarXi{x_j}{n}.
\end{equation}
In the multi-dimensional setting, the $i$-th component of 
first variation of the flow map, 
$\varphi_x = (\varphi_{x_1}, \varphi_{x_2}, \ldots,
\varphi_{x_{d_1}})$, is given by
\begin{equation*}
\label{eq:phiXSdeHigerDimensions*}
\varphi_{x_i} \parenthesis{t, y} = g_{x_j}(X^{y,t}_T)  \partial_{x_i} \parenthesis{X^{y,t}_T}^{(j)}.
\end{equation*}
The first variation is defined as the second component to the solution of the SDE,
\begin{align*}
\begin{split}
d \yy {1,i} s &= a^{(i)} \parenthesis{s, \yy 1 s} ds + b^{(i,j)} \parenthesis{s, \yy 1 s} d W^{(j)}_s
\\
d \yy {2,i,j} s &= a^{(i)}_{x_k} \parenthesis{s, \yy 1 s}   \yy {2,k,j} s ds + b^{(i,\ell)}_{x_k} \parenthesis{s, \yy 1 s}
 \yy {2,k,j} s dW^{(\ell)}_s,
\end{split}
\end{align*}
where $s \in (t,T]$ and the initial conditions are given by 
$\yy 1 t = x \in \R^{d_1}$, $\yy 2 t = I_{d_1}$,
with $I_{d_1}$ denoting the $d_1\times d_1$ identity matrix. Moreover,
the extension of the numerical method for solving the first variation
of the 1D flow map~\eqref{eq:backwardScheme1D} reads
\begin{align}
\label{eq:bwhd}
\overline \varphi_{x_i,n}  &=  c_{x_i}^{(j)} (t_n,\barX_{t_{n}}) \overline \varphi_{x_j,n+1} , \quad n=N-1,N-2,\ldots 0. \\
 \overline \varphi_{x_i,N}&= g_{x_i}(\barX_T),
 \nonumber
\end{align}
with the $j$-th component of $c: [0,T] \times \R^{d_1}  \to \R^{d_1}$ defined by
\begin{equation*}
c^{(j)}
\parenthesis{t_n, \barX_{t_n}} 
= \barX_{t_n}^{(j)} + a^{(j)}(t_n,\barX_{t_n}) \Delta t_n +  b^{(j,k)}(t_n, \barX_{t_n}) \Delta W^{(k)}_n.
\end{equation*}

Let $U$ and $V$ denote subsets of Euclidean spaces and let us introduce the multi-index $\nu = (\nu_1, \nu_2, \ldots, \nu_d)$ 
to represent spatial partial derivatives of order $|\nu| := \sum_{j=1}^d \nu_j$ on the following short form
$\partial_{x_\nu} := \prod_{j=1}^d \partial_{x_j}^\nu$. We further introduce the following function spaces.
\[
\begin{split}
&C(U; V)          := \{f:U \to V \, |\, f \text{ is continuous} \},\\
&C_b(U;V)          := \{f:U \to V \, |\, f \text{ is continuous and bounded} \},\\
&C_b^k(U;V)        := \Big\{f:U \to V \, |\, f \in C(U;V) \text{ and } 
\frac{d^j}{dx^j} f \in C_b(U;V) \\
& \quad \text{for all integers }  1 \le j \leq k \Big\},\\
&C_b^{k_1,k_2}([0,T] \times U; V):= \Big\{f:[0,T] \times U \to V \, |\, f \in C([0,T] \times U; V), \text{ and } \\
& \quad \partial_t^{j}\partial_\nu f  \in C_b([0,T]\times U; V) 
\text{ for all integers s.t. } j \leq k_1 \text{ and }  1 \le j+\abs{\nu} \leq k_2 \Big\}.
\end{split}
\]

\begin{theorem}[MSE leading order error expansion in the multi-dimensional setting]~\label{thm:mseExpansionMultiD} 
Assume that drift and diffusion coefficients and input data of the
SDE~\eqref{eq:sdeProblemHigerDimensions} fulfill
\begin{enumerate}
\item[(R.1)] $a \in C_b^{2,4}([0,T]\times \R^{d_1}; \R^{d_1})$ and 
$b \in C_b^{2,4}([0,T]\times \R^{d_1}; \R^{d_1 \times d_2})$,

\item[(R.2)] there exists a constant $C>0$ such that
\begin{align*}
|a(t,x)|^2 + |b(t,x)|^2 & \leq C(1+|x|^2), & \forall x \in \R^{d_1} \text{ and } \forall t \in [0,T],\\
\end{align*}

\item [(R.3)] $g' \in C^3_b(\R^{d_1})$ and 
there exists a $k \in \N$ such 
\begin{equation*}\label{eq:growthG}
|g(x)| + |g'(x)| \leq C(1+|x|^k), \qquad \forall x \in \R^{d_1},
\end{equation*}

\item[(R.4)] for the initial data, $X_0 \in \FCal_0$ and $\Exp{|X_0|^p} < \infty$ for all $p\ge 1$.
\end{enumerate}
Assume further the mesh points $0=t_0<t_1< \ldots <t_N = T$ 
\begin{enumerate}
\item[(M.1)] are stopping times for which
$t_n \in \FCal_{t_{n-1}}$ for $n=1,2,\ldots, N$,

\item[(M.2)] for all mesh realizations, there exists a deterministic integer, $\check N$, and a
  $c_1>0$ such that 
$c_1\check{N} \le N \le \check N$ and
a $c_2>0$ such that $\max_{n\in \{0,1,\ldots, N-1\}} \Delta t_n < c_2 \check{N}^{-1}$,

\item[(M.3)] and there exists a $c_3>0$ such that for all $p \in [1,8]$ 
and $n \in \{0,1,\ldots, \check{N}-1\}$,
\[
\Exp{\Delta t_n^{2p}} \leq c_3 \parenthesis{\Exp{ \Delta t_n^2 } }^{p}.
\]  

\end{enumerate}
Then, as $\check{N}$ increases, 
\begin{multline*}
\Exp{\parenthesis{g(X_T) -\gBarXT}^2} 
= \Exp{\sum_{n=0}^{N-1} \frac{\parenthesis{\varphi_{x_i} \parenthesis{ (bb^\T)^{(k, \ell)} (b_{x_k} b_{x_\ell}^\T)}^{(i,j)}   \varphi_{x_j}}(t_n,\barX_{t_n})}{2} \Delta t_n^2 + o(\Delta t_n^2)},
\end{multline*}
where we have dropped the arguments of the first variation as well as the diffusion matrices
for clarity.

Replacing the first variation $\phiXi{x_i}{t_n,\barX_n}$ by the
numerical approximation $\phiBarXi{x_i}{n}$, as defined in
\eqref{eq:bwhd} and using the error density notation $\rhoBar$ 
from~\eqref{eq:rhoBarHigherDim}, we obtain 
the following to leading order all-terms-computable error
expansion:
\begin{align}
\Exp{\parenthesis{g(X_T) -\gBarXT}^2} = \Exp{ \sum_{n=0}^{N-1} \rhoBar_n \Delta t_n^2 + o(\Delta t_n^2)}.
\label{eq:mseExpansionMultiDExact}
\end{align}
\end{theorem}

\section{A Uniform Time Step MLMC Algorithm}\label{sec:uniformMlmcAlg}
  The uniform time step MLMC algorithm for MSE approximations of SDE
  was proposed in~\cite{Giles08}. Below, we present the version of that
  method that we use in the numerical tests in this work for reaching the approximation
  goal~\eqref{eq:mlmcGoal}.

\begin{algorithm}[h!]
\caption{{\bf mlmcEstimator}}
\begin{algorithmic}\label{alg:uniformMlmcEstimator}
  \STATE{\bf Input:} $\tolT$, $\tolS$, confidence $\delta$, input mesh $\Dt{-1}$, input mesh intervals $N_{-1}$, 
    inital number of samples $\widehat M$, weak convergence rate $\alpha$, SDE problem.
  \STATE{\bf Output:} Multilevel estimator $\AMLSimple$.
    
\medskip 

\STATE{Compute the confidence parameter $\Cc(\delta)$ by~\eqref{eq:CcDefinition}.}

\STATE{Set $L=-1$.}

\WHILE{$L<3$ \OR ~\eqref{eq:determineL}, using the input rate $\alpha$, is violated}

  \STATE{Set $L=L+1$.}
  
  \STATE {Set $M_L = \widehat M$, generate a set of (Euler--Maruyama) realizations $\{ \DlGO \}_{i=1}^{M_L}$
  on mesh and Wiener path pairs $(\Dt{L-1}, \Dt{L})$ and $(\WL{L-1}, \WL{L})$, 
  where the uniform mesh pairs have step sizes $\Dt{L-1} = T/N_{L-1}$ and $\Dt{L} = T/N_{L})$, respectively.
  }
  \FOR{$\ell=0$ \TO $L$}
    \STATE{Compute the sample variance $\V{ \DlG ; M_l}$.}
  \ENDFOR

  \FOR{$\ell=0$ \TO $L$}
    \STATE{Determine the number of samples by
    \[
     M_\ell = \roundUp{ \frac{\Cc^2}{\tolS^2} \sqrt{\frac{\Var{\DlG}}{N_\ell}} 
     \sum_{\ell=0}^L \sqrt{N_\ell \Var{\DlG } }  }.
    \]
    (The equation for $M_l$ is derived by Lagrangian 
    optimization, cf.~Section~\ref{subsec:statError}.)}
      \IF{New value of $M_\ell$ is larger than the old value}{
	\STATE{Compute additional (Euler--Maruyama) realizations $\{ \DlGO \}_{i=M_\ell+1}^{M_\ell^{new}}$ 
	  on mesh and Wiener path pairs $(\Dt{\ell-1}, \Dt{\ell})$ and $(\WL{\ell-1}, \WL{\ell})$, 
	  where the uniform mesh pairs have step sizes $\Dt{\ell-1} = T/(2^{\ell} N_{-1})$ and $\Dt{\ell} = T/(2^{\ell+1} N_{-1})$, respectively.
	}}
      \ENDIF
  \ENDFOR
\ENDWHILE

Compute $\AMLSimple$ using the generated samples by the formula~\eqref{eq:MLestimator}.

\end{algorithmic}
\end{algorithm}



\medskip
\par
{\bf Acknowledgments.}

The authors would like to thank Arturo Kohatsu-Higa for his helpful
suggestions for improvements in the proof of
Theorem~\ref{thm:mseExpansion1D}.

\medskip
\par

{\bf Support} This work was supported by King
Abdullah University of Science and Technology (KAUST); by Norges
Forskningsr{\aa}d,  research project 214495 LIQCRY; and by the University of Texas,
Austin Subcontract (Project Number 024550, Center for Predictive
Computational Science).  The first author was and the third author is
a member of the Strategic Research Initiative on Uncertainty
Quantification in Computational Science and Engineering at KAUST
(SRI-UQ).

%
\bibliography{mseAdaptiveMLMC}

\begin{thebibliography}{10}

\bibitem{Avikainen09}
R.~Avikainen.
\newblock On irregular functionals of {SDE}s and the {E}uler scheme.
\newblock {\em Finance and Stochastics}, 13(3):381--401, 2009.

\bibitem{Bangerth03}
W.~Bangerth and R.~Rannacher.
\newblock {\em Adaptive finite element methods for differential equations}.
\newblock Lectures in Mathematics ETH Z\"urich. Birkh\"auser Verlag, Basel,
  2003.

\bibitem{Barth12}
A.~Barth and A.~Lang.
\newblock Multilevel {M}onte {C}arlo method with applications to stochastic
  partial differential equations.
\newblock {\em Int. J. Comput. Math.}, 89(18):2479--2498, 2012.

\bibitem{Cliffe11}
K.~A. Cliffe, M.~B. Giles, R.~Scheichl, and A.~L. Teckentrup.
\newblock Multilevel {M}onte {C}arlo methods and applications to elliptic
  {PDE}s with random coefficients.
\newblock {\em Comput. Vis. Sci.}, 14(1):3--15, 2011.

\bibitem{Collier14}
Nathan Collier, Abdul-Lateef Haji-Ali, Fabio Nobile, Erik von Schwerin, and
  Ra{\'u}l Tempone.
\newblock A continuation multilevel monte carlo algorithm.
\newblock {\em BIT Numerical Mathematics}, 55(2):399--432, 2014.

\bibitem{Durrett96}
R.~Durrett.
\newblock {\em Probability: theory and examples}.
\newblock Duxbury Press, Belmont, CA, second edition, 1996.

\bibitem{Gaines97}
J.~G. Gaines and T.~J. Lyons.
\newblock Variable step size control in the numerical solution of stochastic
  differential equations.
\newblock {\em SIAM J. Appl. Math}, 57:1455--1484, 1997.

\bibitem{Giles08}
M.~B. Giles.
\newblock Multilevel {M}onte {C}arlo path simulation.
\newblock {\em Oper. Res.}, 56(3):607--617, 2008.

\bibitem{Giles14}
M.~B. Giles and L.~Szpruch.
\newblock Antithetic multilevel {M}onte {C}arlo estimation for
  multi-dimensional {SDE}s without {L}\'evy area simulation.
\newblock {\em Ann. Appl. Probab.}, 24(4):1585--1620, 2014.

\bibitem{Giles15}
Michael~B Giles.
\newblock Multilevel monte carlo methods.
\newblock {\em Acta Numerica}, 24:259--328, 2015.

\bibitem{Gillespie00}
D.~T. Gillespie.
\newblock The chemical {L}angevin equation.
\newblock {\em The Journal of Chemical Physics}, 113(1):297--306, 2000.

\bibitem{Glasserman04}
P.~Glasserman.
\newblock {\em Monte {C}arlo methods in financial engineering}, volume~53 of
  {\em Applications of Mathematics (New York)}.
\newblock Springer-Verlag, New York, 2004.
\newblock Stochastic Modelling and Applied Probability.

\bibitem{AbdulLateef14}
Abdul-Lateef Haji-Ali, Fabio Nobile, Erik von Schwerin, and Raúl Tempone.
\newblock Optimization of mesh hierarchies in multilevel monte carlo samplers.
\newblock {\em Stochastic Partial Differential Equations: Analysis and
  Computations}, pages 1--37, 2015.

\bibitem{Heinrich98}
S.~Heinrich.
\newblock Monte {C}arlo complexity of global solution of integral equations.
\newblock {\em J. Complexity}, 14(2):151--175, 1998.

\bibitem{HeinrichSin99}
S.~Heinrich and E.~Sindambiwe.
\newblock Monte {C}arlo complexity of parametric integration.
\newblock {\em J. Complexity}, 15(3):317--341, 1999.

\bibitem{HoelMLMC14}
H.~Hoel, E.~von Schwerin, A.~Szepessy, and R.~Tempone.
\newblock Implementation and analysis of an adaptive multilevel {M}onte {C}arlo
  algorithm.
\newblock {\em Monte Carlo Methods and Applications}, 20(1):1--41, 2014.

\bibitem{Hofman00}
N.~Hofmann, T.~M{\"u}ller-Gronbach, and K.~Ritter.
\newblock Optimal approximation of stochastic differential equations by
  adaptive step-size control.
\newblock {\em Math. Comp.}, 69(231):1017--1034, 2000.

\bibitem{Hunter07}
J.~D. Hunter.
\newblock Matplotlib: A 2d graphics environment.
\newblock {\em Computing in Science Engineering}, 9(3):90 --95, may-june 2007.

\bibitem{Silvana12}
S.~Ilie.
\newblock Variable time-stepping in the pathwise numerical solution of the
  chemical {L}angevin equation.
\newblock {\em The Journal of Chemical Physics}, 137(23):--, 2012.

\bibitem{Karatzas91}
I.~Karatzas and S.~E. Shreve.
\newblock {\em Brownian motion and stochastic calculus}, volume 113 of {\em
  Graduate Texts in Mathematics}.
\newblock Springer-Verlag, New York, second edition, 1991.

\bibitem{Kebaier05}
A.~Kebaier.
\newblock Statistical {R}omberg extrapolation: A new variance reduction method
  and applications to option pricing.
\newblock {\em Ann. Appl. Probab.}, 15(4):2681--2705, 2005.

\bibitem{Kloeden92}
P.~E. Kloeden and E.~Platen.
\newblock {\em Numerical solution of stochastic differential equations},
  volume~23 of {\em Applications of Mathematics (New York)}.
\newblock Springer-Verlag, Berlin, 1992.

\bibitem{Mattingly05}
H.~Lamba, J.~C. Mattingly, and A.~M. Stuart.
\newblock An adaptive {E}uler-{M}aruyama scheme for {SDE}s: convergence and
  stability.
\newblock {\em IMA J. Numer. Anal.}, 27(3):479--506, 2007.

\bibitem{Lecuyer05}
P.~L'Ecuyer and E.~Buist.
\newblock Simulation in {J}ava with {SSJ}.
\newblock In {\em Proceedings of the 37th conference on Winter simulation}, WSC
  '05, pages 611--620. Winter Simulation Conference, 2005.

\bibitem{Milstein03}
G.~N. Milstein and M.~V. Tretyakov.
\newblock Quasi-symplectic methods for {L}angevin-type equations.
\newblock {\em IMA J. Numer. Anal.}, 23(4):593--626, 2003.

\bibitem{Mishra12}
S.~Mishra and C.~Schwab.
\newblock Sparse tensor multi-level {M}onte {C}arlo finite volume methods for
  hyperbolic conservation laws with random initial data.
\newblock {\em Math. Comp.}, 81(280):1979--2018, 2012.

\bibitem{Oksendal98}
B.~{\O}ksendal.
\newblock {\em Stochastic differential equations}.
\newblock Universitext. Springer-Verlag, Berlin, fifth edition, 1998.

\bibitem{Platen06}
E.~Platen and D.~Heath.
\newblock {\em A benchmark approach to quantitative finance}.
\newblock Springer Finance. Springer-Verlag, Berlin, 2006.

\bibitem{Shreve04}
S.~E. Shreve.
\newblock {\em Stochastic calculus for finance. {II}}.
\newblock Springer Finance. Springer-Verlag, New York, 2004.
\newblock Continuous-time models.

\bibitem{Skeel02}
R.~D. Skeel and J.~A. Izaguirre.
\newblock An impulse integrator for {L}angevin dynamics.
\newblock {\em Molecular Physics}, 100(24):3885--3891, 2002.

\bibitem{Szepessy01}
A.~Szepessy, R.~Tempone, and G.~E. Zouraris.
\newblock Adaptive weak approximation of stochastic differential equations.
\newblock {\em Comm. Pure Appl. Math.}, 54(10):1169--1214, 2001.

\bibitem{Talay02}
D.~Talay.
\newblock Stochastic {H}amiltonian systems: exponential convergence to the
  invariant measure, and discretization by the implicit {E}uler scheme.
\newblock {\em Markov Process. Related Fields}, 8(2):163--198, 2002.
\newblock Inhomogeneous random systems (Cergy-Pontoise, 2001).

\bibitem{Talay90}
D.~Talay and L.~Tubaro.
\newblock Expansion of the global error for numerical schemes solving
  stochastic differential equations.
\newblock {\em Stochastic Analysis and Applications}, 8(4):483--509, 1990.

\bibitem{Teckentrup13}
A.~L. Teckentrup, R.~Scheichl, M.~B. Giles, and E.~Ullmann.
\newblock Further analysis of multilevel {M}onte {C}arlo methods for elliptic
  {PDE}s with random coefficients.
\newblock {\em Numer. Math.}, 125(3):569--600, 2013.

\bibitem{Yan02}
L.~Yan.
\newblock The {E}uler scheme with irregular coefficients.
\newblock {\em The Annals of Probability}, 30(3):1172--1194, 2002.

\end{thebibliography}
\bibliographystyle{plain}

\end{document}